\documentclass[12pt,reqno]{amsart}
\usepackage{amsbsy,amsfonts,amsmath,amssymb,amscd,amsthm}
\usepackage{mathrsfs,float}
\usepackage[mathcal]{euscript}
\usepackage{enumitem,graphicx,epstopdf,xcolor}
\usepackage[a4paper,text={6.1in,8in},centering,marginparwidth=23mm]{geometry}
\usepackage{pxfonts,epstopdf}
\usepackage[colorlinks=true,linkcolor=blue,citecolor=green]{hyperref}
\usepackage[normalem]{ulem}
\usepackage{soul}

\usepackage[margin=20pt,font=small,labelfont=bf,textfont=it]{caption}
\usepackage{subcaption}
\usepackage{comment}
\usepackage{mathtools}
\usepackage[all]{xy}
\usepackage[utf8]{inputenc}

\usepackage{ascmac}
\usepackage{here}

\usepackage[T1]{fontenc}
\usepackage{inputenc}
\usepackage[french,english]{babel}

\makeatletter
\newenvironment{abstracts}{%
  \ifx\maketitle\relax
    \ClassWarning{\@classname}{Abstract should precede
      \protect\maketitle\space in AMS document classes; reported}%
  \fi
  \global\setbox\abstractbox=\vtop \bgroup
    \normalfont\Small
    \list{}{\labelwidth\z@
      \leftmargin3pc \rightmargin\leftmargin
      \listparindent\normalparindent \itemindent\z@
      \parsep\z@ \@plus\p@
      
      \itemsep\medskipamount
    }%
}{%
  \endlist\egroup
  \ifx\@setabstract\relax \@setabstracta \fi
}

\newcommand{\abstractin}[1]{%
  \otherlanguage{#1}%
  \item[\hskip\labelsep\scshape\abstractname.]%
}
\makeatother

\newcommand{\eqdef}{\stackrel{\scriptscriptstyle\rm def}{=}}

\theoremstyle{plain}
\newtheorem{mainthm}{Theorem} 
\newtheorem{thm}{Theorem}[section]

\newtheorem{cor}[thm]{Corollary}
\newtheorem{lem}[thm]{Lemma}
\newtheorem{prop}[thm]{Proposition}
\newtheorem{claim}{Claim}[thm]

\theoremstyle{definition}
\newtheorem{dfn}[thm]{Definition}

\newtheorem{example}[thm]{Example}

\newtheorem{question}{Question}

\newtheorem{rem}[thm]{Remark}

\numberwithin{equation}{section}

\newcommand{\supp}{\operatorname*{supp}}

\newcommand{\lV}{\left\lVert}

\newcommand{\rV}{\right\rVert}

\setstcolor{red}

\small\normalsize
\let\oldmarginpar\marginpar
\renewcommand\marginpar[1]{\-\oldmarginpar[\raggedleft\tiny #1]%
{\raggedright\tiny #1}}

\newcommand{\relationarrow}[4]{\raise1ex\hbox{$\overset{\text{#1}}{\longrightarrow}$}\hspace{#3}\lower1ex\hbox{$\underset{\text{#2}}{\longleftarrow}$}\hspace{#4}\lower1ex\hbox{$/$}}

\let\oldtocsection=\tocsection
\let\oldtocsubsection=\tocsubsection
\renewcommand{\tocsection}[2]{\hspace{0em}\bf\oldtocsection{#1}{#2}}
\renewcommand{\tocsubsection}[2]{\hspace{1.8em}\oldtocsubsection{#1}{#2}}
\let\oldtocsubsubsection=\tocsubsubsection
\renewcommand{\tocsubsubsection}[2]{\hspace{4.2em}\oldtocsubsubsection{#1}{#2}}

\title[Finitude of physical measures]{Finitude of physical measures  for random maps}

\date{}

\author[Barrientos]{Pablo G.~Barrientos}
\address[Pablo G.~Barrientos]{Instituto de Matem\'atica e Estat\'{\i}stica, UFF Rua M\'ario Santos Braga s/n—Campus Valonguinhos, Niter\'{o}i, Brazil}
\email{pgbarrientos@id.uff.br}

\author[Nakamura]{Fumihiko Nakamura}
\address[Fumihiko Nakamura]{Faculty of Engineering, Kitami Institute of Technology, Hokkaido, 090-8507, JAPAN}
\email{nfumihiko@mail.kitami-it.ac.jp}

\author[Nakano]{Yushi Nakano}
\address[Yushi Nakano]{Graduate School of Science, Hokkaido University, Hokkaido, 060-0810, Japan}
\email{yushi.nakano@math.sci.hokudai.ac.jp}

\author[Toyokawa]{Hisayoshi Toyokawa}
\address[Hisayoshi Toyokawa]{Faculty of Engineering, Kitami Institute of Technology, Hokkaido, 090-8507, JAPAN}
\email{h\_toyokawa@mail.kitami-it.ac.jp.}

\subjclass[2020]{Primary 	37A30, 37C40, 37H05; Secondary 37A50, 37C30, 60J05}

\keywords{Palis conjecture;  physical measures; random maps; absolutely continuous ergodic stationary probability measures;  Markov operators; Markov processes}

\begin{document}

\begin{abstracts}
\abstractin{english}
For random compositions of independent and identically distributed measurable maps on a Polish space, we study the existence and finitude of absolutely continuous ergodic stationary probability measures (which are, in particular, physical measures) whose basins of attraction cover the whole space almost everywhere.
We characterize and hierarchize such random maps in terms of their associated Markov operators, as well as show the difference between classes in the hierarchy by plenty of examples, including additive noise, multiplicative noise, and iterated function systems. We also provide sufficient practical conditions for a random map to belong to these classes. For instance, we establish that any continuous random map on a compact Riemannian manifold with absolutely continuous transition probability has finitely many physical measures whose basins of attraction cover Lebesgue almost all the manifold. 
\end{abstracts}

\maketitle
\thispagestyle{empty}

\thispagestyle{empty}

\section{Introduction}

Uniform hyperbolicity was originally intended to encompass a residual, or at least a dense subset of all smooth dynamical systems~\cite{Anosov1967, Smale1967}, although it was soon realized that this is not true~\cite{AS1970,Newhouse1970}.
Uniformly hyperbolic systems are structurally stable~\cite{Anosov1967} and admit a very precise topological description of their behavior: there are finitely many compact transitive invariant subsets such that every forward orbit of the system accumulates on one of them~\cite{Smale1967}.
The dynamics near such attractors may be quite chaotic and thus essentially unpredictable after a long period of time.
However, Sinai, Ruelle and Bowen demonstrated that these attractors are the supports of physical measures and thus behave well from a statistical point of view~\cite{Bowen1975, Ruelle1976, Sinai1972}.
Kifer further showed that such systems are stochastically stable (i.e.~the physical measures continuously vary) under small random perturbations~\cite{Kifer1974}.
Building on this background, in a conference in honor of Douady in 1995 (refer to~\cite{Palis2000,Palis2005}), Palis developed a global picture recovering, in a more probabilistic formulation, much of the paradigm of uniform hyperbolic systems.
Namely, Palis conjectured that every smooth dynamical system can be approximated by systems having only finitely many attractors, which support physical measures that describe the time averages for Lebesgue almost all points and are stochastically stable under small random perturbations.

An important contribution to the stochastic part of the global Palis conjecture was provided by Brin--Kifer~\cite{BK1987} and Ara\'ujo~\cite{Araujo2000}.
Under some natural nondegenerate assumptions on noise, they found finitely many absolutely continuous ergodic stationary measures (in particular physical measures) with pairwise disjoint supports and whose statistical basins of attraction cover the whole ambient space almost everywhere.
From now on, we refer to this property as \emph{finitude of physical measures} or {\tt(FPM)} for short\footnote{All the measures considered in this paper will be probabilities.
Also, the context of this paper is the study of absolutely continuous ergodic stationary probabilities (which are, in particular, physical).
For that reason, the property is simply called ``finitude of physical measures''.}.
 See Definition~\ref{dfn:11} for a more precise description.
To be more specific on the nondegenerate assumptions, Brin and Kifer assumed that the transition probability has a continuous density, while Ara\'ujo assumed that the transition probability has a density whose support includes a ball whose diameter is uniformly bounded from below.
In the last two decades, their nondegenerate conditions appeared as the main assumption of many works on stochastic stability around the Palis conjecture (especially of systems without uniform hyperbolicity); see e.g.,~\cite{Araujo2001, APP2014, AT2005, BV1996, BV2006}.

In this paper, we will refine the Brin--Kifer and Ara\'ujo conditions by showing that {\tt (FPM)} follows merely if one assumes that the transition probability is absolutely continuous.
This is a significant improvement for applications.
For instance, random dynamical systems generated by additive noise (the most common noise in applications; see Remark~\ref{rm:0302} for details) are some examples that satisfy these conditions.
Our proof is quite different from~\cite{Araujo2000, BK1987}.
 It is based on Markov operators theory, which enables us to obtain
much stronger properties than {\tt (FPM)}, such as the exponential decay of the annealed correlation functions
 of each physical measure for some iterate of the random map.
 Moreover, we give a \emph{necessary and sufficient} condition for {\tt (FPM)} in terms of Markov operators, extending a previous work by Inoue and Ishitani~\cite{II1991} for Perron--Frobenius operators.
That is, we introduce the notion of \emph{mean constrictivity} for Markov operators and show its equivalence with both {\tt (FPM)} and the property of \emph{asymptotic periodicity in mean} introduced by Inoue and Ishitani.
This is also a generalization of the result by Lasota, Li and Yorke in~\cite{LLY1984} for the equivalence between constrictivity and asymptotic periodicity of Markov operators, which served as a stepping-stone for several later papers studying the existence of absolutely continuous invariant density for stochastic operators (see e.g.,~\cite{LM}).
This functional approach allows us to give a hierarchy of classes of Markov operators, which implies, together with plenty of examples indicating the difference between the classes, that {\tt (FPM)} are much weaker than the Brin--Kifer and Ara\'ujo conditions, refer to Figures~\ref{fig:hierarchy} and~\ref{fig:subhierarchy} (and Remark~\ref{rem:AraujoD}).
For instance, we include some random dynamical systems generated by finitely many continuous maps (so-called iterated function systems) and by multiplicative noise with a common fixed point (another important class of noise in applications).
These systems never have absolutely continuous transition~probabilities.

\subsection{Finitude of physical measures {\tt(FPM)}}\label{ss:1.1}
Let $X$ be a Polish space equipped with a probability measure $m$ on the Borel $\sigma$-field $\mathscr B$ of $X$.
Let $(T , \mathscr A, p)$ be a probability space and consider the product space $( \Omega , \mathscr F, \mathbb P)=(T^\mathbb{N}, \mathscr{A}^\mathbb{N},p^\mathbb{N})$.
In this paper, we will deal with a measurable map $f: T\times X \to X$ where we denote $f_t=f(t,\cdot)$ for $t\in T$ and consider the following nonautonomous iterations
\[
f^0_\omega=\mathrm{id} \quad \text{and} \quad f^n_\omega=
f_{\omega_{n}}\circ \dots \circ f_{\omega_1} \
\ \text{for}
\ n\in \mathbb N \ \text{and $\omega =(\omega _1, \omega _2, \ldots )\in \Omega$}.
\]
Since we consider the Bernoulli probability $\mathbb{P}=p^\mathbb{N}$ on $\Omega$, the sequence $\{ \omega =(\omega _1, \omega _2, \ldots ) \mapsto \omega_n\} _{n\geq 1}$ of noises at each step is an independent and identically distributed random process.
Thus, the sequence $\{f_\omega ^n(x_0)\} _{n\geq 0}$ can be   viewed   as a (time-homogeneous) discrete Markov chain $\{ X_n\} _{n\geq 0}$ on $( \Omega , \mathscr F, \mathbb P)$   valued in $X$   with initial distribution $X_0(\omega)=x_0$
and transition probability given by
\begin{equation}\label{eq:0302a}
P(x,A)= p(\{t\in T: \, f_t(x)\in A\})=\mathbb{P}(\{\omega\in \Omega: f_\omega(x)\in A\}) \quad \text{for $x\in X$, \ $A\in\mathscr{B}$}.
\end{equation}
Recall that a nonnegative function $Q(x,A)$ defined for $x \in X$ and $A \in \mathscr{B}$ is called a \emph{(Markov) transition probability} if
\begin{enumerate}[label=(\roman*)]
\item \label{eq:Markov1} $Q(x,\cdot)$ is a probability measure for every fixed $x\in X$,
\item \label{eq:Markov2} $Q(\cdot, A)$ is a $\mathscr{B}$-measurable function for every fixed $A\in \mathscr{B}$.
\end{enumerate}
We shall also use the $n$-th transition probability $P^n(x,A)$ of the process $\{ X_n\} _{n\geq 0}$~given~by
\begin{align*}
P^n(x,A)& =
\mathbb{P}(\{\omega\in \Omega: f^n_\omega(x)\in A\})
\end{align*}
which is also a Markov transition probability.
We refer to~\cite{Arnoldbook} and~\cite{MT2012} for general theories of random dynamical systems and Markov processes, respectively.
A probability measure $\mu$ on $X$ is said to be a \emph{stationary measure} of $f$ if
\begin{equation} \label{def:stationary}
 \mu(A) = \int (f_t)^{}_*\mu(A)\, dp \quad \text{for every $A \in \mathscr{B}$}.
\end{equation}
Here, $g^{}_*\nu$ denotes the pushforward of a measure $\nu$ by a measurable map $g$, that is, $g^{}_*\nu (A)=\nu( g^{-1}A)$ for $A\in \mathscr B$.
It is well known that $\mu$ is a stationary measure of $f$ if and only if $\mathbb{P}\times \mu$ is an invariant measure for the skew-product
\begin{equation}\label{eq:skew}
F: \Omega\times X \to \Omega\times X, \quad F(\omega,x)=(\sigma\omega,f_\omega(x))
\end{equation}
where $\sigma$ is the shift operator defined on $\Omega$ (see~\cite{O83}).
Moreover, we say that $\mu$ is ergodic if $\mathbb{P}\times \mu$ is an ergodic probability measure of $F$.
This is equivalent to asking that any $A\in\mathscr{B}$ such that $P(x,A)\geq 1_A(x)$ where $1_A$ is the indicator function has $\mu$-measure 0 or 1 (see~\cite[Appendix A.1]{Kifer1986} and Appendix~\ref{appendix:B}).
Finally, recall that a $\nu$-null set is a measurable set that has $\nu$-measure zero, and a measure $\nu$ is said to be absolutely continuous with respect to a measure $\lambda$ if any $\lambda$-null set is $\nu$-null.

Our goal is to find a condition on $f$ under which {\tt (FPM)} holds.
\begin{dfn}\label{dfn:11}
We say that $f$ satisfies {\tt (FPM)} if there exist finitely many ergodic stationary probability measures $\mu _1,\dots, \mu_r$ of $f$ such that
\begin{enumerate}[topsep=-0.1cm, label=\arabic*)]
 \item they are absolutely continuous with respect to $m$;
 \item they have pairwise disjoint supports (up to an $m$-null set);
 \item for any $\mathscr{B}$-measurable bounded real-valued function $\psi$,
  \begin{equation*}\label{eq:10100a}
m\left(B_\omega(\mu _1, \psi)\cup \dots \cup B_\omega(\mu_r, \psi)\right)=1 \quad \text{for $\mathbb P$-almost every $ \omega \in \Omega$}
\end{equation*}
where \\[-0.5cm]
\begin{equation*}\label{eq:10100b}
\qquad B_\omega(\mu _i, \psi)=\left\{ x\in X : \, \lim _{n\to \infty} \frac{1}{n} \sum _{j=0}^{n-1} \psi(f_{ \omega}^{j}(x)) = \int \psi \, d\mu _i
\right\} \quad \text{for $i=1,\dots,r$}.
\end{equation*}
\end{enumerate}
\end{dfn}

\begin{rem} \label{rem:fiberwise}
Since $X$ is a Polish space, we have a metric $d$ compatible with the topology of $X$ such that $(X,d)$ is a complete separable metric space.
In such a case, the convergence in the weak* topology on the space of probability measures on $X$ is countably determined by a set $S$ of bounded Lipschitz (with respect to the metric $d$) functions.
See Proposition~\ref{prop:countable-determined-weak}.
Then, the \emph{fiberwise statistical basin of attraction} $B_\omega(\mu_i)$ of $\mu_i$ can be written as
\[
B_{\omega}(\mu_i) \coloneqq \left\{x\in X : \lim_{n\to\infty}\frac{1}{n}\sum_{j=0}^{n-1}\delta_{f_\omega^j (x)}=\mu_i \right\} = \bigcap_{\psi\in S} B_\omega(\mu_i,\psi) \quad \text{for $i=1,\dots,r$}
\]
where the limits of measures are taken in the weak* topology.
Thus, item 3) in Definition~\ref{dfn:11} implies that the union of these fiberwise statistical basins of attraction $\mathbb{P}$-almost surely covers $X$ up to a set of null $m$-measure.
That is,
\begin{equation}\label{eq:3}
 m(B_{\omega}(\mu_1)\cup\cdots\cup B_{\omega}(\mu_r)) = 1 \quad\text{for $\mathbb{P}$-almost every $\omega\in\Omega$}.
\end{equation}
On the other hand, in Proposition~\ref{prop:IFS-circle} we show an example where~\eqref{eq:3} does not imply 3) even under the assumption that  1) and 2)  hold.
See also Remark~\ref{rem:fpm}.
\end{rem}

\begin{rem} It is not difficult to show (see Lemma~\ref{lem-mu1}) that, for the measure $\mu_i$ above, $B_\omega(\mu_i)$ has full $\mu_i$-measure for $\mathbb{P}$-almost all $\omega\in \Omega$ and $i=1,\dots,r$.
Moreover, since $\mu_i$ is absolutely continuous with respect to $m$, it holds that $m(B_\omega(\mu _i))>0$ for $\mathbb{P}$-almost every $\omega\in \Omega$.
Therefore, $\mu_i$ is a \emph{physical measure} with respect to the reference measure~$m$ in the sense that $\mathbb{P}$-almost surely the fiberwise statistical basin of attraction $B_\omega(\mu_i)$ has a positive $m$-measure for each $i=1, \ldots ,r$.
See more details on the notion of physical measures for random maps in  {Remark~\ref{rem:physical2}}.
\end{rem}

In the deterministic case, that is, when $f_t=g$ for all $t\in T$ with some $g:X\to X$, {\tt(FPM)} means that the dynamics of $g$ can be statistically understood by finitely many ergodic invariant physical probability measures that describe the time average of almost all points in $X$.
When and where this property holds attracted great interest in dynamical systems theory, and we refer e.g.~to~\cite{BDV2006, Palis2000}.
Notice that there are some obstacles to the finitude of physical ergodic invariant probability measures, such as the so-called Newhouse domains~\cite{Colli1998,GST1993,Leal2008,Newhouse1979,PV1994} in which generically infinitely many attractors coexist.
Moreover, recently Berger has realized that the coexistence of infinitely many attractors is locally Kolmogorov typical in parametric families of endomorphisms of surfaces and diffeomorphisms of higher dimensional manifolds, see~\cite{BR2021b,BR2021a,barrientos2023typical,Berger2016,Berger2017}.
In contrast to the Palis conjecture, the so-called Takens' last problem asked if it is possible to construct a persistent class of dynamics of a compact manifold where time averages do not exist (named as historic behavior; cf.~\cite{Ruelle2001,Takens2008}) on a Lebesgue positive measure set.
 Some important advances in this question have been obtained in~\cite{KS2017} where the authors constructed a locally dense class of $C^r$-surface diffeomorphisms ($r\geq 2$) with historic behavior on a positive Lebesgue measure set.
 This result has been extended to $C^\infty$ and real analytic surface diffeomorphisms in~\cite{BB2020} and to $C^r$-diffeomorphisms with $r\geq 1$ in dimension three and higher in~\cite{Barrientos2021} (see also~\cite{KNS2021} for a specific three-dimensional example).

Ara\'ujo~\cite{Araujo2000} gave a quite useful sufficient condition {(see~(1) below)} for {\tt(FPM)} for the following class of random maps.
First of all, here $X$ is a compact Riemannian manifold
 and $T$ is the unit ball of a Euclidean space with $m$ and $p$ being the normalized Lebesgue measures on $X$ and $T$, respectively.
 Now,
\begin{enumerate}[label=(\arabic*),leftmargin=1cm]
 \item $f: T \times X\to X$ is a continuous map and $f_{t }$ is a $C^1$-diffeomorphism for every~$t\in T$;
there are $n_0\in \mathbb N$ and a positive number $\xi_0$ such that for every $n\geq n_0$ and $x\in X$
\begin{enumerate}[label=(\Alph*),leftmargin=1cm]
\item $\left\{ f^n_\omega (x) : \omega \in \Omega \right\}$ contains the ball of radius $\xi_0$ and centered at $f^n_0(x)$;
\item $P^n(x,\cdot)$ is absolutely continuous with respect to $m$.
\end{enumerate}
\end{enumerate}
Here,
we wrote $f^n_0$ for the usual $n$-th iteration of the single map $f_0 : X\to X$.
Condition~(A) above is a topological requirement.
On the other hand, since the $n$-th and $(n-1)$-th transition probabilities are related by the recurrence
\begin{equation} \label{eq:recurence}
P^{n}(x,A)=\int P^{n-1}(y,A) P(x,dy)
\end{equation}
we have condition (B) by only requiring that $P^{n_0}(x,\cdot)$ be absolutely continuous with respect to $m$ for all $x\in X$.
A greater requirement is the continuity of $f$ in item~(1).
For instance, such a requirement is not necessary for the approach by Brin and Kifer~\cite[Section 2]{BK1987} to get {\tt(FPM)}.
These authors dealt with abstract Markov chains $\{ X_n\}_{n\geq 0}$ on a compact Riemannian manifold $X$ with transition probability $P(x,A)$ having a continuous density in the following sense:
\begin{enumerate}[resume] \label{Brin--Kifer}
 \item there are an integer $n_0\geq 1$ and a nonnegative function $p(x, y)$ that is continuous in both variables
such that for any $x\in X$ and  $A \in \mathscr{B}$,
\[
P^{n_0}(x,A)=\int_{A} p(x,y) \, dm(y).
\]
 \end{enumerate}
Thus, $P^{n_0}(x,\cdot)$ is absolutely continuous with respect to the normalized Lebesgue measure~$m$ \emph{for all} $x\in X$ as in the Ara\'ujo's condition (B).
Although this abstract approach seems more general, this is not the case because any Markov chain in a Polish space can be represented by a random map as in~\eqref{eq:0302a} (cf.~\cite{JKR2015, Kifer1986}).

We will prove that condition~(1) and condition~(2), respectively, imply that there are $n_0\in\mathbb N$ and a nonnegative function $p(x,y)$ such that
$P^{n_0}(x,dy)=p(x,y)\,dm(y)$ and $x \mapsto p(x,\cdot)$ is a continuous map from $X$ to $L^1(m)$.
See Remarks~\ref{rem:gAraujo} and~\ref{rem:gBrin--Kifer}.
Here $L^1(m)=L^1(X,\mathscr{B},m)$ denotes, as usual,
the Banach space of all real-valued $\mathscr{B}$-measurable functions $\varphi$ on~$X$ whose $L^1$-norm $\|\varphi \|\coloneqq\int _X\vert \varphi \vert dm$ is bounded,
where two functions that coincide with each other $m$-almost everywhere are identified.
In view of this, the following theorem is a notable improvement of Brin--Kifer's and Ara\'ujo's sufficient conditions to get {\tt(FPM)}:

\begin{mainthm} \label{thm:A}
Let $(X , \mathscr B, m)$ and $(T, \mathscr A, p)$ be a compact Polish probability space and a probabilistic space, respectively.
Consider a measurable map $f:T\times X\to X$ and let $P^n(x,A)$ be the $n$-th transition probabilities for the associated Markov chain induced by $f$. Assume that for some $n_0\in \mathbb{N}$,
\begin{equation*}
P^{n_0}(x,dy)=p(x,y)\,dm(y) \quad \text{such that} \quad x \in X \mapsto p(x,\cdot)\in L^1(m) \ \ \text{is continuous}.
\end{equation*}
Then, $f$ satisfies {\tt(FPM)}.
\end{mainthm}

In the following series of remarks, we will show some easily checkable assumptions from which Theorem~\ref{thm:A} is applicable, allowing us to compare with the previous sufficient conditions from the literature.

We say that $f: T\times X\to X$ is a \emph{continuous random map} if $f$ is measurable and $f_t=f(t,\cdot):X\to X$ is continuous for $p$-almost every $t\in T$.

\begin{rem} \label{rem:gAraujo}
Firstly, {\tt(FPM)} follows from the assumption:
 \begin{enumerate}[label=(\roman*), leftmargin=1cm]
 \item \label{item:araujo} {\it Let $(X , \mathscr B, m)$ and $(T, \mathscr A, p)$ be a compact Polish probability space and a probabilistic space, respectively, such that for some $n_0\in \mathbb N$, it holds that
 \begin{enumerate}
 \item \label{item:Feller0} $f:T\times X\to X$ is a continuous random map,
 \item $P^{n_0}(x,\cdot)$ is absolutely continuous with respect to $m$ for all $x\in X$.
 \end{enumerate}}
\end{enumerate}
This generalizes Ara\'ujo's result in~\cite{Araujo2000}. The fact that~(i) implies {\tt(FPM)} follows directly from Corollary~\ref{cor:rem} and Theorem~\ref{thm:A}. 
In Proposition~\ref{prop:multi:ex1}, we prove that, in the context of (i), neither condition (a) nor condition (b) is sufficient to obtain~{\tt(FPM)}, and we also show that these conditions are not necessary for~{\tt(FPM)}. 
Furthermore, (i)~extends Ara\'ujo's result by removing condition~(A). In Theorem~\ref{AAraujo2000-generalization} (1), we provide another sufficient condition for {\tt(FPM)} that incorporates condition~(A).
\end{rem}

\begin{rem} \label{rem:gBrin--Kifer}
Secondly, {\tt(FPM)} holds under the following assumption:
\begin{enumerate}[label=(\roman*), leftmargin=1cm] \stepcounter{enumi}
 \item \label{item:L1-cont} \emph{Let $(X,\mathscr{B},m)$ be a compact probability Polish space and
$P^{n_0}(x,dy)=p(x,y)\, dm(y)$ for some $n_0\in\mathbb N$, where $p(\cdot,y)$ is a continuous function on $X$ for $m$-almost every $y\in X$.}
 \end{enumerate}
This clearly generalizes Brin--Kifer's result in~\cite{BK1987}.
The proof of this observation follows from the Schaff\'e--Riez theorem, cf.~\cite{kusolitsch2010theorem}. Indeed, if $\{x_n\}_{n\in\mathbb N}$ is a converging sequence to $x$, then by the continuity of $p(\cdot,y)$ we get $p(x_n,y) \to p(x,y)$ and thus the Schaff\'e--Riez theorem implies that $\|p(x_n,\cdot)-p(x,\cdot)\| \to 0$ as $n\to\infty$. This means that $p(x,\cdot)$ varies continuously on $L^1(m)$ with respect to $x$. Now, Theorem~\ref{thm:A} implies~{\tt(FPM)}.
See also
Theorem~\ref{AAraujo2000-generalization} (2) for another weakening of Brin--Kifer's condition.
\end{rem}

\begin{rem}\label{rm:0302}
One of the most important noises in real applications that satisfy the assumption in Theorem~\ref{thm:A} is the random dynamical system generated by
 \emph{additive noise with absolutely continuous distribution}.
 Here, $X=T$ is a {compact} Lie group where the algebraic multiplication is denoted by "$+$", $m$ is the Haar measure, and $p$ is an absolutely continuous probability with respect to $m$.
 The torus $\mathbb R^d/\mathbb Z^d$ is often considered in the literature as an example.
 The random map $f$ is defined by $f_t(x) =f_0(x) +t $ for some continuous map $f_0:X\to X$.
 As a slight abuse of notation, we denote the density function of $p$ by $p(x)$.
Then, $P(x,A)$ can be written as
\[
P(x,A) = \int _A p(y-f_0(x))\, dm(y) \quad \text{for $x\in X$ and $A\in\mathscr B$}
\]
(cf.~\cite[Equation (10.5.5)]{LM}).
Hence, $P(x,\cdot )$ is absolutely continuous with respect to $m$ for all $x\in X$.
Notice that if the support of $p(x)$ does not include any open ball centered at $0$ or $p(x)$ is not continuous, then Ara\'ujo's condition (A) and Brin--Kifer's condition (2) are violated in general.
Furthermore, it seems difficult to have the condition in Theorem~\ref{thm:A} with $n_0=1$ because $x\in X\mapsto p(y-f_0(x))\in \mathbb R$ is not continuous in general.
However, surprisingly, by virtue of Remark~\ref{rem:gAraujo} we know that the condition in Theorem~\ref{thm:A} is always satisfied (and thus {\tt(FPM)} holds).
In fact, by Corollary~\ref{cor:rem}, the condition in Theorem~\ref{thm:A} holds with $n_0=2$.
\end{rem}

\begin{rem} \label{rem:multiplicative-noise}
Another important class of noise that appears in real applications is \emph{multiplicative noise with absolutely continuous distribution} (cf.~\cite{Athreya2003, Sumi2021}).
For instance, consider the case when $X=T=[0,1]$, $m$ is the Lebesgue measure and $p$ is an absolutely continuous probability with respect to $m$ with density function $p(x)$. Define $f_t(x) = t g(x) $, where $g$ is some continuous map on $X$.
It is straightforward to see that $P(x,A)$ is of the form
\[
P(x,A) = \int _A p\left(\frac{y}{g(x)}\right)dm(y) \quad \text{for $x\in X$ and $A\in\mathscr B$}
\]
(cf.~\cite[Equation (10.7.5)]{LM}).
Therefore, if $g$ is bounded away from zero, then $P(x,\cdot )$ is absolutely continuous with respect to $m$, and by Remark~\ref{rem:gAraujo}, {\tt(FPM)} holds. Actually, the condition in Theorem~\ref{thm:A} holds (with $n_0=2$).
In fact, this (together with Theorem~\ref{thm:B}) generalizes~\cite[Theorem 10.7.1]{LM}.
On the other hand, in contrast to additive noise, if $g(x)=0$ for some $x\in X$, then the absolute continuity of $P(x,\cdot )$, the condition in Theorem~\ref{thm:A} and {\tt(FPM)} can fail to hold.
See Section~\ref{subsec:multiple} for details.
\end{rem}

\begin{rem}\label{rmk:IFS}
Finally, we remark that the important class of random dynamical systems generated by iterated function systems with probabilities (see Section~\ref{ss:e} for its formal definition) does not meet the assumption in Theorem~\ref{thm:A}.
Indeed, in this case, we have that $T$ is a finite set $\{1,\dots,k\}$ and thus
\[
P(x,\cdot)= \sum_{i=1}^k p_i \delta_{f_i(x)} \quad \text{where $p_i =p(\{ i\})>0$}
\]
which cannot be absolutely continuous with respect to $m$ in general.
However, some of them will be in the range of Section~\ref{s:MO} to get {\tt(FPM)}, in which several weaker versions of the condition in Theorem~\ref{thm:A} are given.
\end{rem}

The proof of Theorem~\ref{thm:A} is based on the analysis of the annealed Perron--Frobenius operator associated with the random dynamical system generated by $f$.
We will show that this operator belongs to the class of constrictive Markov operators, which has been extensively studied in the literature.
This observation allows us to generalize Theorem~\ref{thm:A} in terms of Markov operators.
In particular, we can obtain new practical sufficient conditions implying {\tt(FPM)}, as indicated in Theorem~\ref{AAraujo2000-generalization}.
As we will explain in Remark~\ref{rmk:6.7}, such conditions generalize the sufficient condition studied in the work by Ara\'ujo and Ayta\c{c}~\cite{AA2017} to get uniform ergodicity (a stronger property than {\tt(FPM)}).

\subsection{Markov operators}\label{s:MO}
In order to provide a general definition of Markov operator, assume that $(X,\mathscr{B},m)$ is any abstract probability space (not necessarily a Polish space as in the previous subsection). Let $D(m)=D(X,\mathscr{B},m)$ be the space of density functions, that is,
\[
 D(m) =\left\{h \in L^1(m) : h \geq 0 \ \text{$ m$-almost everywhere and $\Vert h \Vert =1$} \right\}.
 \]
An operator $P: L^1(m)\to L^1 (m)$ is called a \emph{Markov operator} if $P$ is linear, positive (i.e.~$P\varphi \geq 0$ $m$-almost everywhere if $\varphi \geq 0$ $m$-almost everywhere)
and
 \begin{equation}\label{eq:0219}
\int P\varphi \, dm = \int \varphi \, dm \quad
\text{ for every $\varphi \in L^1(m)$}.
\end{equation}
Note that a  positive linear operator $P$ on $L^1(m)$ is a Markov operator\footnote{Any Markov operator $P$ is a bounded operator:
Given $\varphi \in L^1(m)$, consider
$\varphi _+\coloneqq \max\{\varphi ,0\}$ and  $\varphi _-\coloneqq \max\{ -\varphi ,0\}$.
Then, $P\varphi _+, P\varphi _- \geq 0$,
 and thus
$\Vert P\varphi \Vert \leq \int (P\varphi _+ + P\varphi _- )\, dm = \int \varphi _+ dm + \int \varphi _-dm = \Vert \varphi \Vert$.
} if and only if $P(D(m)) \subset D(m)$.
It is not difficult to see that this is equivalent to $P^*1_X = 1_X$, where $P^*$ is the adjoint operator of $P$, that is, $P^*$ is a bounded linear operator on $L^\infty (m)\cong (L^1(m))^*$ given by $$\int P^*\psi \cdot \varphi \, dm = \int \psi \cdot P\varphi \, dm \quad \text{for $\psi\in L^\infty (m)$ and $\varphi\in L^1(m)$.}$$
Recall that $L^\infty(m)=L^\infty(X,\mathscr{B},m)$ is the Banach space of $\mathscr{B}$-measurable, $m$-essentially bounded, real-valued functions defined on $X$ where, as usual, two functions that coincide with each other $m$-almost everywhere are identified.

A key property of Markov operators for the purpose of this paper is \emph{constrictivity}.
\begin{dfn}\label{dfn:1011}
A sequence $(Q_n)_{n\geq 1}$ of Markov operators on $L^1(m)$ is called
 \begin{enumerate}
 \item {\it constrictive} if
 there exists a compact set $F$ of $L^1(m)$ such that for any $h \in D(m)$,
\[
  \lim_{n\to \infty} d(Q_n h,F)=0;
\]
 \item {\it uniformly constrictive} if
 there exists a compact set $F$ of $L^1(m)$ such that
\[
\lim _{n\to \infty} \sup _{h \in D(m)} d(Q_n h, F)=0.
\]
\end{enumerate}
Here $d(\varphi,F)=\inf_{\psi\in F}\|\varphi -\psi \|$.
In particular, a Markov operator $P$ on $L^1(m)$ is called { 
\begin{enumerate}
 \item[\tt(C)] {\it constrictive} if the sequence $(P^n)_{n\geq 1}$ is constrictive.
 \item[\tt(UC)] {\it uniformly constrictive} if the sequence $(P^n)_{n\geq 1}$ is uniformly constrictive.
 \item[\tt(MC)] {\it mean constrictive} if the sequence
 $(A_n)_{n\geq 1}$
 is constrictive, where $A_n$ is given by
 \[
A_n\varphi = \frac{1}{n}\sum_{i=0}^{n-1}P^i\varphi \quad \text{for $\varphi \in L^1(m)$}.
\]
 \end{enumerate}}
\end{dfn}
The compact set $F$ above is called a \emph{constrictor}.
These conditions appeared in the context of mean ergodic theorems, see~\cite{Emelyanov, LM}.
Notice that, by definition, we have
\[
\text{ \tt (UC) \ \ $\Rightarrow$ \ \ (C) \ $\Rightarrow$ \ \ (MC).}
\]
See also Figure~\ref{fig:hierarchy} for a global picture.
Furthermore, the uniform constrictivity of $P$ is known to be equivalent to the quasi-compactness of $P$ in $L^1(m)$, cf.~\cite[Theorem 2]{Bart95}.
Recall that an operator $Q$ is called \emph{quasi-compact} if there is a compact linear operator $R$ such that $\|Q^n-R\|_{\rm op}<1$ for some $n\in \mathbb{N}$.
Hence, the class of quasi-compact operators contains the subclass of \emph{eventually compact operators}, that is, the linear operators $Q$ such that $Q^n$ is compact for some $n\in\mathbb{N}$.

\subsubsection{Perron--Frobenius operators}
One of the most important examples of Markov operators is the \emph{Perron--Frobenius operator} induced by a nonsingular transformation $g:X\to X$.
Recall that $g$ is nonsingular (with respect to the reference measure $m$ on~$X$) if the preimage of any $m$-null set by $g$ is $m$-null.
The Perron--Frobenius operator $\mathcal{L}_{g}: L^1(m) \to L^1(m)$ of $g$ is defined by the formula
\[
 g_*m_\varphi(A) = \int_A \mathcal{L}_g\varphi \, dm \quad
\text{ for all $\varphi \in L^1(m)$ and $A\in \mathscr{B}$}
\]
where $m_\varphi$ is the finite signed measure given by
\begin{equation}\label{eq:mphi}
 m_\varphi(A) =\int _A \varphi \, dm \quad \text{for $A\in \mathscr{B}$.}
\end{equation}
It is easy to see that $\mathcal{L}_g$ is a Markov operator and $m_{\mathcal{L}_g\varphi}=g_*m_\varphi$.
As in this example, Markov operators $P$ naturally appear in the study of (random) dynamical systems, and $( P^n\varphi ) _{n\geq 0}$ is interpreted as the evolution of density functions driven by the system.
We refer to~\cite{baladi2000positive, baladi2018dynamical,Emelyanov,foguel2007ergodic,LM}.

Let $f: T \times X\to X$ be as in Section~\ref{ss:1.1}.
 We simply write $\mathcal{L}_{t}$ for the Perron--Frobenius operator of $f_{t}=f(t,\cdot)$ for $t\in T$ and define the \emph{annealed Perron--Frobenius operator} ${\mathcal{L}}_f: L^1(m) \to L^1(m)$ by
\[
 {\mathcal{L}}_f \varphi (x) =\int \mathcal{L}_{t} \varphi(x) \, dp.
\]
Then, it is straightforward to see that ${\mathcal{L}}_f$ is also a Markov operator.
Now we are in a position to state two of our main results.

\begin{mainthm} \label{thm:B} Under the assumption of Theorem~\ref{thm:A}, $\mathcal{L}_f$ is eventually compact, in particular, uniformly constrictive.
\end{mainthm}

We will obtain the above result by proving an equivalent but apparently more general theorem stated in terms of Markov operators and Markov processes.
See Theorem~\ref{thm:B-Markov} together with Proposition~\ref{prop:ultra-Feller+B} and Remark~\ref{rem:P=Lf}.

In the setting of Theorem~\ref{thm:B}, since $\mathcal{ L}_f$ is quasi-compact, we can obtain exponential decay of annealed correlation functions, in the sense that there are finitely many absolutely continuous ergodic stationary measures $\mu_1, \ldots ,\mu_r$ with pairwise disjoint supports and constants $k\in \mathbb N$, $C>0$, $\kappa \in (0,1)$ such that for any $\varphi , \psi\in L^\infty(m)$, $n\in \mathbb N$ and $ j=1,\ldots ,r$, we have
\begin{equation}\label{eq:expmixing}
 \left\vert \int \psi \circ f^{kn}_\omega \cdot \varphi \, d(\mathbb P\times \mu _j) - \int \psi\, d\mu_j \int \varphi\, d\mu _j\right\vert \leq C\kappa ^{kn}\Vert \varphi \psi\Vert _{\infty}.
\end{equation}
Refer to, for example,~\cite{Buzzi1999,Liverani1995} for background and~\cite{BG2009} for a proof\footnote{Although the paper~\cite{BG2009} only dealt with deterministic maps, one can show the claim by literally repeating the argument in~\cite[Appendix B]{BG2009} with $\mathcal L_f$ instead of a Perron--Frobenius operator $\mathcal L_g$ of a nonsingular map $g:X\to X$, after realizing the duality $ \int \psi \circ f^{n}_\omega \cdot \varphi \, d(\mathbb P\times m) = \int \psi \cdot \mathcal L_f^n \varphi \, dm$ corresponding to $ \int \psi \circ g^{n} \cdot \varphi \, dm = \int \psi \cdot \mathcal L_g^n \varphi \, dm$.}.
Actually, this claim with $r=k=1$ was shown by Ara\'ujo and Ayta\c c in~\cite{AA2017} under a bit stronger assumption than Ara\'ujo's conditions~(A) and~(B) in a different manner (using a purely probability-theoretic technique); see also Remark~\ref{rmk:6.7}.
Furthermore, the quasi-compactness of $\mathcal L_f$ may lead to other several limit theorems (such as central limit theorem, large deviation principle, local limit theorem and almost sure invariance principle) via the so-called Nagaev--Guivarc'h perturbative spectral method, refer to~\cite{ANV2015,Gouezel2015,HH2001}.

\begin{mainthm} \label{thm:C}
Let $(X , \mathscr B, m)$ and $(\Omega , \mathscr F, \mathbb{P})=(T^\mathbb{N}, \mathscr A^\mathbb{N}, p^\mathbb{N})$ be a Polish probability space and the infinite product space of a probability space $(T, \mathscr A, p)$, respectively.
Consider a measurable map $f:T\times X\to X$.
Then, the followings are equivalent:
\begin{enumerate}[label=(\roman*), leftmargin=1cm]
\item
${\mathcal{L}}_f$ is mean constrictive;
\item
$f$ satisfies {\tt(FPM)}.
\end{enumerate}
Moreover, if $f$ satisfies any of the above equivalent conditions and $\mu_1,\dots,\mu_r$ denote the measures that appear in the property {\tt(FPM)}, then it holds
 \begin{enumerate}[ leftmargin=1cm]
\item
$(\mathbb{P}\times m)\left(B(\mu _1)\cup \dots \cup B(\mu_r)\right)=1$;
\item
$\mathbb{P}\left( B_x(\mu_1) \cup \dots \cup
B_x(\mu_r)\right) = 1$ for $m$-almost every~$x \in X $;
\item
$m\left(B_\omega(\mu _1)\cup \dots \cup B_\omega(\mu_r)\right)=1$ for $\mathbb P$-almost every $\omega \in \Omega$;
\end{enumerate}
where for each $i=1,\dots,r$,
\begin{equation*}
\begin{gathered}
B(\mu_i)=\left\{ (\omega,x)\in \Omega\times X : \, \lim _{n\to \infty} \frac{1}{n} \sum _{j=0}^{n-1} \delta_{f^j_\omega (x)} = \mu _i \right\} \\
B_\omega(\mu_i)=\left\{x\in X: \, (\omega,x)\in B(\mu_i)\right\} \ \ \text{and} \ \ B_x(\mu_i)= \left\{\omega\in \Omega: \, (\omega,x)\in B(\mu_i)\right\}.
\end{gathered}
\end{equation*}
\end{mainthm}

We note that since uniform constrictivity implies mean constrictivity, Theorem~\ref{thm:A} is just a consequence of Theorems~\ref{thm:B} and~\ref{thm:C}.
In fact, conclusions (1)--(3) in Theorem~\ref{thm:C} also hold under the assumptions of Theorem~\ref{thm:A}.

Notice that $B(\mu_i)$ is basically the statistical basin of attraction of the measure $\mathbb{P}\times \mu_i$ for the skew-product map $F$ in~\eqref{eq:skew}.
Thus, the above theorem is actually a characterization of an absolutely continuous version of the Palis conjecture for the class of deterministic systems of the form~\eqref{eq:skew}, in terms of the annealed Perron--Frobenius operator of the random map $f$.
In particular, when $T$ (or $\Omega$) is a singleton, Theorem~\ref{thm:C} provides a characterization of the existence of finitely many $m$-absolutely continuous invariant probability measures for deterministic dynamics such that the union of their basins has full $m$-measure, in terms of the Perron--Frobenius operator.

\subsubsection{General Markov operators}
We next generalize the equivalence in Theorem~\ref{thm:C} to general Markov operators.
However, to do this, we need some preliminaries.

Let $P$ be a Markov operator on $L^1(m)$ and consider the adjoint operator $P^*$ on $L^\infty(m)$.
 We define the support of a real-valued function $h$ (up to an $m$-null set) by $\supp h = \{ x\in X : \, h(x) \not= 0\}$.
We say that $h$ is an \emph{invariant density} of $P$ (or a \emph{$P$-invariant density}) if $h \in D(m)$ and $Ph =h $.
As usual, $P^n$ and $(P^*)^n$ denote the $n$-th iterated of $P$ and $P^*$ respectively. Let $(P^n)^*$ be the adjoint operator of $P^n$.
Using recursively the duality relation, it is not difficult to see that $(P^n)^*=(P^*)^n$. For simplicity of notation, we will simply write $P^{n*}$ when no confusion can arise. Having this notation in mind, we say that a $P$-invariant density $h$ has the \emph{maximal support} if
\begin{align}\label{max.supp}
\lim_{n\to\infty}P^{n*}1_{\supp h}(x)=1 \quad \text{for $m$-almost all $x\in X$.}
\end{align}

A probability measure $\mu$ is called \emph{ergodic} if any $A\in\mathscr{B}$ such that $P^*1_A\geq 1_A$ satisfies $\mu(A)\in \{0,1\}$.
We will say that a $P$-invariant density $h \in D(m)$ is \emph{ergodic} if the probability measure $m_h$ given in~\eqref{eq:mphi} is ergodic. See Appendix~\ref{appendix:B} for more details on equivalent definitions.

\begin{mainthm}\label{thm:D}
{Let $(X,\mathscr{B},m)$ be an abstract probability space and consider a Markov operator $P: L^1(m) \to L^1(m)$. Then, the following conditions are equivalent:}
\begin{enumerate}
\item[\tt (MC)]
$P$ is mean constrictive;
\item[\tt (FED)]
$P$ admits finitely many ergodic $P$-invariant densities $h_1,\dots,h_r$ with mutually disjoint supports (up to an $m$-null set) and the invariant density function $h=\frac{1}{r}(h_1+\dots+h_r)$ has the maximal support;
\item[\tt (APM)]
There exist finitely many ergodic $P$-invariant densities $h_1,\dots,h_r$ with mutually disjoint supports (up to an $m$-null set) and positive bounded linear functionals $\lambda_1, \dots, \lambda_r$ on $L^1(m)$ such that
\begin{align*}
\lim_{n\to\infty} \bigg\| A_n\varphi-\sum_{i=1}^r\lambda_i(\varphi)h_i\bigg\| =0 \quad \text{for any $\varphi\in L^1(m)$}.
\end{align*}
\end{enumerate}
\end{mainthm}

The condition {\tt(APM)}, named as \emph{asymptotic periodicity in mean}, was introduced by Inoue and Ishitani for Perron--Frobenius operators in~\cite{II1991} as a weaker version of the classic property \emph{asymptotic periodicity}. In turn, asymptotic periodicity was introduced and shown to be equivalent to constrictivity in~\cite{LLY1984}.
The equivalence between the conditions {\tt(FED)} and {\tt(APM)} is the generalization of Inoue--Ishitani~\cite{II1991} to the case of general Markov operators.
Similarly, the equivalence between {\tt(APM)} and {\tt(MC)} is a generalization of the spectral decomposition theorem (on $L^1(m)$) that has been developed by Lasota, Li, Yorke, Komorn\'ik, Bartoszek during the eighties and nineties, and by Storozhuk, Toyokawa more recently. Refer to~\cite{Bartoszek2008,Toyokawa2020} and references therein.
Actually, we will prove in Theorem~\ref{thmeq} a more complete version of Theorem~\ref{thm:D} where we will provide a sequence of equivalences between {\tt(MC)} and, a priori, weaker conditions.

\begin{rem} \label{rem:ergodic-implies-MC} If $1_X$ is an ergodic $P$-invariant density, then $P$ is {\tt (MC)}. Indeed, since any two ergodic $P$-invariant densities are equal or have disjoint support up to an $m$-null set, see Proposition~\ref{prop:ergodic} in Appendix~\ref{appendix:B}, $1_X$ is actually the unique invariant ergodic density of $P$. Then $P$ satisfies~{\tt (FEM)} and consequently, by Theorem~\ref{thm:D}, $P$ is~{\tt(MC)}.
\end{rem}

\subsection{Hierarchy of classes of Markov operators}
We have considered the classes of \emph{uniform constrictive} {\tt(UC)}, \emph{constrictive} {\tt(C)} and \emph{mean constrictive} {\tt(MC)} Markov operators introduced in Definition~\ref{dfn:1011}.
In this subsection, we introduce in Definition~\ref{asy.const} a new class between them that we call \emph{asymptotic constrictivity} {\tt(AC)}.
In the next subsection, we will provide examples of random maps that show the difference between these classes. First, to show the global picture, we characterize the class of Markov operators in $L^1(m)$ for which there is an invariant density, which we denote~by {\tt(S)}.

\subsubsection{{Straube class}}
In~\cite{straube1981existence} Straube studied the existence of invariant densities for the Perron--Frobenius operator associated with a nonsingular transformation $g:X\to X$.
Namely, Straube showed that there exists a $g$-invariant probability measure which is absolutely continuous with respect to $m$ if and only if there exist $\delta>0$ and $0 < \alpha < 1$ such that $m(A)<\delta$ implies
$m(g^{-k}(A))<\alpha$ for all $k\geq 0$.
Recall that a $g$-invariant absolutely continuous probability measure corresponds to an invariant density of the Perron--Frobenius operator $\mathcal{L}_{g}$.
Moreover,
\[
m\left(g^{-k}(A)\right)= \int 1_A \circ g^k(x) \, dm
 = \int_A \mathcal{L}^k_{g} 1_X \, dm.
\]
More recently, Islam, G\'{o}ra and Boyarsky in~\cite{islam2005generalization} proved also similar necessary and sufficient conditions in terms of Markov operator for the existence of absolutely continuous stationary measures for a certain class of iterated function systems with probabilities.
The following theorem finally provides a
 characterization of the class {\tt(S)} in the spirit of {\tt (MC)}, generalizing the previous results in~\cite{islam2005generalization,straube1981existence}.
 A more complete version will be given in Theorem~\ref{cor:E}.

\begin{mainthm}
\label{GST}
For a Markov operator $P:L^1(m)\to L^1(m)$, the following assertions are equivalent:
\begin{enumerate}
\item \label{item1}
There exists an invariant density for $P$;
\item \label{item2}
There exist $\alpha\in(0,1)$ and $\delta>0$ such that
\[
\sup_{n\ge0}\int_AP^n1_X \, dm<\alpha \quad \text{for any $A\in\mathscr{B}$ with $m(A)<\delta$};
\]
\item \label{item3}
There exist $\alpha\in(0,1)$ and $\delta>0$ such that
\[
\sup_{n\ge0}\int_AA_n1_X \, dm<\alpha \quad \text{for any $A\in\mathscr{B}$ with $m(A)<\delta$};
\]
\end{enumerate}
Moreover, the equivalence also holds by taking $\alpha=1$ in all the above items.
\end{mainthm}

\subsubsection{Weakly almost periodic class}
Consider the class of Markov operators on $L^1(m)$ having an invariant density with the maximal support as in~\eqref{max.supp}.
This class was characterized in~\cite{Toyokawa2020} as the well-known class of Markov operators in $L^1(m)$ called weakly almost periodic.
\begin{dfn}\label{wap}
A Markov operator $P$ on $L^1(m)$ is {\it weakly almost periodic} if $(P^n\varphi)_{n\geq 1}$ is weakly precompact for any $\varphi\in L^1(m)$, that is, if any sequence $(P^{n_k}\varphi)_{k\geq 1}$ contains a further subsequence $(P^{n_{k_j}}\varphi)_{j\geq 1}$ that weakly converges to a function in~$L^1(m)$.
\end{dfn}

By the Dunford--Pettis theorem\footnote{\label{f:DP}
The theorem states that $F\subset L^1(m)$ is weakly precompact (i.e.~any sequence in $F$ contains a weakly converging subsequence in $L^1(m)$) if and only if $F$ is bounded (i.e.~there is $M>0$ such that $\Vert \varphi \Vert<M$ for all $\varphi \in F$) and uniformly integrable (i.e.~for any $\varepsilon >0$, there is $\delta >0$ such that $\int _A \varphi dm <\epsilon$ for any $\varphi \in F$ and $A\in\mathscr B$ with $m(A)<\delta$). Refer to e.g.~\cite[Theorem~4.7.18]{bogachev2007measure}.},
a Markov operator $P$ is weakly almost periodic
if and only if
\begin{enumerate}[leftmargin=2cm]
\item[{\tt(WAP)}] {\it for any $\varepsilon>0$ and $\varphi\in L^1(m)$, there is $\delta>0$ such that
\[
\int_A P^n\varphi\,dm<\varepsilon \quad \text{for any $n\in\mathbb{N}$ and $A\in\mathscr{B}$ with $m(A)<\delta$.}
\]}
\end{enumerate}

As mentioned, due to~\cite[Theorem~3.1]{Toyokawa2020}, we have that {\tt(WAP)} is equivalent to the
existence of a $P$-invariant density with the maximal support. Therefore, {\tt(WAP)} implies {\tt(S)}, and moreover, {\tt (MC)} implies {\tt(WAP)} from Theorem~\ref{thm:D} (see also Proposition~\eqref{prop:wap}).

\subsubsection{Asymptotically constrictive class}\label{sss:acc}
{An equivalent formulation for the class of constrictive Markov operators $P:L^1(m)\to L^1(m)$ is}
\begin{enumerate}[leftmargin=1.75cm]
\item[{\tt(C)}] \label{CDP} {\it for any $\varepsilon>0$ there is $\delta>0$ such that for any $\varphi\in D(m)$, there is $n_0\in\mathbb N$ satisfying
\[
\int_A P^n\varphi \,dm<\varepsilon \quad \text{for any $n\geq n_0$ and $A\in\mathscr{B}$ with $m(A)<\delta$.}
\]}
\end{enumerate}
One of the implications follows easily from the Dunford--Pettis theorem. Indeed, if $P$ is constrictive with a constrictor $F$, since $F$ is
compact (in particular, weakly precompact) in $L^1(m)$ and $\int _AP^n \varphi\, dm\leq d(P^n \varphi ,F) +\sup _{\phi \in F}\int _A\phi \, dm$ for each $\varphi\in D(m)$, $n\in \mathbb N$ and $A\in \mathscr B$, we immediately obtain {\tt (C)}.
One can find the proof of the converse e.g.~in~\cite{komornik1993asymptotic}.
Having in mind these equivalent formulations of {\tt (C)} and~{\tt (WAP)},
 we introduce the following middle property between constrictivity and weak almost periodicity, which we could come up with while looking for non-constrictive examples.

\begin{dfn}\label{asy.const}
A Markov operator $P$ on $L^1(m)$ is said to be {\it asymptotically constrictive} {\tt(AC)} if for any $ \varepsilon>0$, there is $\delta>0$
such that
\begin{equation}\label{dfn:AC}
\limsup _{n \rightarrow \infty}
\int_{A} P^{n} \varphi\, dm<\varepsilon \quad \text{for any $\varphi \in D(m)$ and $A \in \mathscr{B}$ with~$m(A)<\delta$.}
\end{equation}
\end{dfn}

It is straightforward to see that {\tt (C)} implies {\tt(AC)}.
We prove in Theorem~\ref{propC2}~that {\tt(AC)} implies {\tt(MC)}.
We summarize the hierarchy between classes in the following~figure.
\begin{figure}[H]
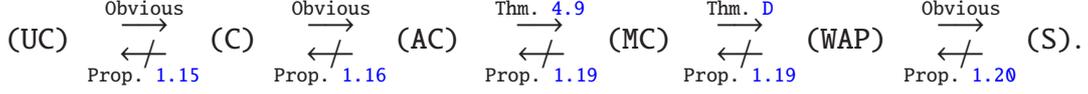
 ~\\[0.3cm]
\begin{center}
\text{\tt (UC) \
$\relationarrow{Obvious}{Prop.~\ref{prop:figa}}{-3em}{-1.8em}$
\ \ (C) \
$\relationarrow{Obvious}{Prop.~\ref{prop:contracting}}{-3em}{-1.8em}$
\ \ (AC) \
$\relationarrow{Thm.~\ref{propC2}}{Prop.~\ref{prop:rotations}}{-3.1em}{-1.8em}$
\ \ (MC) \
$\relationarrow{Thm.~\ref{thm:D}}{Prop.~\ref{prop:rotations}}{-2.8em}{-1.8em}$
\ \ (WAP) \
$\relationarrow{Obvious}{Prop.~\ref{prop:figb}}{-3em}{-1.8em}$
\ \ (S)}. \\[0.1cm]
\end{center}
\caption{The hierarchy of the classes between {\tt (UC)} and {\tt (S)}.}
\label{fig:hierarchy}
\end{figure}

\subsubsection{Sub-hierarchy in {\tt(UC)}} \label{sec:(UC)}
A transition probability $P(x,A)$ on $X\times \mathscr{B}$ is said to be \emph{$m$-nonsingular} if $P(x,A)=0$ for $m$-almost every $x\in X$ whenever $m(A)=0$.
It is well-known that the theories of Markov operators and Markov processes (transition probabilities) are intimately related.
 Indeed, given an $m$-nonsingular transition probability $P(x,A)$ for $x\in X$ and $A\in \mathscr{B}$,
 we can induce a unique Markov operator $P$ on $L^1(m)$ such that $P\varphi$ is the Radon--Nikod\'{y}m derivative with respect to $m$ of the finite signed measure
\[
 \mu_\varphi(A)= \int \varphi(x) P(x,A)\, dm, \quad \text{for $A\in \mathscr{B}$}.
\]
See~\cite{ito1964invariant} or~\cite[Proposition~V.4.2]{neveu1965mathematical} for more details on the construction. 
Moreover, it is easily seen $P^n(x,A)$ is also $m$-nonsingular and induces the iterates $P^n$ of $P$.

Conversely, given a Markov operator $P$ on $L^1(m)$, one can define
\[
 P(\cdot,A)=P^{*}1_A \quad \text{for each $A\in \mathscr{B}$}.
\]
 Notice that $P^{*}1_A$ is an equivalence class and, as a real-valued function, is only defined up to $m$-null sets. It is not hard to see that one may choose in each equivalence class $P^{*}1_A$ a $\mathscr{B}$-measurable function $P(x,A)$ on $X$, for every fixed $A\in\mathscr{B}$, such that it differs from an $m$-nonsingular transition probability only on a negligible set of points. Nevertheless, if $X$ is a Polish space, Neveu proved in~\cite[Proposition~V.4.4]{neveu1965mathematical} that these representations can be chosen appropriately so that $P(x,A)$ becomes a transition probability inducing the Markov operator $P$ on $L^1(m)$, as explained above. Taking into account this relation between Markov operators and Markov processes, in what follows, we introduce two classes of transition probabilities that ensure {\tt(UC)}.

First, we adapt the classical \emph{Doeblin} condition\footnote{Recall that $P(x,A)$ is said to satisfy the \emph{Doeblin} condition if there exist $n_0\geq 1$, $\varepsilon >0$, $\delta <1$ and a probability $\mu$ such that $P^{n_0}(x,A)>\varepsilon$ for all $x\in X$ and $A\in\mathscr{B}$ with $\mu(A)>\delta$ (cf.~\cite{MT2012}). If $\mu$ is in addition required to be absolutely continuous with respect to $m$, then this condition is equivalent to {\tt (D)}.}
for an $m$-nonsigular transition probability  $P(x,A)$ 
as follows:
 \begin{enumerate} \it
 \item[{\tt(D)}] there exist $n_0\geq 1$, $0<\varepsilon<1$, $\delta>0$ and a probability $\mu$ absolutely continuous with respect to $m$ such that $P^{n_0}(x,A)<\varepsilon$ for all $x\in X$ and $A\in\mathscr{B}$ with $\mu(A)<\delta$.
\end{enumerate}

The second class that we introduce is also (an equivalent condition to) a classical property widely discussed and studied in the literature of Markov processes under the name of \emph{uniform ergodicity} that we adapt to $m$-nonsingular transition probabilities: 
\begin{enumerate} \it
 \item[{\tt(D*)}] there exist $n_0\geq 1$, $0<\varepsilon<1$, $\delta>\frac{1}{2}$ and a probability $\mu$ absolutely continuous with respect to $m$ such that $P^{n_0}(x,A)<\varepsilon$ for all $x\in X$ and $A\in\mathscr{B}$ with $\mu(A)<\delta$.
\end{enumerate}
Dorea and Pereira~\cite{Dorea2006} introduced {\tt(D*)} without the absolute continuity of $\mu$
 as an equivalent condition to uniform ergodicity.
In Proposition~\ref{prop:equi-uni-erg} we will show equivalent conditions to {\tt(D*)} including
uniform ergodicity adapted to Markov operators in $L^1(m)$.
Furthermore, we will show that the condition {\tt(D*)} implies that the induced Markov operator $P$ in $L^1(m)$ only admits a unique invariant density, see Remark~\ref{rem:uni-erg}.

\enlargethispage{-1.45cm}
In Proposition~\ref{prop:UC} we will prove that {\tt(UC)} is the class of Markov operators $P$ on $L^1(m)$ that satisfy the following condition:
\begin{enumerate} \it
 \item[{\tt (UC)}] there are $n_0\geq 1$, $0<\varepsilon<1$, $\delta>0$ and a probability $\mu$ absolutely continuous with respect to $m$ such that $P^{n_0}(x,A)<\varepsilon$ for all $A\in\mathscr{B}$ with $\mu(A)<\delta$ and $m$-almost every $x\in X$ (depending on $A$).
\end{enumerate}
In view of these characterizations, one immediately obtains the following relations:
\begin{figure}[H]\label{fig:subuc}
\begin{center}
\text{\tt (D*) \
$\relationarrow{Obvious}{Prop~\ref{prop:dd}}{-2.8em}{-1.7em}$
\ \ (D) \
$\relationarrow{Prop~\ref{prop:UC}}{Prop~\ref{prop:ucd}}{-3.0em}{-1.8em}$
\ \ (UC)}.
\end{center}
\caption{The subhierarchy in {\tt (UC)}.}
\label{fig:subhierarchy}
\end{figure}
In the next subsection of examples, we demonstrate that neither of the converse implications mentioned above holds.

On the other hand, to prove Theorem~\ref{thm:B}, we first show that the assumption in Theorem~\ref{thm:A} is equivalent to the transition probability being $m$-nonsingular and strongly Feller continuous (see the definition at the beginning of Section~\ref{sec:thmB}). As noted in Proposition~\ref{prop:Felle+UC=D}, we conclude that any $m$-nonsingular transition probability that is strongly Feller continuous satisfies condition {\tt (D)}. As a consequence, the transition probability associated with a random map $f$, under the conditions of Ara\'ujo or Brin--Kifer, also satisfies {\tt (D)} (see Remark~\ref{rem:AraujoD}). Furthermore, such transition probabilities satisfy {\tt (D*)} under the conditions in Ara\'ujo--Ayta\c{c}~\cite{AA2017}. In fact, we generalize the argument in Ara\'ujo--Ayta\c{c} to show that a version of Ara\'ujo's condition (other than the one in Remark~\ref{rem:gAraujo}) is sufficient to obtain {\tt (UC)}. See Remark~\ref{rmk:6.7} for further details.

\subsection{Examples and counterexamples}\label{ss:e}
In this subsection, to complete {Remarks~\ref{rem:fiberwise} and}~\ref{rem:gAraujo}, and the list of reverse implications in Figures~\ref{fig:hierarchy} and~\ref{fig:subhierarchy} that fail, we consider several examples from the three aforementioned important classes of random dynamical systems: additive noise, multiplicative noise, and iterated function systems. We also explain how some of these examples can be easily modified to deterministic systems.

\subsubsection{Additive type noise}
\label{subsec:additive}
First, we consider some perturbed systems with additive type noise, which will indicate the reverse implications in Figure~\ref{fig:subhierarchy}.

Let $X$ and $T$ be the closed interval $[0,1]$ equipped with the  Lebesgue measure, denoted by $m$ and $p$, respectively. Consider a measurable map $f_0:X\to X$ and consider the random map $f:T\times X\to X$ given by $f(t,x)=f_t(x)$,
\[
f_t(x) =
\begin{cases}
0 \quad & \text{for} \ x=0,\\
f_0(x) +t \pmod{1} \quad & \text{for} \ x\neq 0.
\end{cases}
\]
In the following proposition $P(x,A)$ and $\mathcal{L}_f$ denote the transition probability and annelead Perron--Frobenius operator associated with this $f$ respectively.
\begin{prop}\label{prop:ucd}
$\mathcal L_f$ satisfies {\tt(UC)}, but $P(x,A)$ does not satisfy {\tt(D)}.
\end{prop}

Next, let us consider
 $X=X_-\cup X_+$ with $X_-=(-1,0]$, $X_+=(0,1]$ and $T=X_+$.
 We equip $X$ and $T$ with
 the normalized Lebesgue measures $m$ and $p$, respectively.
Let $\iota :X \to X$ be the involution
 given by $\iota (x) =x+1$ for $x\in X_-$ and $\iota (x) = x-1$ for $x\in X_+$, so that $\iota (X_-) = X_+$ and $\iota (X_+) = X_-$.
Let $f_0:X\to X$ be a measurable map satisfying that $f_0(X_-) \subset X_+$, $f_0(X_+) \subset X_-$ and $ f_0 \circ \iota = \iota \circ f_0 $ on $ X_+$.
Define the random map $\widetilde f: T\times X_+\to X_+$ by
\[
\widetilde f_t(x)
 = \iota \circ f_0(x)+t \pmod{1}
 \]
and let $f: T\times X\to X$ be the random map given by
\begin{equation}
f_t(x)
=
\begin{cases}
\iota \circ \widetilde f_t(x) & \text{for $x\in X_+$},\\
 \widetilde f_t\circ \iota (x) & \text{for $x\in X_-$}.
\end{cases}\label{eq:dd}
\end{equation}
See Figure~\ref{fig_eg_X}.
Then, the following proposition holds for the transition probability $P(x,A)$ associated with this $f$.
\begin{prop}\label{prop:dd} $P(x,A)$ satisfies {\tt(D)}, but does not satisfy {\tt(D*)}.
\end{prop}

\subsubsection{Multiplicative noise}\label{subsec:multiple}
We secondly focus on a family of perturbed systems with multiplicative noise.
As announced, these examples
prove the assertions in Remark~\ref{rem:gAraujo}.

Let us consider a random map $f:T\times X\to X$ given by the following multiplicative noise,
\begin{equation}\label{eq:0621a}
f_{t}(x)=(1-\varepsilon t) f_{0}(x), \quad x \in[0,1], \ \ t \in[0,1] \ \ \text{and} \ \ 0<\varepsilon<1.
\end{equation}
Here $f_{0}:[0,1] \to [0,1]$ is a measurable map and $X=T=[0,1]$ is equipped with the Lebesgue measure which, as before, we denote by $m$ and $p$ respectively.
The next three examples in Proposition~\ref{prop:multi:ex1} show that both conditions (a) and~(b) in Remark~\ref{rem:gAraujo} are neither sufficient nor necessary to get {\tt(FPM)}.
Realize that the first example
gives an important warning: the natural relaxation of the condition~(b) from ``for all $x\in X$'' to ``for $m$-almost every $x\in X$'' (denoted by \emph{almost-(b)}) is not useful to get {\tt(FPM)}.

\begin{prop}\label{prop:multi:ex1}
Consider the random map given in~\eqref{eq:0621a}.
\begin{enumerate}
\item
Let $f_0(x)=\frac{x}{2}$. Then, $f$ does not satisfy (b), but satisfies (a) and almost-(b). Moreover, {\tt(S)} does not hold. In particular, {\tt(FPM)} does not hold. However, we have that
\[
\lim _{n\to \infty} \frac{1}{n}\sum _{j=0}^{n-1} \delta _{f^j_\omega (x)} = \delta _0 \quad \text{for all $x\in X$ and $\omega \in \Omega$.}
\]
\item
Let $f_0(x)=\frac{x}{2}$ if $x\not=0$ and $f_0(0)=\frac{1}{2}$. Then, $f$ does not satisfy (a), but satisfies (b). Moreover, {\tt(S)} does not hold. In particular, {\tt(FPM)} does not hold.
\item
Let $f_0(x)=2x \text{ (mod 1)}$.
Then, $f$ does not satisfy (a) and~(b), but $\mathcal{L}_f$ satisfies {\tt (C)}. In particular, {\tt(FPM)} holds.
\end{enumerate}
\end{prop}

Proposition~\ref{prop:multi:ex1} (1) provides a class of examples that does not satisfy {\tt(FPM)} but almost-(b) holds. However,
 this example still has a unique non-absolutely continuous physical measure. The next example shows that there is a drastic gap between (b) and almost-(b): finitude
versus infinitude.

Consider a random map $f$ under a multiplicative type noise, given by
\begin{equation}\label{eq:0707c}
f_t(x) =tx + (1- t)f_0(x) \quad x \in X=[0,1], \ \ t \in T=[0,1].
\end{equation}
Now, we consider a concrete $C^1$ map $f_0: X\to X$ with infinitely many sinks, which was essentially given by Ara\'ujo in~\cite[Example 1]{Araujo2001}.
Let $\phi : X \to \mathbb R$ be a $C^1$ function given by
\begin{equation}\label{eq:0707b}
\phi(x) =
X^4 \sin \frac{1}{X} \qquad \text{where $\displaystyle X=\frac{2}{\pi} \left(x-\frac{1}{2}\right)$}.
\end{equation}
Note that $\phi$ can be seen as a $C^1$ function on the Lie group $\mathbb{S}^1=\mathbb R/\mathbb Z$ under the identification of $\mathbb{S}^1$ with $X=[0,1]$.
Notice also that $\phi$ has (countably) infinitely many local maxima and minima, which accumulates to $\frac{1}{2}$.
Therefore, the one-time map $f_0$ of the gradient flow given by $\dot{x} = \nabla \phi (x)$ has infinitely many sinks that accumulate to $\frac{1}{2}$ and whose basin covers $X$ except the sources of $f_0$.

\begin{prop}\label{prop:0707}
 Consider the random map given in~\eqref{eq:0707c} with the time-one map $f_0$ of the gradient flow induced by the potential function~\eqref{eq:0707b}.
 Then, $f$ does not satisfy (b), but satisfies (a) and almost-(b). Moreover, there are infinitely many points $(s_k) _{k\geq 1} \subset X$ such that
 \[
 \mathfrak{B}(\delta_{s_k})=\left\{x\in X \, :\,
 \lim _{n\to \infty} \frac{1}{n}\sum _{j=0}^{n-1} \delta _{
 f^n_\omega (x)}
 =\delta _{s_k}
 \quad \text{for all $\omega\in\Omega$}
 \right\}
 \]
 is a non-empty open set (in particular, has an $m$-positive measure) for all $k\geq 1$, and $\bigcup _{k=1}^\infty \mathfrak{B}(\delta_{s_k}) =X$ up to an $m$-measure zero set.
In particular, {\tt(FPM)} does not hold.
\end{prop}

\subsubsection{Iterated function systems}\label{sss:ifs}
Finally, we consider some random maps generated by iterated function systems~(IFS).
These examples will disprove the reverse implications in Figure~\ref{fig:hierarchy}.

As explained in Remark~\ref{rmk:IFS}, an IFS with probabilities is a random map {$f: T\times X\to X$} on a finite set $T=\{1, 2, \ldots ,k\}$ with a probability measure $p$ where $p(\{i\})=p_i>0$ for $i=1,\dots,k$.
Notice that this setting allows the deterministic case, that is, the case $k=1$.
As before, set $\Omega =T^{\mathbb N}$ and $\mathbb P=p^{\mathbb N}$.
Then, $\mathbb P$ is the Bernoulli probability on $\Omega$,
and the corresponding annealed Perron--Frobenius operator is of the form
\[
 {\mathcal{L}}_f \varphi = \int \mathcal{L}_t \varphi \, dp(t) = \sum_{i=1}^k p_i \mathcal{L}_i\varphi \quad \text{for $\varphi\in L^1(m)$}
\]
where $\mathcal{L}_i$ is the Perron--Frobenius operator of $f_i=f(i ,\cdot )$ for $i=1,\dots, k$.
Throughout this subsection, we keep the setting and notations.

\vspace{0.1cm}

\noindent (a) {\it Random expanding maps.}
Let $X$ be the unit interval $[0,1]$ equipped with the Borel $\sigma$-field $\mathscr B$ and the normalized Lebesgue measure $m$.
Let $f_i$ be a piecewise $C^2$ nonsingular transformation with a finite partition for each $i=1,\dots, k$.
Assume the following expanding (on average) condition:
\begin{eqnarray}\label{expanding_condition}
\sum_{i=1}^k \frac{p_i}{\left\lvert f'_i(x)\right\rvert}
 <1
\quad\text{for every $x\in X$.}
\end{eqnarray}
When $x$ is a discontinuity point of some $f_i$, the left and right limits of $f'_i(x)$ are considered and~\eqref{expanding_condition} with these limits instead of $f_i'(x)$ are required.
Then, the following proposition holds for this $f$.
\begin{prop}\label{prop:figa}
${\mathcal{L}}_f$ satisfies {\tt(C)}, but does not satisfy {\tt(UC)}.
\end{prop}

\enlargethispage{1cm}

\vspace{0.1cm}
\noindent (b) {\it Random contracting maps.}
Let $(X,\mathscr{B},m)$ be as in the previous example.
Consider the case $k=2$ and
\[
f_1(x)=\frac{x}{2}, \quad f_2(x)=\frac{x}{2}+\frac{1}{2} \quad \text{($x\in X$)} \quad \text{with} \quad p_1=p_2=\frac{1}{2}
\]
known as a special case of Bernoulli convolution~\cite{peres1996absolute}.
Then the following proposition holds for this $f$.
\begin{prop}\label{prop:contracting}
 ${\mathcal{L}}_f$ satisfies {\tt(AC)}, but does not satisfy {\tt(C)}.
\end{prop}
\begin{rem}
Note that each $f_i$ does not satisfy even {\tt (S)}: the Dirac measure at $0$ (resp.~$1$) is the only invariant probability measure of $f_1$ (resp.~$f_2$), but it is not absolutely continuous with respect to Lebesgue measure.
Namely, the random map $f$ has a property that each deterministic map $f_\omega$ does not have (such behaviors
 are called \emph{noise-induced phenomena}, and have been attracting attention among physicists, cf.~\cite{galatolo2020existence}).
This is contrastive to other examples in Section~\ref{sss:ifs},
which causes us to work a bit harder in Section~\ref{sss:dce}.
\end{rem}
\begin{rem}
A slightly weaker condition than {\tt (AC)} (i.e.~``some $ \varepsilon>0$'' instead of ``any $ \varepsilon>0$'') appeared in {\cite{Komornik1989,komornik1991asymptotic}} with the name of {\it smoothing} property, and was later called {\it almost constrictivity} in~\cite{komornik1993asymptotic}.
{Note that the definition of smoothing in~\cite{KL1987} is stronger than ones in~\cite{Komornik1989,komornik1991asymptotic}, whence the one in~\cite{KL1987} is indeed equivalent to constrictivity.}
The author of {\cite{Komornik1989}} asserted that a Markov operator $P$ is
{asymptotically periodic (or equivalently constrictive) when $P$ is}
smoothing.
However, this claim cannot be true because the example in Proposition~\ref{prop:contracting} satisfies the smoothing property, but does not satisfy {\tt(C)}.
We also remark that it was proven in~\cite{komornik1993asymptotic} that the almost constrictivity implies {\tt (WAP)}.
Figure~\ref{fig:hierarchy} shows that the results in this paper largely improved it.
\end{rem}

\vspace{0.1cm}
\noindent (c) {\it Random rotations.}
Now let
$X=\mathbb{S}^1=\mathbb{R}/\mathbb{Z}$ equipped with the Borel $\sigma$-field $\mathscr B$ and the normalized Lebesgue measure $m$.
Let $k=2$ and
consider two irrational rotations $f_1$ and $f_2$ with angles $\alpha$ and $\beta$, that is,
\[
f_1(x)=x+\alpha \pmod{1}
\quad {\rm and}\quad
f_2(x)=x+\beta \pmod{1}
\]
where $\alpha, \beta \in [0,1]$.
Let $p_1=p_2=\frac{1}{2}$.
\begin{prop}\label{prop:rotations}
The following hold.
\begin{enumerate}
\item {If $\alpha-\beta$ is irrational}, then ${\mathcal{L}}_{f}$ satisfies {\tt(C)}.
\item If $\alpha$ and $\beta$ are irrational numbers such that $\alpha-\beta$ is rational, then ${\mathcal{L}}_{f}$ satisfies {\tt(MC)} but does not satisfy {\tt(AC)}.
\item If $\alpha$ and $\beta$ are rational numbers, then ${\mathcal{L}}_{f}$ satisfies {\tt(WAP)} but does not satisfy {\tt(MC)}.
\end{enumerate}
\end{prop}

\vspace{0.1cm}
\noindent (d) {\it Direct sums of random contraction and expanding map.}
Let $k=2$ and consider
the direct sum of random transformations, where one satisfies {\tt (S)}
 and the other does not.
That is, let $X= X_-\sqcup X_+$ and equip $X$ with a probability measure $m$ for which both $X_-$ and $X_+$ have a positive measure.
With $p_1=p_2=\frac{1}{2}$, define
\[
f_1(x)=\begin{cases}
\tau_1^-(x) & \text{for $x\in X_-$}\\
\tau_1^+(x) & \text{for $x\in X_+$}\\
\end{cases}
\quad {\rm and}\quad
f_2(x)=\begin{cases}
\tau_2^-(x) & \text{for $x\in X_-$}\\
\tau_2^+(x) & \text{for $x\in X_+$}\\
\end{cases}
\]
where the random dynamics generated by $\tau_1^+,\tau_2^+:X_+\to X_+$ (with equal probabilities) satisfies {\tt (S)},
and the random dynamics generated by $\tau_1^-,\tau_2^-:X_-\to X_-$ does not satisfy {\tt (S)}.
A typical example is the case when $\tau_1^+,\tau_2^+$ are one-dimensional piecewise $C^2$ nonsingular transformations with finite partitions
 satisfying the expanding condition~\eqref{expanding_condition} (with $\tau _i^+$ instead of $f_i$), and
$\tau _1^-(x) =\tau _2^-(x) = \frac{x}{2}$ on $X_-=[0,1]$.
Then, the following proposition holds for this $f$.
\begin{prop}\label{prop:figb}
${\mathcal{L}}_f$ satisfies {\tt(S)}, but does not satisfy {\tt(WAP)}.
\end{prop}

\vspace{0.1cm}
\noindent (e) {\it Random circle homeomorphisms.} A \emph{non-degenerate random walk} on a topological semigroup $S$ of homeomorphisms of the circle $\mathbb{S}^1$ is a random map $f:T\times \mathbb{S}^1\to \mathbb{S}^1$ where $(T,\mathscr{A},p)$ satisfies $\supp p =T$ and $\langle T\rangle^+=S$.
In the case where the random walk is non-degenerate on a \emph{group} of homeomorphisms of $\mathbb{S}^1$, it has been established that there is a unique stationary probability measure (refer to~\cite{deroin2007dynamique}). However, for a random walk that is non-degenerate only on a \emph{semigroup}, this does not generally hold. Nevertheless, Malicet demonstrated in~\cite[Theorem~B and proof of Proposition~4.9]{malicet2017random} that under the assumption that $S$ does not have a finite orbit on $\mathbb{S}^1$, the following conditions hold:
\begin{enumerate}
\item[(1)] the number of ergodic stationary measures is finite, denoted as $\mu_1,\dots,\mu_r$,
\item[(2)] these measures have pairwise disjoint topological supports and correspond exactly to the minimal invariant compact sets of $S$, and
\item[(3)] for any probability measure $\nu$ on $\mathbb{S}^1$,
$$\nu(B_\omega(\mu_1)\cup \dots \cup B_\omega(\mu_r))=1 \quad \text{for $\mathbb{P}$-almost every $\omega \in \Omega$.}$$
\end{enumerate}
Consider a random walk on a finitely generated group $G$ of homeomorphisms of the circle.
That is, an IFS of circle homeomorphisms generated by a symmetric set $T$ of homomorphisms (i.e., if $g\in T$, then $g^{-1}\in T$).
We assume that the group $G=\langle T\rangle=\langle T \rangle^+$ possesses an exceptional minimal set $K$ (i.e., a unique invariant minimal Cantor set for the action of ${G}$ on $\mathbb{S}^1$). In this scenario, we have only one stationary measure $\mu$ supported in $K$. Now, let us consider the reference measure $m=\frac{1}{2}(\mu+\eta)$ where $\eta$ is the normalized Lebesgue measure supported in a gap of $K$.

\begin{prop} \label{prop:IFS-circle}
 $f$ satisfies items 1) and 2) in Definition~\ref{dfn:11} and Equation~\eqref{eq:3} but does not satisfy {\tt(FPM)} with respect to $m$, i.e., 3) does not hold.
\end{prop}

\subsubsection{Deterministic systems}\label{sss:dce}
By slightly modifying examples in Section~\ref{sss:ifs}, one can easily make examples of deterministic dynamical systems that disprove the reverse implications in Figure~\ref{fig:hierarchy}, as follows.
Recall that the baker's transformation is a map $g: X\to X$ on $X=[0,1]^2$ (equipped with the Lebesgue measure $m$) defined by
$
g(x,y)=
(2x,\frac{y}{2})$ when $0\leq x<\frac {1}{2}$ and
$g(x,y)=(2x-2,\frac{y}{2}+\frac{1}{2})$ when $\frac {1}{2}\leq x\leq 1$.
\begin{prop}\label{prop:dce}
The following hold:
\begin{enumerate}
\item
Any one-dimensional piecewise $C^2$ nonsingular transformation with a finite partition satisfying~\eqref{expanding_condition}
is {\tt(C)} but not {\tt(UC)}.
\item The baker's transformation is {\tt (AC)} but not {\tt (C)}.
\item Any irrational rotations are {\tt(MC)} but not {\tt(AC)}.
\item Any rational rotations are {\tt(WAP)} but not {\tt(MC)}.
\item Define $g:X\to X$ as a measurable map preserving a splitting of $X$ by two positive measure sets $X_-$, $X_+$ such that the restriction of $g$ on $X_+$ satisfies {\tt (S)}
and the restriction of $g$ on $X_-$ does not.
 Then, $g$ is {\tt (S)} but not {\tt (WAP)}.
\end{enumerate}
\end{prop}

\subsection{Questions} \label{sec:1.4}
We would like to leave questions on the notion of physical noise, a possible generalization of Theorem~\ref{thm:C} for more general probabilities, and statistical properties of Markov operators satisfying {\tt (AC)}.

\subsubsection{On the physical noise}
We call noise satisfying condition (b) in Remark~\ref{rem:gAraujo} as \emph{physical noise}.
This kind of noise is frequently used in the literature as indicated in Remarks~\ref{rm:0302} and~\ref{rem:multiplicative-noise}. As an important consequence of Remark~\ref{rem:gAraujo}, any continuous random perturbations of a dynamics on a compact space $X$ by a physical noise satisfies {\tt(FPM)}.
Proposition~\ref{prop:multi:ex1} provides a class of examples of continuous random perturbation of a dynamics by an \emph{almost physical noise} (that is, condition almost-(b)) which does not satisfy {\tt(FPM)}.
However,
 this example still has a unique non-absolutely continuous physical measure where any fiberwise statistical basin of attraction covers the full space up to a null set.
Then, a natural question is whether {\tt(FPM)} could be weakened by asking finitely many physical ergodic stationary measures instead of finitely many absolutely continuous ergodic stationary measures. 
Namely, we say that $f$ satisfies {\tt (fpm)} if Definition~\ref{dfn:11} holds but modifying 1) and 3) for the following conditions:
\begin{enumerate}
\item[1')] $\mu_i$ is a physical measure with respect to $m$ for $i=1,\dots,r$;
\item[3')] Equation~\eqref{eq:3} holds.
\end{enumerate}
\begin{rem} \label{rem:fpm} We observe that if $f$ satisfies {\tt (fpm)}, then {\tt (FPM)} also holds with respect to the reference measure $m'=\frac{1}{r}(\mu_1+\dots+\mu_r)$. Indeed, clearly~1) and~2) hold and~3) follows since, by the Birkhoff ergodic theorem, the support of $\mu_i$ is included in $B_\omega(\mu_i,\psi)$ for $\mathbb{P}$-almost every $\omega\in\Omega$, $i=1,\dots,r$ and any bounded function $\psi$.
\end{rem}
On the other hand, Proposition~\ref{prop:0707} shows that considering random perturbations by almost physical noise is still insufficient to get {\tt(fpm)}.

\begin{question}
Find a reasonable sufficient condition to get {\tt(fpm)} generalizing the physical noise.
\end{question}

\subsubsection{Generalization of Theorem~\ref{thm:C}}
Let $f:T\times X\to X$ be a $\mathbb{P}$-random map as defined in Appendix~\ref{sec:apendix} where $\mathbb{P}$ is a shift-invariant probability measure on $(\Omega,\mathscr{F})=(T^\mathbb{N},\mathscr{A}^{\mathbb{N}})$ but not necessarily a Bernoulli probability as before.
Consider the annealed Perron--Frobenius operator ${\mathcal{L}}_f$~given~by
\[
 {\mathcal{L}}_f \varphi = \int \mathcal{L}_\omega \varphi \, d\mathbb{P} \quad \text{for $\varphi \in L^1(m)$}
\]
where $\mathcal{L}_\omega$ is the Perron--Frobenius of $f_\omega=f_t$ for $(\omega_i)_{i\geq 0} \in \Omega$ with $\omega_0=t$. Then Theorem~\ref{thm:D} applies to ${\mathcal{L}}_f$. However, we cannot conclude Theorem~\ref{thm:C} from this because an invariant density of ${\mathcal{L}}_f$ does not correspond in general with a stationary measure with respect to $\mathbb{P}$. To be more clear, an invariant density $h$ for $ {\mathcal{L}}_f$ defines a probability measure $\mu$ by $d\mu=h\, dm$ which satisfies that
\begin{equation} \label{eq:stationary}
\mu(A) =\int (f_\omega)_*\mu(A) \, d\mathbb{P} \quad \text{for all $A\in \mathscr{B}$}.
\end{equation}
However, a stationary measure $\mu$ with respect to $\mathbb{P}$ is not necessarily a measure that satisfies~\eqref{eq:stationary}.
See, for instance,~\cite{matias_2021} where stationary measures for general Markov probabilities are characterized.
This stationary is necessary to obtain invariant measures for the skew-product $F$ in~\eqref{eq:skew} associated with $f$. One can see that if $\mu$ satisfies~\eqref{eq:stationary} and $\mathbb{P}$ is not a Bernoulli measure, $\mathbb{P}\times\mu $ is not in general an invariant measure of $F$ (cf.~\cite[Example~1.4.7]{Arnoldbook}).
Since the $F$-invariance of $\mathbb{P}\times\mu$ is essential for us to prove Theorem~\ref{thm:C}, we do not know how to obtain this result in this case.

\begin{question}
Is it possible to obtain a similar result to Theorem~\ref{thm:C} for a random iteration driven by a general probability $\mathbb{P}$ on $\Omega$ (in particular not of Bernoulli type) of a measurable map $f$ as above?
\end{question}

\subsubsection{Equivalent conditions to {\tt (AC)} and {\tt(UC)}} \label{sec:Questions-(AC)-(UC)}
We next consider the counterpart for {\tt (AC)} in the position of {\tt (APM)} for {\tt (MC)}. First, we introduce some definitions.

 An invariant density $h$ of a Markov operator $P: L^1(m) \to L^1(m)$ is called \emph{mixing} (resp.~\emph{exact}) if
\[
\lim_{n\to\infty}P^n\varphi = h\int \varphi \, dm \quad \text{weakly (resp.~strongly) in $L^1(m)$}
\]
 for any $\varphi\in L^1(m)$ whose support is included in $\supp h$.
{If $h$ is exact and the above converge occurs exponentially fast (i.e.~there are constants $C>0$, $0<\kappa <1$ independently of $\varphi$ such that $\Vert P^n\varphi - h\int \varphi \, dm\Vert \leq C\kappa^n \Vert \varphi \Vert$ for each $n$), we say that $P$ is exponentially exact.
}
 In particular, given a non-singular measurable map $g: X\to X$ and an $\mathcal L_g$-invariant density $h$, a probability measure $m_h$ given in~\eqref{eq:mphi} is mixing (resp.~exact) for $g$ if and only if $h$ is mixing (resp.~exact) for $\mathcal L_g$.
 See Appendix~\ref{appendix:D} for more details.

We also recall the well-known fact that
a Markov operator $P$ is {\tt (C)} if and only if
\begin{itemize}
\item[{\tt (AP)}] \label{AP} \it $P$ is \emph{asymptotically periodic}, that is,
there exist finitely many densities $g_1,\dots,g_r$ with mutually disjoint supports (up to an $m$-null set), positive bounded linear functionals $\lambda_1, \dots, \lambda_r$ on $L^1(m)$ and a permutation $\rho $ of $\{1,\ldots ,r\}$ such that $P g_i = g_{\rho (i)}$ for each $i=1,\dots,r$ and for $\varphi \in L^1(m)$,
\begin{equation}\label{eq:0630exa}
 P^{n} \big( \varphi - \sum _{i=1}^r \lambda _i(\varphi ) g_i\big) \to 0 \quad \text{strongly in $L^1(m)$ as $n\to\infty$.}
\end{equation}
\end{itemize}
It is also known that if $P$ satisfies {\tt (AP)} with $r=1$, then $g_1$ is exact.
Refer to e.g.~\cite{LM}.

Mimicking the above asymptotically periodic condition, we introduce the following class of Markov operators:
\begin{itemize}
\item[{\tt (APW)}] \label{APW} \it $P$ is \emph{asymptotically periodic in weak topology}, that is, {\tt (AP)} holds with ``weakly'' instead of ``strongly'' in~\eqref{eq:0630exa}.
\end{itemize}
\begin{question} \label{q:AC}
Is {\tt (AC)} equivalent to {\tt (APW)}? Moreover, is $g_1$ mixing if $r=1$?
\end{question}

{
The different classes considered in this paper can be classified into three categories according to the type of conditions involved in the definition of the class: conditions on \emph{constrictor}, conditions \`a la \emph{Dunford--Pettis} or conditions on \emph{periodicity} in the limit. Since many of these conditions are ultimately equivalent, we have avoided as far as possible introducing many names to indicate the different equivalent definitions. In Table~\ref{table:question} we organize this classification of classes, but first we introduce two more classes:

\begin{enumerate}[itemsep=0.5cm] \it
\item[{\tt (APE)}] \label{APE} $P$ is \emph{asymptotically periodic at an exponential rate}, that is, {\tt (AP)} holds with
\[
\bigg\| P^{n} \big( \varphi - \sum _{j=1}^r \lambda _j(\varphi ) g_j\big)\bigg \|
\leq C\kappa^n \Vert\varphi\Vert \quad \text{for any $\varphi\in L^1(m)$}
\]
for some $C>0$, $0< \kappa <1$ taken independently of $n$, $\varphi$, instead of~\eqref{eq:0630exa}.
\item[{\tt (WC)}] \label{def:WC} $P$ is \emph{weakly constrictive}\footnote{
In literature, there is another ``weak'' constrictivity in the sense that there is a weakly compact set $F$ with $\lim _{n\to\infty}d(P^nh,F ) = 0$ for any $h\in D(m)$.
However this ``weak'' constrictivity is known to be equivalent to {\tt (C)}, whereas our {\tt (WC)} seems an appropriate weaker version of {\tt (C)} from a perspective of~\eqref{eq:implications} and Figure~\ref{fig:hierarchy}, and hence we dare to use the terminology of weak constrictivity.
}, that is, there exists a weakly compact set $F$ of $D(m)$ such that for any $h\in D(m)$ there exists $(\psi _n)_{n\in \mathbb N} \subset F$ satisfying that
\[
(P^nh -\psi _n ) \to 0\quad \text{weakly in $L^1(m)$ as $n\to\infty$.}
\]
\end{enumerate}

In Proposition~\ref{prop:efinal} we  show that {\tt (APW)} $\Rightarrow$ {\tt (WC)} $\Rightarrow$ {\tt (AC)}. As we see below, these implications have some consequences. The first obvious consequence is that Question~\ref{q:AC} is actually reduced to prove {\tt(AC)} $\Rightarrow$ {\tt(APW)}. Moreover,
\begin{equation}\label{eq:implications}
 \text{{\tt (C)} $\Rightarrow$ {\tt(WC)} $\Rightarrow$ {\tt (MC)}} \quad \text{and} \quad \text{{\tt (AP)} $\Rightarrow$ {\tt (APW)} $\Rightarrow$ {\tt(APM)}.}
\end{equation}
 Indeed, notice first that ``$\lim _{n\to\infty}d(P^nh , F ) = 0$'' in the notion of constrictive Markov operator given in Definition~\ref{dfn:1011} can be rephrased as there is $(\psi _n)_{n\in \mathbb N} \in F$ such that \mbox{$P^nh -\psi _n \to 0$} strongly in $L^1(m)$. Thus, clearly {\tt (C)} implies {\tt(WC)}. Also, since {\tt(AC)} implies {\tt (MC)}, from Proposition~\ref{prop:efinal} and Theorem~\ref{thm:D}, we have {\tt(WC)} $\Rightarrow$ {\tt(MC)}. Similarly, we have {\tt (AP)} $\Rightarrow$ {\tt (APW)} $\Rightarrow$ {\tt(APM)}. Finally, we mention that neither of the converses of the implications in~\eqref{eq:implications}
holds.
To see this, we will prove in Proposition~\ref{prop:mixing-no-exact} that in example (b) of Section~\ref{sss:ifs}, $1_X$ is a $P$-invariant density (called the trivial density) which is mixing, but not exact. In particular, $P$ satisfies {\tt (APW)} with $r=1$ and consequently {\tt(WC)} from Proposition~\ref{prop:efinal}.
However, since any Markov operator with the mixing and non-exact trivial density is not constrictive (cf.~\cite{LM}), $P$ in the example does not satisfy {\tt (C)}. Also, {\tt(MC)} (resp.~{\tt(APM)}) does not imply {\tt(WC)} (resp.~{\tt(APW)}) since {\tt (AC)} is not equivalent to {\tt (MC)} (and thus {\tt(APM)}) from Propositions~\ref{prop:rotations}~(2) and~\ref{prop:dce}~(3).

\begin{table} \vspace{0.5cm}
 \centering
 \footnotesize
 \begin{tabular}{|c||c|c|c|c|}
 \cline{1-5}
 constrictor & \begin{tabular}{@{}c@{}} \tt(UC) \\[-0.1cm] \tiny Def.~\ref{dfn:1011} \end{tabular} &
     \begin{tabular}{@{}c@{}} \tt (C) \\[-0.1cm] \tiny Def.~\ref{dfn:1011} \end{tabular} &
     \begin{tabular}{@{}c@{}} \tt(WC)? \\[-0.1cm] \tiny p.~\pageref{def:WC} \end{tabular} &
     \begin{tabular}{@{}c@{}} {\tt (MC)} \\[-0.1cm] \tiny Def.~\ref{dfn:1011} \end{tabular} \\
  \hline
 Dunford--Pettis & \begin{tabular}{@{}c@{}} \tt(UC) \\[-0.1cm] \tiny Prop.~\ref{prop:UC} (3) \end{tabular}    &
 \begin{tabular}{@{}c@{}} {\tt (C)} \\[-0.1cm] \tiny p.~\pageref{CDP} \end{tabular}  &
 \begin{tabular}{@{}c@{}} {\tt (AC)} \\[-0.1cm] \tiny Def.~\ref{asy.const} \end{tabular} &
 \begin{tabular}{@{}c@{}} {\tt (MC)} \\[-0.1cm] \tiny Thm.~\ref{thmeq} \end{tabular} \\
 \hline
 periodicity & \begin{tabular}{@{}c@{}} {\tt(APE)}? \\[-0.1cm] \tiny p.~\pageref{APE} \end{tabular}   &
 \begin{tabular}{@{}c@{}} {\tt (AP)} \\[-0.1cm] \tiny p.~\pageref{AP} \end{tabular} &
 \begin{tabular}{@{}c@{}} {\tt (APW)}? \\[-0.1cm] \tiny p.~\pageref{APW} \end{tabular} &
 \begin{tabular}{@{}c@{}} {\tt (APM)} \\[-0.1cm] \tiny Thm.~\ref{thm:D} \end{tabular} \\
 \hline
 $r=1$   & exp.~exact ?  & exact  & mixing ? & ergodic \\ \hline
 \end{tabular}
 \caption{Classification of different notions of constrictive type classes.}\label{table:question}
\end{table}

In Theorem~\ref{thmeq} we prove the equivalence between the different definitions for the class of constrictive in mean Markov operators.
In Proposition~\ref{prop:UC}, we prove the equivalence between the definitions of uniformly constrictive \`a la Danford--Pettis and the version on the constrictor.
Finally, to complete Table~\ref{table:question} it remains to solve (together with {\tt(AC)} $\Rightarrow$ {\tt(APW)}) the following question:
\begin{question} Is {\tt(UC)} equivalent with {\tt(APE)}?
Moreover, if {\tt(APE)} holds with $r=1$, then is $g_1$ exponentially exact?
\end{question}
}

We note that if an $m$-nonsigular transition probability $P(x,A)$ satisfies {\tt (D*)}, then the induced operator $P$ on $L^1(m)$ is {\tt (APE)}  with $r=1$ and $g_1$ is exponentially exact, due to the equivalence between {\tt (D*)} and uniform ergodicity in Proposition~\ref{prop:equi-uni-erg}.
More precisely, by the observation $P^n(x,A)=P^{n*}1_A(x)$, Proposition~\ref{prop:equi-uni-erg} and Remark~\ref{rem:uni-erg},
it is straightforward to see that {\tt (D*)} implies the following version of~\eqref{eq:expmixing}: There are a probability measure $\pi$ absolutely continuous with respect to $m$ whose density $g_1$ is $P$-invariant and constants $C>0$, $0<\kappa <1$ such that
for any $A\in\mathscr{B}$ and $\varphi\in L^1(m)$,
\[
\left\vert \int _A (P^{n} \varphi - \lambda _1(\varphi ) g_1)\, dm \right\vert = \left\vert \int \left(P^{n*}1_A - \pi (A)\right) \cdot \varphi \, dm \right\vert \le C\kappa ^n\Vert \varphi \Vert ,
\]
where $\lambda_1(\varphi )=\int \varphi dm$.
This immediately concludes {\tt (APE)} with $r=1$ and that $g_1$ is exponentially exact.
However, we should remember that {\tt (D*)} (and {\tt (D)}) is strictly stronger than {\tt (UC)}, as explained in Section~\ref{sec:(UC)}.

\subsection{Organization of the paper} In Section~\ref{sec:thmB} we will prove Theorem~\ref{thm:B} and show the generalization of Ara\'ujo's result mentioned after Theorem~\ref{thm:A}. The proof of the most general versions of Theorems~\ref{thm:D} and~\ref{GST} is carried on in Sections~\ref{s:thmD} and~\ref{sec:GST} respectively. Theorem~\ref{thm:C} is proved in Section~\ref{s:thmc}. In Section~\ref{s:UC} we study the sub-hierarchy in {\tt(UC)}. Finally, in Section~\ref{s:examples} we provide the proof of the propositions of the examples discussed in the introduction. We also include four appendices that may be of independent interest. In Appendix~\ref{sec:apendix} we briefly mention in a general framework some properties of the annealed Perron--Frobenius operator that we will use. Appendix~\ref{sec:Markov-restriction2} studies and generalizes the restriction of a Markov operator to the support of a density. In Appendix~\ref{appendix:B}, we study the ergodicity of invariant densities.
Finally, in Appendix~\ref{appendix:D} we relate the definition of mixing and exactness in this section with the classical definition of mixing and exactness for deterministic maps.

\section{Feller continuity and quasi-compactness: proof of Theorem~\ref{thm:B}}
\label{sec:thmB}

A Markov transition probability $P(x,A)$ is said to be \emph{Feller continuous}, \emph{strong Feller continuous} and \emph{ultra Feller continuous} if the family of probabilities
$P(x,\cdot)$ varies continuously with respect to the weak* topology,
setwise convergence and
  total variation distance on the space of probabilities respectively. Namely,
for any $x\in X$ and every sequence $\{x_n\}_{n\geq 1}$ with $x_n\to x$,
\begin{enumerate}[itemsep=0.2cm]
  \item $P(x,A)$ is Feller if $P(x_n,\cdot) \to P(x,\cdot)$ in the weak* topology, i.e., if
  \[
  \int \varphi(y) \, P(x_n,dy) \to \int \varphi(y) \, P(x,dy)
  \]
   \text{for all bounded continuous real-valued function $\varphi$ on $X$.}
  \item $P(x,A)$ is strong Feller if $P(x_n,\cdot) \to P(x,\cdot)$ in setwise convergence, i.e., if
  \[
  P(x_n,A) \to P(x,A) \quad  \text{for all $A\in\mathscr{B}$}.
  \]
  \item $P(x,A)$ is ultra Feller if $P(x_n,\cdot) \to P(x,\cdot)$ in total variation distance. i.e., if
  \[
  \left\lVert P(x_n,A)-P(x,A)\right\Vert_{TV} \to 0.
  \]
\end{enumerate}
Here, the total variation distance of two Borel probability measures $\mu$, $\nu$ on $X$ is given by
$
\lVert\mu-\nu\rVert_{TV}=2\cdot \sup _{A\in {\mathscr{B}}} \lvert\mu (A)-\nu (A)\rvert$.
It is clear by definition that ultra Feller continuity implies strong Feller continuity. Moreover, since $X$ is a Polish space, an equivalent way of describing the setwise convergence of a sequence of measures $(\mu_n)_{n\geq 1}$ to $\mu$ is the following:
\[
 \lim_{n\to \infty}  \int \varphi \, d\mu_n = \int \varphi \, d\mu,
\]
for all bounded Borel measurable real-valued function $\varphi$ on $X$.  This is because the simple functions are dense among the bounded Borel measurable real-valued functions on $X$ under the supremum norm. Thus, as a consequence, strong Feller continuity implies Feller continuity. The converse of these implications is not true in general. However, although ultra Feller continuity seems, at first sight, to be stronger than the strong Feller continuity,  it turns out that
the two are almost
 equivalent. More precisely, according to~\cite[Theorem~3.37]{hairer2009non}, if two Markov transition
probabilities $Q(x,A)$ and $R(x,A)$ are strong Feller, then the convolution $QR$ given by
\[
 QR(x,A)=\int R(y,A)\, Q(x,dy)
\]
is an ultra Feller continuous Markov transition probability.  In view of Chapman--Kolmogorov relation,
\[
P^{n+k}(x,A)=\int P^n(y,A) P^k(x,dy) \quad \text{for all $n,k\in \mathbb{N}$},
\]
we have that $P^{n+k}(x,A)$ is the convolution of $P^{n}(x,A)$ and $P^{k}(x,A)$. In particular,  we get the following remark:

\begin{rem} \label{rem1}
  If $P^{n_0}(x,A)$ is strong Feller continuous for some $n_0\geq 1$, then $P^{2n_0}(x,A)$ is ultra Feller continuous.
\end{rem}

To prove Theorem~\ref{thm:B} we need the following proposition that shows some equivalent formulation of the assumption of Theorem~\ref{thm:A}. 
First, recall that a Markov transition probability $P(x,A)$ is said to be $m$-nonsigular if $m(A)=0$ implies that $P(x,A)=0$ for $m$-almost every $x\in X$.
We also address the reader to the beginning of Section~\ref{sec:(UC)}, where the association of a transition probability with a Markov operator is detailed. For an additional context, see~\cite[Chapter V.4]{neveu1965mathematical}.

\begin{prop} \label{prop:ultra-Feller+B}
Let $(X,\mathscr{B},m)$  be a Polish probability space. Consider a Markov transition probability $P(x,A)$ with $x\in X$ and $A\in \mathscr{B}$.
Then, the following conditions are equivalent: 
\begin{enumerate}[label=(\alph*),itemsep=0.2cm]
  \item there exists   $n_0\geq 1$ such that
      \begin{enumerate}[label=(\arabic*)]
        \item $P^{n_0}(x,A)$ is strong Feller continuous, and
        \item $P^{n_0}(x,\cdot)$  is  $m$-nonsigular;
       \end{enumerate}
  \item there exists $n_0\geq 1$ such that
       \begin{enumerate}[label=(\roman*)]
       \item $P^{n_0}(x,A)$ is ultra Feller continuous, and
       \item $P^{n_0}(x,\cdot)$ is absolutely continuous with respect to $m$ for all $x\in X$;
       \end{enumerate}
    \item there exists   $n_0\geq 1$ such that
\[
P^{n_0}(x,dy)=p(x,y)\, dm(y) \ \ \text{with} \ \ x \in X\mapsto p(x,\cdot) \in L^1(m) \ \ \text{continuous.}
\]
\end{enumerate}
Moreover, if $P^n(x,A)$ is an $n$-th  Markov transition probability associated to a Markov operator $P:L^1(m) \to L^1(m)$,
then $P^n(x,A)$ is $m$-nonsingular for each $n\geq 1$.
\end{prop}

\begin{proof}
Assume the condition (a), and show (b).
According to Remark~\ref{rem1}, the condition (1) implies that $P^{2n_0}(x,A)$ is ultra Feller continuous.
Moreover, since $P^{n}(x,A)$ is the convolution of $P$ and $P^{n-1}$ (see~\eqref{eq:recurence}) we get from (2) that $P^n(x,\cdot)$ is $m$-nonsingular for all $n\geq n_0$ and $m$-almost every $x\in X$.
In particular, $P^{2n_0}(x,\cdot)$ is $m$-nonsingular.
In fact, the continuity with respect to the total variation distance implies that $P^{2n_0}(x,\cdot)$ is actually absolutely continuous with respect to $m$ for all $x\in X$.
Indeed, take $A\in\mathscr{B}$ with $m(A)=0$, $x\in X$ and consider $x_n\to x$ with $x_n\in X$ such that $P^{2n_0}(x_n,A)=0$ for all $n\geq 1$.
By the ultra Feller continuity of $P^{2n_0}(\cdot,A)$ we have that $P^{2n_0}(x_n,A)\to P^{2n_0}(x,A)$ as $n\to \infty$.
Then $P^{2n_0}(x,A)=0$ as required. Consequently, we get (i) and (ii) for the positive integer~$2n_0$.

Conversely, (i) and (ii) clearly imply (1) and (2). We now prove that (i) and (ii) are equivalent to $L^1(m)$-continuity of the Radon--Nikod\'{y}m derivative  $p(x,\cdot)$ of $P^{n_0}(x,\cdot)$ with respect to $m$.
But this follows immediately from the well-known fact that\footnote{See the equation before Lemma 2 of~\cite{kusolitsch2010theorem}.}
\[
\left\lVert P^{n_0}(x,\cdot)-P^{n_0}(x',\cdot)\right\rVert_{TV} = \left\lVert p(x,\cdot)-p(x',\cdot)\right\rVert \quad \text{for every $x,x'\in X$}.
\]
This completes the proof of the equivalences.

 The last assertion follows immediately from the fact that $P^n(x,A)=P^{n*}1_A(x)$ for $m$-almost every $x\in X$. Indeed, by duality $\int P^n(x,A) \, dm = \int_A P^n1_X \, dm =0$  whenever $m(A)=0$. Consequently, $P^n(x,A)=0$ for $m$-almost every $x\in X$ concluding that $P^n(x,A)$ is $m$-singular.
 \end{proof}

The following result is  Theorem~\ref{thm:B} stated in terms of Markov processes and Markov operators.

\begin{thm} \label{thm:B-Markov}
  Let $(X ,  \mathscr B,  m)$ be a compact Polish probability space. Consider a Markov operator $P:L^1(m) \to L^1(m)$ and let $P(x,A)$ be an associated transition probability. Assume that there exists   $n_0\in \mathbb{N}$ such that
$P^{n_0}(x,A)$ is strong Feller continuous.
Then $P$ is eventually compact.
\end{thm}

\begin{proof}
According to Proposition~\ref{prop:ultra-Feller+B}, we can assume that $P^{n_0}(x,A)$ satisfies conditions (i) and (ii) in that proposition.
Now, from~\cite[Lemma~1]{HK1964}, we have that the absolute continuity of $P^{n_0}(x,\cdot)$ with respect to $m$ is equivalent to the uniformly absolutely continuity of such measure with respect to $m$ for all $x\in X$. That is, for any $x\in X$ and $\varepsilon>0$, there is $\delta_x=\delta_x(\varepsilon)>0$ such that
\[
P^{n_0}(x,A)<\varepsilon  \quad \text{for every $A\in \mathscr{B}$ with $m(A)<\delta_x$.}
\]

\begin{claim} \label{claim:compact}
  For each $\varepsilon>0$, there is $\delta=\delta(\varepsilon)>0$  such that
  \[\text{$P^{n_0}(x,A)<\varepsilon$ \ \ for every $A\in \mathscr{B}$ with $m(A)<\delta$ and $x\in X$.}
  \]
\end{claim}
\begin{proof}
By the continuity of the family of probability measures $P^{n_0}(x,\cdot)$ with respect to the total variation, for each $x\in X$,   there exists a neighborhood $V(x,\varepsilon)$ of $x$ such that $P^{n_0}(x',A)<\varepsilon$ for all $x'\in V(x,\varepsilon)$ and $A \in \mathscr{B}$ with $m(A)<\delta_x$. Since the union of the open sets $V(x,\varepsilon)$ for $x\in X$ covers the compact set $X$, we can extract a finite subcover $V(x_1,\varepsilon), \dots, V(x_k,\varepsilon)$. Thus, the claim follows by taking $\delta=\min\{\delta_{x_1},\dots,\delta_{x_k}\}$.
\end{proof}
Finally, let us conclude the proof. First, recall that a \emph{bounded} family $F\subset L^1(m)$ is said to be \emph{uniform integrable} (in $L^1(m)$) if for every $\varepsilon>0$ there is $\delta>0$ such that
\[
\int_A |g| \, dm <\varepsilon \quad \text{for all $g\in F$ and $A\in\mathscr{B}$ with $m(A)<\delta$}.
\]
Writing $P^{n_0}(x,dy)=p(x,y)\, dm(y)$ and take $F=\{p(x,\cdot): x\in X\} \subset D(m)$. Claim~\ref{claim:compact} implies that $F$ is uniform integrable (in $L^1(m)$). Then, according to~\cite[Corollary~2.5 (b)]{wu2000uniformly}, the operator $\pi^2$ is compact where $\pi:L^\infty(m)\to L^\infty(m)$ is given by
\[
\pi \psi (x)  =  \int \psi(y) p(x,y)\, dm(y) = \int \psi(y) P^{n_0}(x,dy) = P^{n_0*}\psi(x), \quad \psi\in L^\infty(m).
\]
That is, $\pi$ is the adjoint operator $P^{n_0*}$ of $P^{n_0}$. Thus, $P^{2n_0*}=(P^{n_0*})^2=\pi^2$ is a compact operator.
Hence, by Schauder's theorem~(cf.~\cite[Theorem~4.19]{rudin1991functional}), $P^{2n_0}$ is also compact, concluding the proof.
  \end{proof}

We just indicate that  Theorem~\ref{thm:B} follows immediately from Proposition~\ref{prop:ultra-Feller+B} and Theorem~\ref{thm:B-Markov}. Actually, Theorem~\ref{thm:B} and Theorem~\ref{thm:B-Markov} are equivalent from the following observation.

\begin{rem} \label{rem:P=Lf} As indicated in the introduction, it is well known that the theory of Markov operators and Markov processes are intimately related.  Less known perhaps is that the general theory of Markov operators is actually equivalent to the particular theory of annealed Perron--Frobenius operator associated with random maps on Polish probability space $(X,\mathscr{B},m)$. Indeed, clearly given a random map $f$ we define a Markov operator by means of the annealed Perron--Frobenius operator $\mathcal{L}_f$. Conversely, given a Markov operator $P:L^1(m)\to L^1(m)$, we consider a transition probability $P(x,A)$ which induces $P$. See~\cite[Proposition~V.4.4]{neveu1965mathematical} and Section~\ref{sec:(UC)} for more details.
Notice that
\[P(x,A)=P^*1_A(x) \quad \text{for $m$-almost every $x\in X$}
\]
where $P^*$ is the adjoint operator of $P$.
Now, Kifer proved in~\cite[Theorem~1.1]{Kifer1986} that any transition probability in a Polish\footnote{Actually the only requirement is that $(X,\mathscr{B})$ will be countably generated measurable space, i.e., that the $\sigma$-algebra $\mathscr{B}=\sigma(\mathcal{A})$ for some countable subset $\mathcal{A}$ of $\mathscr{B}$.}
 space $X$ can be represented by a random product of independent and identically distributed measurable maps. That is, there exists a probability space $(T,\mathscr{A},p)$ and a measurable map $f:T\times X \to X$ such that
\[
P(x,A)=p(\{t\in T: f_t(x) \in A\})
\]
where $f_t=f(t,\cdot)$. This implies that $P(x,A)=\mathcal{L}^*_f1_A(x)$ for $m$-almost every $x\in X$ where
{$\mathcal{L}^*_f$ is the adjoint operator of the annealed Perron--Frobenius operator $\mathcal{L}_f$ associated with~$f$}.
Hence, $P^*1_A=\mathcal{L}^*_f1_A$ and therefore $P=\mathcal{L}_f$. To see this final assertion, observe first that any $g \in L^{\infty}(m)$ can be approximated uniformly by simple functions $g_n=\sum_{i=1}^N a_i 1_{A_i}$ where $a_i=a_i(n) \in \mathbb{R}$, $A_i=A_i(n) \in \mathscr{B}$ and $N=N(n)\in\mathbb{N}$.  Then, $g_n$ converges to $g$ in $L^\infty(m)$-norm and hence
\[
P^*g =  \lim_{n\to\infty} P^*g_n =  \lim_{n\to\infty}  \sum_{i=1}^N a_i \, P^*1_{A_i} =
         \lim_{n\to\infty} \sum_{i=1}^N a_i \, \mathcal{L}_f^*1_{A_i} = \lim_{n\to \infty} \mathcal{L}^*_f g_n = \mathcal{L}^*_f g
\]
This implies that $P^*=\mathcal{L}_f^*$ as well as $P=\mathcal{L}_f$ as required.
\end{rem}

\subsection{Generalization of Ara\'ujo's result}
We will prove that under the condition (i) mentioned in Remark~\ref{rem:gAraujo}, the assumptions of Theorems~\ref{thm:A} and~\ref{thm:B} hold. To do this, we need some preliminaries.

The following proposition provides a well-known sufficient condition to obtain that a given Markov transition probability is  Feller continuous. The proof is straightforward and can be found in~\cite[Theorem~4.22]{hairer2006ergodic}.  Recall that a continuous random map is a measurable map $f:T\times X\to X$ such that $f_t=f(t,\cdot)$ is continuous for $p$-almost every $t\in T$.

\begin{prop}  \label{prop:sinprueba}
If $P(x,A)$ is the transition probability associated with a continuous random map, then $P(x,A)$ is Feller continuous.
\end{prop}

 The following result extends~\cite[Theorem~1.2 in Section~4]{kushner2001heavy} to the case of compact Polish probability spaces.

 \begin{thm} \label{thm:book-generalized}
   Let $(X,\mathscr{B},m)$  be a compact Polish probability space.
   Consider a Markov transition probability $P(x,A)$ with $x\in X$ and $A\in \mathscr{B}$. Assume that
   \begin{enumerate}
\item $P(x,A)$ is Feller continuous, and
\item $P(x,\cdot)$ is absolutely continuous with respect to $m$ for all $x\in X$.
\end{enumerate}
Then $P(x,A)$ is a strong Feller continuous transition probability.
 \end{thm}

We will prove the above result by modifying the argument in~\cite{kushner2001heavy} where Theorem~\ref{thm:book-generalized}   was shown with $\mathbb R^d$ instead of $X$.
 For this modification, we need the following.

\begin{lem}\label{lem:0204}
Let $(X, \mathscr B, m)$ be a Polish probability space. Then, for any $A\in \mathscr{B}$   and  $\delta >0$, there exists  $B\in \mathscr{B}$ such that
\[
m(\partial B) =0 \quad \text{and} \quad  m(A \Delta B) <\delta
\]
where $\partial B$ is the boundary of $B$ and  $A \Delta B$ is the symmetric difference of $A$ and $B$.
\end{lem}
 \begin{proof}
 Consider the subalgebra $\mathcal{A}$ in $\mathscr{B}$ of  $m$-continuity sets, i.e., $B\in \mathcal{A}$ if and only if $B\in \mathscr{B}$ and $m(\partial B)=0$.
 Let us prove that the $\sigma$-algebra generated by the $\mathcal{A}$  is $\mathscr{B}$. To see this, fix any metric $d$ compatible with the topology of $X$.
 Note that for $x\in X$, $\partial{B}_r(x)\subset \{ y \in X: d(x,y)=r\}$ where $B_r(x)$ denotes the ball centered at $x$ and of radius $r>0$. In particular, $\partial{B}_r(x)\cap \partial B_s(s)=\emptyset$  for $r\not= s$. But in a space of finite measure, there can only be countably many pairwise disjoint sets of positive measure. Hence, $B_r(x) \in \mathcal{A}$  for all but at most countably many $r$. In particular, $\mathcal{A}$ contains a neighborhood base of $x$. This shows that $\sigma(\mathcal{A})=\mathscr{B}$ as desired. Now, since the subalgebra $\mathcal{A} \subset \mathscr{B}$ generates $\mathscr{B}$, according to the well-known result on approximation generating subalgebras (cf.~\cite[Theorem~1.1]{liu2006smooth}), for every $A\in \mathscr{B}$ and $\delta>0$, there is $B\in \mathcal{A}$ such that $m(A\Delta B) < \delta$. This concludes the proof of the lemma.
 \end{proof}

Now, we are in a position to prove Theorem~\ref{thm:book-generalized}.

\begin{proof}[Proof of Theorem~\ref{thm:book-generalized}]
Let us fix $A\in \mathscr{B}$ and $x\in X$, and consider a sequence $\{ x_n\} _{n\geq 1}$ converging to~$x$. We will prove that $P(x_n,A) \to P(x,A)$ as $n\to\infty$ concluding the strong Feller continuity of $P(x,A)$.
It follows from Portemanteau's theorem~(cf.~\cite[Theorem~13.16]{Klenke2008})
 that $P(x_n, \cdot )$ converges to $P(x, \cdot )$ in the weak$^*$ topology  if and only if $P(x_n, B) $ converges to $P(x, B)$ for any continuity set $B$ of $P(x, \cdot )$, i.e.~any $B\in \mathscr{B}$ satisfying $P(x, \partial B) =0$.
The former condition holds  because we assumed that $P(x,A)$ is Feller continuous, and
thus, recalling that $P(x, \cdot )$ is absolutely continuous with respect to $m$ by  assumption,
\begin{equation}\label{eq:0204c}
P(x_n, B) \to P(x, B) \quad \text{for any $B\in \mathscr{B}$ satisfying $m( \partial B) =0$.}
\end{equation}

Now, according to~\cite[Lemma~1]{HK1964}, the absolute continuity of a
probability measure $\nu$ with respect to $m$  is equivalent to the uniform absolute continuity of $\nu$ with respect to $m$. That is,
$\nu (B) \to 0$ as $m(B) \to 0$. Thus, for each $\varepsilon >0$ and $n\geq 1$, by Lemma~\ref{lem:0204} and the absolute continuity of $P(x_n ,\cdot )$ with respect to $m$,
 one can find  $B_{\varepsilon ,n}\in \mathscr{B}$ such that
\[
m(\partial B_{\varepsilon ,n}) =0 \quad \text{and}
\quad P(x_n, A \Delta B_{\varepsilon ,n}) < \frac{\varepsilon }{2^n}.
\]
Set
\[
C_\varepsilon=\bigcup _{n\geq 1}B_{\varepsilon ,n} \quad  \text{and} \quad  D_\varepsilon=\bigcap _{n\geq 1}B_{\varepsilon ,n}.
\]
\begin{claim} \label{claim-Yushi}
It holds that
\[
P(x, C_\varepsilon) -2\varepsilon  \leq \liminf _{n\to\infty}P(x_n, A) \leq \limsup _{n\to\infty} P(x_n, A) \leq P(x, C_\varepsilon).
\]
\end{claim}
\begin{proof}
Since
\[
m(\partial (C_\varepsilon \setminus D_\varepsilon)) \leq m(\partial C_\varepsilon \cup \partial D_\varepsilon ) \leq
2 \sum _{n=1}^\infty m(\partial B_{\varepsilon ,n}) =0
\]
it follows from~\eqref{eq:0204c} that
\begin{align*}
\limsup _{n\to\infty} P(x_n, C_\varepsilon \setminus B_{\varepsilon ,n}) &\leq \limsup _{n\to\infty} P(x_n, C_\varepsilon\setminus D_\varepsilon) =P(x, C_\varepsilon\setminus D_\varepsilon) \\
&\leq P(x, (A\Delta C_\varepsilon) \cup (A \Delta D_\varepsilon )) \leq 2\sum _{n=1}^\infty \frac{\varepsilon}{2^n} = 2\varepsilon .
\end{align*}
Therefore, since
\begin{align*}
P(x_n, C_\varepsilon) &\leq  P(x_n, C_\varepsilon\setminus B_{\varepsilon ,n}) + P(x_n, A\cup B_{\varepsilon ,n}) \\
 &\leq  P(x_n, C_\varepsilon\setminus B_{\varepsilon ,n}) + P(x_n, A\Delta B_{\varepsilon ,n}) + P(x_n, A)
\end{align*}
by applying~\eqref{eq:0204c} again
(note that $m(\partial C_\varepsilon) \leq \sum _{n=1}^\infty m(\partial B_{\varepsilon ,n}) =0$)
we have
\[
\liminf _{n\to\infty}P(x_n, A) \geq \liminf _{n\to\infty}P(x_n, C_\varepsilon) -2\varepsilon -\lim_{n\to\infty} \frac{\varepsilon}{2^n} = P(x, C_\varepsilon) -2\varepsilon .
\]
On the other hand, since
\[
P(x_n, A) \leq P(x_n,  B_{\varepsilon ,n})  + P(x_n, A\Delta B_{\varepsilon ,n})  \leq P(x_n, C_\varepsilon) +  \frac{\varepsilon}{2^n}
\]
we have
\begin{align*}
\limsup _{n\to\infty} P(x_n, A) \leq P(x, C_\varepsilon)
\end{align*}
concluding the proof of the claim.
\end{proof}

Finally, observe that
\[
\left\lvert P(x,C_\varepsilon ) - P(x,A)\right\rvert \leq P(x,A\Delta C_\varepsilon) \leq \sum _{n=1}^\infty  \frac{\varepsilon }{2^n}  = \varepsilon .
\]
Thus, $P(x,C_\varepsilon) \to P(x,A)$ as $\varepsilon\to 0$. Hence, by Claim~\ref{claim-Yushi},
 we get that
 \[
 \lim _{n\to\infty} P(x_n, A) = \lim _{\varepsilon\to\infty} P(x, C_\varepsilon ) = P(x, A).
 \]
This concludes the proof of the theorem.
\end{proof}

\begin{rem} A converse of Theorem~\ref{thm:book-generalized} follows from~\cite[Proposition~12.1.7]{douc2018markov}. Also, from~\cite[Remark~2]{ito1964invariant}, by arguing as in the implication of  (b) from (a) in Proposition~\ref{prop:ultra-Feller+B},  we immediately get the following more specific converse: if $P(x,A)$ is a strong Feller continuous Markov transition probability on a Polish probability space $(X,\mathscr{B},m)$, then there exists a probability measure $m^*$ on $(X,\mathscr{B})$ such that $m$ is absolutely continuous with respect to $m^*$ and $P(x,A)$ satisfies (1) and (2) with $m^*$ instead~$m$.
\end{rem}

\begin{cor} \label{cor:rem}
Let $(X,\mathscr{B},m)$  be a compact Polish probability space.
Consider a continuous random map \mbox{$f:T\times X\to X$}  where $(T,\mathscr{A},p)$ is a probability space and
 let $P^n(x,A)$ be the associated $n$-th transition probability. Assume that there is $n_0\geq 1$ such that
 \[
 P^{n_0}(x,\cdot) \ \ \text{is absolutely continuous with respect to $m$ for all $x\in X$}.
 \]
Then, for each $x\in X$,  there exists $p(x,\cdot)\in D(m)$ such that
\[
P^{2n_0}(x,A)=p(x,y) \, dm(y) \quad \text{and} \quad x \in X\mapsto p(x,\cdot)\in L^1(m) \ \ \text{is continuous}.\]
Moreover, the annealed Perron--Frobenius operator $\mathcal{L}_f$ associated with $f$ is eventually compact.
\end{cor}
\begin{proof}
According to Proposition~\ref{prop:sinprueba}, $P^{n_0}(x,A)$ is a Feller continuous Markov transition probability. Hence, by Theorem~\ref{thm:book-generalized},  we have that $P^{n_0}(x,A)$ is actually strong Feller continuous. Therefore, Remark~\ref{rem1}, Proposition~\ref{prop:ultra-Feller+B} and Theorem~\ref{thm:B-Markov} immediately imply the conclusion of the corollary.
\end{proof}

\section{Existence of invariant measures: proof of Theorem~\ref{GST}} \label{sec:GST}

In what follows, $(X,\mathscr{B},m)$ denotes any abstract probability space. {We will prove in this section a generalization of Theorem~\ref{GST} which will be used to prove Theorem~\ref{thm:D} in the next section. But first, to prove these results}  we will need  some
preliminaries that can be interesting by themselves.

\subsection{Fundamental lemma}
If $\phi$ belongs to $L^1(m)$, then the support of $\phi$, denoted by $\supp \phi$, is not defined in a   unique way, since $\phi$ can be represented by two functions whose values are different in a zero $m$-measuring set.
However, since it is common to simplify terminology, we consider  $\supp \phi$  as a set defined by the relation $\phi\not =0$.
Moreover, if we want to emphasize that a   relationship between sets holds except a set of zero $m$-measure, we say that it holds \emph{up to an $m$-null set}.
For instance, $A\subseteq B$ up to an $m$-null set if $m(A\setminus B)=0$.

\begin{lem}\label{supp}
Let $E$ be either $L^1(m)$ or $L^\infty(m)$. Let $P:E\to E$ be a positive linear bounded operator and consider $\phi, \psi \in E$  with $\psi\geq 0$.  It holds that
\begin{enumerate}
\item $ \supp P\phi \subset \supp P1_{\supp \phi }$ and $\supp P\psi=\supp P1_{\supp \psi }$,
\item if $\supp \phi \subset \supp \psi$, then $\supp P\phi\subset \supp P\psi$.
\end{enumerate}
\end{lem}

\begin{proof} It is easy to see that (1) implies (2). Indeed, if $\supp \phi \subset \supp \psi$, we have that $1_{\supp \phi } \leq 1_{\supp \psi}$ and since $P$ is positive linear operator, $P1_{\supp \phi } \leq P 1_{\supp \psi}$. This implies that $\supp P1_{\supp \phi}\subset \supp P1_{\supp \psi}$. Now, (1) immediately implies the conclusion of~(2). But, actually, if (2) holds for any pair $\phi$ and $\psi$ in the assumptions of the lemma, we also get (1) as follows. Since $\supp \phi = \supp 1_{\supp \phi}$, by (2) we get that $\supp P\phi\subset \supp P1_{\supp \phi}$. Same argument and inclusion hold for $\psi$. But, in this case, taking into account that $\psi \geq 0$, we can also apply (2) to $\supp 1_{\supp \psi} = \supp \psi$ getting the other inclusion and proving (1).

In view of the above observation, the proof of the lemma is reduced to show~(2). To do this,
let $\phi$ and $\psi$ be as in the statement.
From the approximation by simple functions,
$\phi$ and $\psi$ can be written as the pointwise limit of a functions
\[
\phi_n=\sum_{k=1}^{N}a_k 1_{F_k} \quad \text{and} \quad \psi_n=\sum_{k=1}^{M}b_k 1_{G_k}
\]
respectively, where $N=N(n), M=M(n)\in \mathbb{N}$, $a_k=a_k(n) \not=0$, $b_k=b_k(n)>0$ and $F_k=F_k(n),G_k=G_k(n) \in \mathscr{B}$   such that $F_\ell \cap F_j =\emptyset$  and $G_\ell\cap G_j = \emptyset$ if $\ell\not =j$. Moreover, $\phi_n$ and $\psi_n$ can be chosen\footnote{For a non-negative measurable function $g$, we consider the mesh with width $2^{-n}$ from the level $2^{-n}$ up to $2^n$, i.e.,
\[
g_n = \sum_{k=0}^{2^{2n}-1} (k+1)2^{-n}1_{g^{-1}(k2^{-n},\,(k+1)2^{-n}]}+2^n1_{g^{-1}(2^n,\,\infty)}.
\]
Then we can see that $g_n$  pointwise converges to $g$ since $0\leq g_n(x)-g(x)<2^{-n}$ on the set where $g\leq 2^n$. Moreover,  $\supp g=\supp g_n$.
For a general $g$, we consider $g=g^+ - g^-$ where $g^+$ and $g^-$ are the positive and negative parts of $g$ which we approximate by simple functions $g^+_n$ and $g_n^-$ as before. Finally, it is simple to verify that $g_n=g^+_n-g^-_n$ converges in the norm $L^1(m)$ or $L^{\infty}(m)$ where appropriate, to $g$ and $\supp g_n = \supp g$.}
so that
\begin{enumerate}[label=(\roman*)]
  \item $\supp \phi_n = \supp \phi \subset \supp \psi = \supp \psi_n$, and
  \item $\phi_n$ and $\psi_n$ converges in norm of $E$ to $\phi$ and $\psi$, respectively.
\end{enumerate}
Since $P$ is a bounded operator, (ii) implies that $P\phi_n \to P\phi$ and $P\psi_n \to P\psi$. Then
\begin{enumerate}[resume, label=(\roman*)]
\item $m((\supp P\phi_n) \Delta (\supp P\phi))\to 0$ and $m((\supp P\psi_n)  \Delta (\supp P\psi))\to 0$ as $n\to\infty$.
\end{enumerate}
Note that by (i), $\phi_n$ and $\psi_n$ satisfy the assumption in (2). It is easy to see that if~(2) holds for $\phi_n$ and $\psi_n$, then (iii) implies that $\supp P\phi \subset \supp P\psi$ up to an $m$-null set. Hence, we reduce the proof to check the conclusion of (2) for $\phi_n$ and $\psi_n$.

Taking into account $a_k \not = 0$, $P1_{F_k} \geq 0$ and the linearity of $P$, we see that
\begin{equation}\label{eq:lem1}
\begin{aligned}
\supp P\phi_n=\supp
\sum_{k=1}^N a_k P1_{F_k}   &\subset \bigcup_{k=1}^N \supp P1_{F_k}  \\ &=  \supp  \sum_{k=1}^N P1_{F_k}=\supp P1_{\cup_{k=1}^N F_k} = \supp P1_{\supp \phi_n}.
\end{aligned}
\end{equation}
Now, using that $b_k> 0$,  $P1_{G_k} \geq 0$ and the linearity of $P$,
\begin{equation}\label{eq:lem3}
      \supp P1_{\supp \psi_n}=\supp P1_{\cup_{k=1}^M G_k}=\supp \sum_{k=1}^MP1_{G_k} = \supp \sum_{k=1}^M b_k P1_{G_k} = \supp P\psi_n.
\end{equation}
Observe that~\eqref{eq:lem1} and~\eqref{eq:lem3} implies (1) for $\phi_n$ and $\psi_n$. As we see at the beginning of the proof, (1) implies (2) and we conclude the lemma.
\end{proof}

\subsection{Restrictions of Markov operators}\label{sec:Markov-restriction}

Fix $S\in \mathscr{B}$ with $m(S)>0$. Let us consider the probability space $(S,\mathscr{B}_S,m_S)$ where
\[
 \mathscr{B}_S=\{A\in \mathscr{B}: A\subset S\} \quad \text{and} \quad m_S(A)=\frac{m(A)}{m(S)}, \ \ A\in \mathscr{B}_S \ \ \ \left(\text{i.e.,} \ dm_S=\frac{1_S}{m(S)}\, dm\right).
\]
Denote  $L^1(S,\mathscr{B}_S,m_S)$ by $L^1(m_S)$. Observe that $L^1(m_S) \hookrightarrow L^1(m)$ by means of the inclusion
\begin{equation}\label{eq:def-mS}
  1_S: \phi\in L^1(m_S) \mapsto 1_S\phi \in L^1(m) \quad \text{given by} \quad 1_S\phi(x)= \begin{cases} \phi(x) & x\in S, \\ 0 & x\in X\setminus S.\end{cases}
\end{equation}
Although we are using the same notation $1_S$ to designate the characteristic function of  $S$ and the inclusion by  the characteristic function, it should cause no confusion.
Now, given a Markov operator \mbox{$P:L^1(m)\to L^1(m)$,} we define an operator
\begin{equation}\label{def:P_S}
   P_S: L^1(m_S) \to L^1(m_S),  \qquad P_S\phi = 1_S\cdot P(1_S\phi)  \quad  \text{for} \ \phi\in L^1(m_S).
\end{equation}
It is not difficult to see that $P_S$ is a positive contraction of $L^1(m_S)$, that is, $P_S\phi\geq 0$ if $\phi\geq 0$ and $\|P_S\|\leq 1$. Taking advantage of the abuse of notation, we can extend $P_S$ to $L^1(m)$ as follows:
\[
 P_S\phi = 1_S \cdot P(1_S \cdot \phi) \quad  \text{for} \  \phi\in L^1(m).
\]
When no confusion can arise, we will omit the dot in the above expression and in~\eqref{def:P_S}.  We also identify $L^1(m_S)$ with
\[
1_S(L^1(m_S))=\{\phi \in L^1(m): \supp \phi \subset S\} \subset L^1(m).
\]

\begin{lem} \label{lem:P_S-Markov} Let $P:L^1(m)\to L^1(m)$ be a Markov operator and consider $S\in\mathscr{B}$ with $m(S)>0$.
If $\supp P1_S \subset S$ up to an $m$-null set, then  $P_S:L^1(m_S) \to L^1(m_S)$ defined in~\eqref{def:P_S} is a Markov operator and
\[
P_S\phi=P(1_S\phi) \quad \text{for all $\phi\in L^1(m)$}.
\]
\end{lem}
\begin{proof}
Let $\phi \in L^1(m)$. Note that $\supp 1_S\phi \subset S = \supp 1_S$. Thus, by Lemma~\ref{supp}, $\supp P(1_S\phi) \subset \supp P1_S \subset S$ up to an $m$-null set. In particular,  $1_S  P(1_S\phi) = P(1_S\phi)$.  Then, for any $\phi \in L^1(m)$,
\[
\int P_S\phi \, dm_S = \frac{1}{m(S)} \int_S P(1_S\phi) \, dm = \frac{1}{m(S)}\int P(1_S\phi) \, dm  =\frac{1}{m(S)} \int_S \phi \, dm =\int \phi \, dm_S.
\]
Therefore, $P_S$ is a Markov operator.
\end{proof}

The implication in the previous lemma is actually an equivalence. To see this observation and other interesting equivalent conditions see Proposition~\ref{prop:P_h-Markov} in the Appendix~\ref{sec:Markov-restriction2} where the theory of the restriction of a Markov operator is extended and clarified.
For the following result, recall the notion of weak almost periodicity {\tt(WAP)} in Definition~\ref{wap}.  As mentioned, according to~\cite[Theorem~3.1]{Toyokawa2020},  {\tt(WAP)} is equivalent to the existence of an invariant density with the maximal support.

\begin{prop}  \label{prop:wap} Let $P:L^1(m)\to L^1(m)$ be a Markov operator and consider $h\in D(m)$ such that $Ph=h$ and set $S=\supp h$. Then, $P_S:L^1(m_S) \to L^1(m_S)$ is a weakly almost periodic Markov operator. Moreover, $P_Sh=h$ and $P_S\phi = P(1_S\phi)$ for all $\phi \in L^1(m)$.
\end{prop}
\begin{proof}
Note that $\supp h = S = \supp 1_S$ and by Lemma~\ref{supp} and the $P$-invariance of $h$, it follows that $S=\supp h = \supp Ph  = \supp P1_S$ up to an $m$-null set as well as $m(S)>0$. Hence, Lemma~\ref{lem:P_S-Markov} implies that $P_S$ is a Markov operator and $P_S \phi = P(1_S\phi)$ for all $\phi \in L^1(m)$. To prove that $P_S$ is also weakly almost periodic, it suffices to see that $P_Sh=h$ since $P_S^*1_S =1_S$  by the Markov property and thus $h$ has trivially the maximal support (as a fixed point of $P_S$). But this is proved as follows:
$
    P_S h = P(1_S h) = Ph =h.
$
\end{proof}

\subsection{Proof of  Theorem~\ref{GST}}
In this subsection, we prove a more general version of Theorem~\ref{GST}. In particular, we remark that the new item~\eqref{item5} below was shown in~\cite{socala1988existence} to be a sufficient condition for the existence of a $P$-invariant density, and below we show it is indeed also a necessary condition. 
In this context, we also refer to the work of Ito~\cite{ito1964invariant}, who investigates similar necessary and sufficient conditions for the existence of a finite invariant measure under a Markov process, which is equivalent to the reference measure.

\begin{thm} \label{cor:E} Let $(X,\mathscr{B},m)$ be an abstract probability space and consider a Markov operator $P: L^1(m) \to L^1(m)$. Then, the following conditions are  equivalent:
\begin{enumerate}
\item \label{item1}
There exists an invariant density for $P$;
\item \label{item2}
There exist $\alpha\in(0,1)$ and $\delta>0$ such that
\[
\sup_{n\ge0}\int_AP^n1_X \, dm<\alpha \quad \text{for any $A\in\mathscr{B}$ with $m(A)<\delta$};
\]
\item \label{item3}
There exist $\alpha\in(0,1)$ and $\delta>0$ such that
\[
\sup_{n\ge0}\int_AA_n1_X \, dm<\alpha \quad \text{for any $A\in\mathscr{B}$ with $m(A)<\delta$};
\]
\item \label{item4}
There exist $\alpha\in(0,1)$ and $\delta\geq 0$ such that
\[
\inf_{n\ge1}\int_A A_n1_X \, dm>1-\alpha \quad \text{for any $A\in\mathscr{B}$ with $m(A)>\delta$}.
\]
\item \label{item5} There exist $\alpha\in(0,1)$ and $\delta>0$ such that
\[
\limsup_{n\to \infty} \int_AP^n1_X \, dm<\alpha \quad \text{for any $A\in\mathscr{B}$ with $m(A)<\delta$};
\]
\item \label{item7} There exist $\alpha\in(0,1)$ and $\delta>0$ such that
\[
\limsup_{n\to \infty} \int_A A_n1_X \, dm<\alpha \quad \text{for any $A\in\mathscr{B}$ with $m(A)<\delta$};
\]
\end{enumerate}
The function $1_X$ in (5) and (6) could be substituted by any arbitrary density $h\in D(m)$.
Moreover, the equivalence also holds by taking $\alpha=1$ in all the above items.
\end{thm}

\begin{proof}
We first prove~(1) implies~(2).
Let $g\in D(m)$ be an invariant density of $P$, that is, $Pg=g$.
Hence, Proposition~\ref{prop:wap} implies that $P_S:L^1(m_S) \to L^1(m_S)$ is a weak almost periodic Markov operator where $S=\supp g$.
According to~\cite[Theorem~3.1]{Toyokawa2020}, this is  equivalent to weak compactness of $\{P_S^n1_S\}_{n\geq 0}$ in $L^1(m_S)$.
Hence, by Dunford--Pettis characterization, for each $\varepsilon>0$, there exists $\delta>0$ such that
\begin{equation} \label{eq:DP} \sup_{n\geq 0} \int_{A}P_S^n1_S \, dm_S <\varepsilon \quad \text{for any $A\in \mathscr{B}_S$ with $m_S(A)<\delta$}. \end{equation}
Note that clearly $m(S)>0$ since $g$ is a density function. Thus, $m(X\setminus S) <1$ and hence,  we can
take $\varepsilon>0$ small enough such that $\alpha\coloneqq\varepsilon \cdot m(S) +m(X\setminus S) <1$. Then, by Dunford--Pettis characterization, there is $\delta>0$ such that~\eqref{eq:DP} holds. Let $\delta'=\delta \cdot m(S)>0$. Observe that, according to Proposition~\ref{prop:wap},  $P_S 1_S = P1_S$  and then, $P^n_S 1_S = P^n 1_S$ and $\supp P^n1_S \subset S$ up to an $m$-null set for all $n\geq 0$. Finally, for any $A\in \mathscr{B}$ with $m(A)<\delta'$ it holds that $m_S(A\cap S) <\delta$ and then,
\begin{align*}
\sup_{n\ge 0}\int_A P^n1_X \, dm  &= \sup_{n\ge0} \left(\int_A P^n 1_S \, dm + \int_A P^n1_{X\setminus S} \, dm\right) \\
&=\sup_{n\ge0} \left( m(S) \cdot \int_{A\cap S} P^n_S 1_S \, dm_S + \int_A P^n1_{X\setminus S} \, dm\right)\\
&< m(S) \cdot \varepsilon + \sup_{n\ge 0}\int_{X\setminus S}  P^{n*}1_X \, dm\\
&<\varepsilon\cdot m(S) +m(X\setminus S)=\alpha.
\end{align*}

The implication from the condition~(2) to~(3) follows immediately taking into account that
\[
\int_A A_n1_X \, dm = \frac{1}{n} \sum_{i=0}^{n-1} \int_A P^i1_X \, dm < \frac{1}{n} \sum_{i=0}^{n-1}  \alpha =\alpha
\]
for all $n\geq 1$ and $A\in\mathscr{B}$ with $m(A)<\delta$.

{We show the implication from~(3) to~(4).
Observe that since $P$ is a Markov operator, then $P^*1_X=1_X$ and thus,
\[
\int A_n 1_X \, dm = \int A^*_n 1_X \, dm = \frac{1}{n}\sum_{k=0}^{n-1} \int P^{k*}1_X \,dm = 1.
\]
From this and assumption~(3), we have
\begin{align*}
\inf_{n\geq 1}\int_{X\setminus E}A_n1_X \, dm
 = \inf_{n\geq 1} \int A_n 1_X \, dm - \sup_{n\geq 1} \int_E A_n1_X \, dm
> 1 -\alpha
\end{align*}
for all  $E\in\mathscr{B}$ with $m(E)<\delta$. Setting $\delta$ in~(4) as $1-\delta$ in~(3), we get the desired implication.
}

Next we show the implication from~(4) to~(1).
Suppose contrarily that there is no invariant density for $P$.
Then, according to~\cite[Theorem~4.6 and Lemma~4.5]{Krengel}, there exists $\psi\in L^\infty(m)$ with $\psi\geq 0$, fully supported on $X$ (i.e., $\supp \psi =X$) and $\|A^*_n\psi\|_{\infty} \to 0$ as $n\to 0$. Now, since $h$ is fully supported on $X$, for $\delta\geq 0$ as in the condition~(4), we can find $\eta>0$ such that the set $E=\{h\geq \eta\}$ satisfies $m(E)>\delta$.  Since $\psi\ge\eta$ on $E$, $ \eta^{-1} \psi \geq 1_E$. Hence,
\begin{align*}
\int_{E} A_n1_X  \, dm
&=\int 1_{E} A_n 1_{X} \, dm \leq \int \frac{\psi}{\eta} \, A_n 1_{X} \, dm
=\frac{1}{\eta}\int  A_n^*\psi \, dm
\le\frac{1}{\eta}\lV A_n^*\psi \rV_{\infty}
\to 0
\end{align*}
as $n\to\infty$.
This contradicts~(4) and the proof is done.
Finally, observe that the previous arguments work assuming $\alpha=1$ in all of the items.

To complete the proof we will see now the equivalence with~(5), (6).
Clearly (2)  implies (5).
Also, (5) immediately implies (6). On the other hand, from~\cite[Theorem~1]{socala1988existence}, we have that~(6) implies~(1).
Moreover, a slight modification of the argument of~\cite[Theorem~1]{socala1988existence} also shows that~(6) implies~(1).
Indeed, assume (6) instead of (5) (which is called (C) in~\cite{socala1988existence} and where $1_X$ could be substituted by any density in $D(m)$). Define
\[
\lambda(\varphi )= \mathrm{Lim} \, \int A_n1_X \, \varphi\, dm \quad \text{for $\varphi\in L^\infty(m)$}
\]
where $\mathrm{Lim}$ is a Banach limit (i.e.~replace $P^n$ in the definition of $\lambda (h)$ of~\cite{socala1988existence} with $A_n$).
The rest of the proof in~\cite[Theorem~1]{socala1988existence} literally  works to conclude~(1) in our case, except proving $\lambda(P^*\varphi)=\lambda(\varphi)$ for each $\varphi \in L^\infty (m)$,
that is the unique calculation that one needs to do.
To see it,
\begin{align*}
\lambda(P^*\varphi) &= \mathrm{Lim}  \, \int  A_n1_X\, P^*\varphi\,dm
=  \mathrm{Lim} \, \int \frac{1}{n}\sum_{i=0}^{n-1} P^{i+1}1_X \, \varphi \, dm\\
& =\mathrm{Lim} \, \int \left( A_{n} 1_X + \frac{1}{n}P^n1_X - \frac{1}{n} 1_X\right) \, \varphi \, dm = \lambda (\varphi) + \mathrm{Lim} \,  \frac{1}{n} \int ( P^n1_X - 1_X)\, \varphi \, dm.
\end{align*}
In the last equality we used the linearity of Banach limits.
On the other hand,
\[
\left\vert \int ( P^n1_X - 1_X)\, \varphi \, dm \right\vert \leq \Vert \varphi\Vert _{\infty} \ \int (P^n1_X + 1_X )\, dm =   2\Vert \varphi\Vert _{\infty}
\]
since $P$ is a Markov operator.
Hence we get $\frac{1}{n} \int ( P^n1_X - 1_X)\, \varphi \, dm \to 0$ as $n\to \infty$, and obtain the desired equality $\lambda(P^*\varphi)=\lambda(\varphi)$.
The version for $\alpha=1$ follows since, actually, the assumption in~\cite[Theorem~1]{socala1988existence} is (5) for $\alpha=1$ and where $1_X$ could be substituted by any density in $D(m)$.  This completes the proof.
\end{proof}

The implications between~(1)--(4)  in Theorem~\ref{cor:E} strongly require that $P$ is a Markov operator.
However, the equivalence between (1) and (4)  for $\alpha=1$ holds even under the weaker assumption that the linear positive operator $P$ is just a contraction (i.e., $\|P\|_{\rm op}\leq 1$).  As a final result of this section, we will prove it as follows.  Compare with~\cite[Theorem~4.2 and Corollary~4.7 in Chapter~3]{Krengel} {and~\cite[Theorem~2]{neveu1967existence}.}

\begin{thm}
If $P:L^1(m)\to L^1(m)$ is a linear positive contraction, then the following are equivalent:
\begin{enumerate}[label=(\roman*)]
  \item \label{item11} there exists a $P$-invariant density;
  \item \label{item22} there exists $\delta\geq 0$ such that
  \[
  \inf_{n\geq 0} \int_E P^{n}1_X \, dm >0 \quad \text{for all $E\in\mathscr{B}$ with $m(E)>\delta$;}
  \]
  \item \label{item33} there exists $\delta\geq 0$ such that
  \[
  \inf_{n\geq 0} \int_E A^{n}1_X \, dm >0 \quad \text{for all $E\in\mathscr{B}$ with $m(E)>\delta$.}
  \]
\end{enumerate}
\end{thm}
\begin{proof}
Observe that the proof of the implication of (1) from (4) in Theorem~\ref{cor:E} does not require that $\int P\varphi \, dm =\int \varphi \, dm$ for all $\varphi\in L^1(m)$. This shows that~(iii) implies~(i).
The equivalence between~(i) and~(ii) follows from~\cite[Theorem~4.2]{Krengel} as follows.
This result says that if $h$ is a $P$-invariant density, then
\[
\inf_{n\geq 0} \int_{\tilde{E}} P^n 1_X \, dm >0 \quad \text{for all $\tilde{E}\subset  \supp h$ with $m(\tilde{E})>0$}.
\]
 Taking $\delta=1-m(\supp h)\geq 0$, we have that if $E\subset X$ with $m(E)>\delta$, then $m(\tilde{E})>0$ where $\tilde{E}=E\cap \supp h$. Hence, from the above inequality it holds that
\[
\inf_{n\geq 0} \int_{E} P^n1_X \, dm \geq \inf_{n\geq 0} \int_{\tilde{E}} P^n1_X \, dm  >0.
\]
Finally, since~(ii) implies (iii) immediately, we conclude the cycle of implications and thus the proof of the theorem.
\end{proof}

\section{Mean constrictivity
: proof of Theorem~\ref{thm:D}}\label{s:thmD}
In the sequel, we will prove the following result which is a more general version of Theorem~\ref{thm:D}.

\begin{thm}\label{thmeq}
 Let $(X,\mathscr{B},m)$ be an abstract probability space and consider a Markov operator $P: L^1(m) \to L^1(m)$. Then the following conditions are equivalent:
\begin{enumerate}[itemsep=0.2cm]
\item[\tt (MC)]\label{MC}
$P$ is \emph{mean constrictive};
\item[ \tt (MC2)]\label{MC2}
There is a compact set $F\subset L^1(m)$ and  $\kappa<1$ such that
\begin{align*}
\limsup_{n\to\infty}  d(A_n \varphi,F) \leq \kappa \quad \text{for any $\varphi\in D(m)$};
\end{align*}
\item[\tt (WMC)]\label{WMC}
$P$ is \emph{weakly mean constrictive}, i.e., there is a weakly compact set $F\subset L^1(m)$~such~that
\begin{align*}
\lim_{n\to\infty} d(A_n \varphi,F)=0 \quad \text{for any $\varphi\in D(m)$};
\end{align*}
\item[\tt (WMC2)]\label{WMC2}
There is a weakly compact set $F\subset L^1(m)$ and  $\kappa<1$ such that
\begin{align*}
\limsup_{n\to\infty} d(A_n \varphi,F) \leq \kappa \quad \text{for any $\varphi\in D(m)$};
\end{align*}
\item[\tt (MCDP)]\label{MC3}
$P$ is \emph{mean constrictive \`a la Dunford--Pettis}, i.e., for every $\varepsilon>0$, there exists $\delta>0$ such that for any $\varphi\in D(m)$, there is $n_0=n_0(\varepsilon,\varphi)\geq 1$ satisfying that
\begin{align*}
\int_E A_n \varphi\,dm <\varepsilon \quad \text{for any $n\geq n_0$ and $E\in\mathscr{B}$ with $m(E)<\delta$};
\end{align*}
\item[\tt (MCDP2)]\label{MC4}
There exist $\kappa<1$ and $\delta>0$ such that for any $\varphi\in D(m)$, there is $n_0=n_0(\varphi)\geq 1$ satisfying that
\begin{align*}
\int_E A_n \varphi\,dm \leq\kappa \quad \text{for any $n\geq n_0$ and $E\in\mathscr{B}$ with $m(E)<\delta$};
\end{align*}
\item[\tt (AMC)]\label{AMC} $P$ is \emph{asymptotically  mean constrictive}, i.e., for every $\varepsilon>0$, there is $\delta>0$~such~that
\begin{align*}
\limsup_{n\to\infty} \int_E A_n \varphi\,dm <\varepsilon \quad \text{for any $\varphi \in D(m)$ and $E \in \mathscr{B}$ with~$m(E)<\delta$};
\end{align*}
\item[\tt (AMC2)]\label{AMC2}
There exist $\kappa<1$ and  $\delta>0$ such that
\begin{align*}
\limsup_{n\to\infty} \int_E A_n \varphi\,dm <\kappa \quad \text{for any $\varphi \in D(m)$ and $E \in \mathscr{B}$ with~$m(E)<\delta$};
\end{align*}
\item[\tt (FED)]\label{FPM}
$P$ admits finitely many ergodic invariant densities $h_1,\dots,h_r$  with mutually disjoint supports (up to an $m$-null set) and the invariant density function $h=\frac{1}{r}(h_1+\dots+h_r)$ has the maximal support;
\item[\tt (APM)]\label{AD}
$P$ is \emph{asymptotically periodic in mean}, i.e., there exist finitely many ergodic invariant densities $h_1,\dots,h_r$ with mutually disjoint supports (up to an $m$-null set) and positive bounded linear functionals  $\lambda_1, \dots, \lambda_r$ on $L^1(m)$ such that
\begin{align*}
\lim_{n\to\infty} \bigg\lVert A_n\varphi-\sum_{i=1}^r\lambda_i(\varphi)h_i\bigg\rVert =0 \quad \text{for any $\varphi\in L^1(m)$}.
\end{align*}
\end{enumerate}
\end{thm}

\begin{figure}
\[
\xymatrix{ \tt
(MCDP) \ar@{=>}[rd]_{trivial} \ar@{=>}[dd]_{trivial}  &     &   &      &      &      &  &  \ar@{=>}[lllllll]_{Lemma~\tiny\ref{lemC1}} \tt (WMC) \\
         & \tt (ACM)  \ar@{=>}[d]_{trivial}  &  &
                                                                                              &      &      &  &     \\
\tt (MCDP2)  \ar@{=>}[r]_{trivial} &  \tt (AMC2) \ar@{=>}[rr]^{Theorem~\tiny\ref{propC2}} & & \tt (FED) \ar@{=>}[rr]_{Lemma~\tiny\ref{lemC3}} & & \tt (APM) \ar@{=>}[rr]^{Lemma~\tiny\ref{lemC4}} & & \ar@{=>}[uu]_{trivial} \tt (MC) \ar@{=>}[d]^{trivial} \\
\tt (WMC2) \ar@{=>}[u]^{Lemma~\tiny\ref{lemC1}} & &    &        &  &          &   &  \tt (MC2) \ar@{=>}[lllllll]^{trivial}
}
\]
\caption{Diagram of the relations between the conditions of Theorem~\ref{thmeq}.}
\label{fig2}
\end{figure}
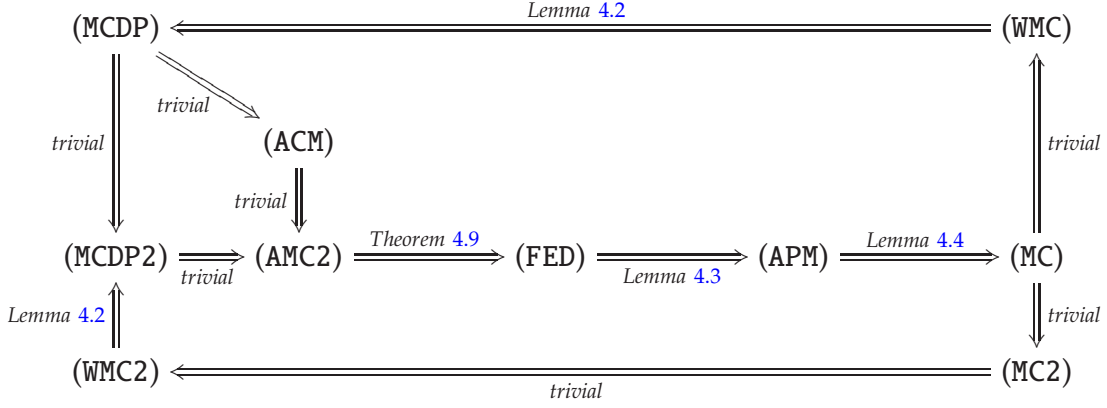

We will prove Theorem~\ref{thmeq} following the implications described in Figure~\ref{fig2}. To organize the proof, we divide the proof into two subsections.

\subsection{Proof of the   lemmas}
\begin{lem}\label{lemC1}
Condition {\tt(WMC)} (resp.~{\tt(WMC2)}) implies condition {\tt (MCDP)} (resp.~{\tt(MCDP2)}).
\end{lem}

\begin{proof}
Let $F$ be as in the condition {\tt(WMC)}.  Fix $\varepsilon>0$. The condition {\tt(WMC)} guarantees that for any $\varphi\in D(m)$ there exists $n_0=n_0(\varepsilon,\varphi)\geq 1$ such that
\begin{align*}
\inf_{\phi\in F}\lV A_n \varphi-\phi\rV <\frac{\varepsilon}{2} \quad \text{for any $n\geq n_0$}.
\end{align*}
Since $F$ is weakly compact, there exists $\delta=\delta(\varepsilon)>0$ such that for any $\phi\in F$ and $E\in \mathscr{B}$ with $m(E)<\delta$, we have $\int_E | \phi| \,  dm<\frac{\varepsilon}{2}$. Thus, for any set $E$ with $m(E)<\delta$
\begin{align*}
\int_EA_n \varphi\,dm&\le\inf_{\phi\in F}\lV A_n \varphi-\phi\rV+\int_E | \phi| \, dm <\varepsilon \quad \text{for any  $n\ge n_0$.}
\end{align*}
This proves {\tt (MCDP)}.

The implication {\tt(WMC2)} implies {\tt(MCDP2)} follows analogously. Namely, we put $\kappa$ in {\tt(MCDP2)} as the sum of $\kappa$ in {\tt(WMC2)} and $\varepsilon=\frac{1-\kappa}{2}$ from the weak compactness~of~$F$.
\end{proof}

\begin{lem}\label{lemC3}
The condition {\tt(FED)} implies the condition {\tt(APM)}.
\end{lem}

\begin{proof}
Suppose the condition {\tt(FED)}. In particular, the existence of an invariant density with the maximal support implies that $P$ is {\tt(WAP)}, cf.~\cite[Theorem~3.1]{Toyokawa2020}. Yosida and Kakutani proved in~\cite{yosida1941operator} that {\tt(WAP)} implies mean ergodicity.\footnote{The converse of this implication was recently proved in~\cite[Proposition~3.9]{Toyokawa2020}.} Recall that a Markov operator $P:L^1(m)\to L^1(m)$ is called mean ergodic if $A_n\varphi$ converges in $L^1(m)$-norm for any $\varphi\in L^1(m)$. For each density $h\in D(m)$, denote by $h^*$  the limit in $L^1(m)$-norm of $A_nh$. Observe that $h^*$ is an invariant density of $P$. Hence, from the finitude of ergodic invariant densities $h_1,\dots,h_r$, or equivalently extremal points in the space of invariant densities~(see Proposition~\ref{prop:ergodic} in Appendix~\ref{appendix:B}), we can find  $\lambda_1(h), \dots \lambda_r(h)$  with $\lambda_i(h)\ge 0$, $\lambda_1(h)+\dots+\lambda_r(h)=1$ such $h^*=\lambda_1(h) h_1 + \dots + \lambda_r(h) h_r$.
We have the equation in {\tt (APM)} for any $h\in D(m)$ and also for $\varphi\in L^1(m)$ with $\varphi\geq 0$.
Then, considering the positive part and negative part of any function in $L^1(m)$, we can extend the above argument to $L^1(m)$.
{Obviously} $\lambda_i:L^1(m)\to\mathbb{R}$ is bounded and linear so that the proof is done.
\end{proof}

\begin{lem}\label{lemC4}
The condition {\tt(APM)} implies the condition {\tt(MC)}.
\end{lem}

\begin{proof}
The proof is straightforward, taking the compact constrictor for $A_n$ by
\begin{align*}
F=\left\{ \varphi\in L^1(m): \ \varphi\text{ is a convex combination of }h_1,\dots,h_r \right\}.
\end{align*}
Therefore, the proof is done.
\end{proof}

\subsection{Proof of the theorem}

In this subsection, we will show that {\tt (AMC2)} implies~{\tt(FED)}.

\begin{rem} \label{prop:invariant} Observe that from Proposition~\ref{prop:wap}, the restriction $P_S$ of a Markov operator $P$ to the support $S$ of a $P$-invariant density is a Markov operator. Hence, Lemma~\ref{prop:P_h-Markov} in the Appendix~\ref{sec:Markov-restriction2} implies that   $P^*1_S\geq 1_S$.  
In particular, the sequence $(P^{n*}1_S)_{n\geq 1}$ is increasing.
\end{rem}

\begin{prop}\label{propC22}
{\tt (AMC2)}  implies that the set of ergodic invariant densities of $P$ is nonempty and finite. Moreover, these ergodic invariant densities have mutually disjoint supports.
\end{prop}
\begin{proof}
Observe first that by virtue of Theorem~\ref{cor:E} (6), 
it immediately follows from {\tt (AMC2)}  that there exists a $P$-invariant density.
Thus,
by Proposition~\ref{prop:ergodic},
\[
\mathcal{E}=\{g \in L^1(m): \  \text{$g$ is an ergodic $P$-invariant density} \} \not = \emptyset.
\]
The following claim is the key to obtaining the finitude of ergodic $P$-invariant densities.

\begin{claim}  \label{claim:Edelta}
  Consider $E\in \mathscr{B}$ such that $m(E)> 0$ and $P^*1_E\geq 1_E$. Let $\delta >0$ be the constant given   in {\tt (AMC2)}. Then, $m(E)\geq \delta$.
\end{claim}
\begin{proof}
Assume that $m(E)<\delta$.
Then, applying {\tt (AMC2)} to $h=\frac{1_E}{m(E)} \in D(m)$, we have
\[
\limsup_{n\to\infty}  \int_E A_n\frac{1_E}{m(E)} \, dm \leq \kappa <1.
\]
On the other hand, for any $n\geq 1$,
\begin{align*}
  \int_E A_n\frac{1_E}{m(E)} \, dm &= \frac{1}{m(E)} \int _EA^*_n 1_E \, dm = \frac{1}{m(E)} \int _E\frac{1}{n} \sum_{i=0}^{n-1} P^{i*}1_E \, dm
             \geq  \frac{1}{m(E)} \int _E 1_E \, dm =1.
\end{align*}
This is a contradiction, and thus, we have $m(E)\geq \delta$.
\end{proof}

Now,  by Remark~\ref{prop:invariant}, the support $S$ of any invariant density satisfies $P^*1_S\geq 1_S$ and by Proposition~\ref{prop:ergodic} and Remark~\ref{rem:disjoin-support} two distinct ergodic $P$-invariant densities have disjoint support.
Hence, it follows from Claim~\ref{claim:Edelta} that the cardinality of $\mathcal{E}$ is less than $\delta^{-1}$. Set $r$ to be this cardinality (notice that $r>0$ since $\mathcal{E}\not=\emptyset$). Thus, we get that $P$ admits finitely many ergodic invariant densities $h_1,\dots,h_r$ in $D(m)$ with mutually disjoint supports.
\end{proof}

\begin{prop} \label{propWAP}
  {\tt(WAP)} holds if {\tt(AMC2)} holds.
\end{prop}

We postpone the proof of the above proposition and show how to use it to conclude the following main result of this subsection:

\begin{thm}\label{propC2}
  {\tt (AMC2)} implies  {\tt (FED)}. In particular,  {\tt (AC)} implies {\tt (MC)}.
\end{thm}
\begin{proof}
Let $h_1,\dots,h_r$ be the ergodic invariant densities of $P$ obtained in Proposition~\ref{propC22}.
To conclude {\tt (FED)} we need to prove that the $P$-invariant density $h= \frac{1}{r} (h_1+\dots+h_r)$ has the maximal support.
According to Proposition~\ref{propWAP}, $P$ is {\tt(WAP)}. Then there is a $P$-invariant density $g$ with the maximal support, i.e., such that $\lim_{n\to\infty} P^{n*}1_{\supp g}=1_X$.
Since the set of $P$-invariant densities is convex and the ergodic densities are its extremal points, it follows that $g=\lambda_1 h_1 + \dots+ \lambda_r h_r$ with $\lambda_i \geq 0$ and $\lambda_1+\dots+\lambda_r=1$.
Then, $\supp g \subset \supp h$ and thus we also have $\lim_{n\to\infty} P^{n*}1_{\supp h}=1_X$. This shows that $h$ has the maximal support concluding the proof.

The final implication follows immediately from the first part of the theorem since, trivially, {\tt(AC)} $\Rightarrow$ {\tt (AMC2)} and {\tt (FED)} $\Rightarrow$ {\tt(MC)} by Lemmas~\ref{lemC3} and~\ref{lemC4}.
\end{proof}

\subsubsection{Proof of Proposition~\ref{propWAP}}

Let $h$ be a $P$-invariant density. Set $S=\supp h$. Define
\[
\varphi\eqdef \lim_{n\to\infty} P^{n*}1_{S}.
\]
This limit exists since the sequence $(P^{n*}1_{S})_{n\geq 1}$ is increasing. See Remark~\ref{prop:invariant}.
The first observation is that $0\leq \varphi \leq 1$. Also, by the monotone continuity property of $P^*$ (see~\cite[Proposition V.4.1]{neveu1965mathematical}),
   $P^*\varphi =\varphi$.
Moreover, since
$P^*1_{S} \geq 1_{S}$, one has that $\varphi=1$ on $S$.
Let us split
\[
X=E_0 \cup E \cup E_1, \quad E=\{x\in X: 0<\varphi<1\} \ \ \text{and} \ \ E_i=\{x\in X: \varphi=i\} \ \text{for $i=0,1$}.
\]

\begin{lem} \label{lem1}
  We have $P^*1_E\leq 1_E$ and $P^*1_{E_0} \geq 1_{E_0}$.
\end{lem}
\begin{proof}
  Note that $\supp 1_E \subset \supp \varphi$  and  $\supp 1_E \subset \supp (1_X-\varphi)$. Then, by Lemma~\ref{supp}, it follows that
  \[
  \supp P^*1_E \subset \supp P^*\varphi \cap \supp P^*(1_X-\varphi)=\supp \varphi \cap \supp (1_X-\varphi)=E.
  \]
  Then, Proposition~\ref{prop:P_h-Markov} in Appendix~\ref{sec:Markov-restriction2} implies that $P^*1_{E} \leq 1_{E}$.
On the other hand, since $\lim_{n\to\infty} \min\{1_X,n\varphi\}=1_{\supp \varphi}$, it holds that
\[
P^*1_{\supp \varphi} =\lim_{n\to\infty} P^*\min\{1_X,n\varphi\}   \leq \lim_{n\to\infty} \min\{1_X,n\varphi\}=1_{\supp \varphi}.
\]
Equivalently,  $P^*1_{E_0} \geq 1_{E_0}$ since $E_0=X\setminus \supp \varphi$.
\end{proof}

\begin{lem} \label{lem:anterior}
For every $A\subset \supp \varphi \setminus S$, $\lim_{n\to\infty} P^{n*}(\varphi 1_A)=0$.
\end{lem}
\begin{proof} By assumption, $1_A \leq  1_{\supp \varphi \setminus S}$. Hence,  $0\leq \varphi 1_A \leq \varphi 1_{\supp \varphi \setminus S}$.
Then, since
\[
\lim_{n\to\infty} P^{n*}(\varphi 1_{\supp \varphi \setminus S}) =
\lim_{n\to\infty} P^{n*}(\varphi 1_{\supp \varphi} - \varphi 1_S)=
\lim_{n\to\infty} P^{n*}\varphi - \lim_{n\to\infty} P^{n*}1_{S}= \varphi-\varphi=0,
\]
it also follows that $\lim_{n\to\infty} P^{n*}(\varphi 1_A)=0$.
\end{proof}

Let us define
\[
\psi\eqdef\lim_{n\to\infty} P^{n*}1_{E}.
\]
As above, notice that this limit exists since, from Lemma~\ref{lem1}, the sequence $(P^{n*}1_{E})_{n\geq 1}$ is decreasing.
Moreover, again, by the monotone continuity property of $P^*$,  we have that~$P^*\psi=\psi$ and $\supp \psi \subset E$.

\begin{lem}  \label{lem:1-phi}
  If $m(E_0)=0$, then $\varphi = 1_X -\psi$.
\end{lem}
\begin{proof}
Since $m(E_0)=0$,
\begin{equation*}\label{eq:lim}
  1_X=\lim_{n\to\infty} P^{n*}1_X= \lim_{n\to\infty} P^{n*}(1_{S}+1_{E_1\setminus S}+1_E) = \varphi +\lim_{n\to\infty} P^{n*}1_{E_1\setminus S} + \psi.
\end{equation*}
Having it in mind that $1_{E_1\setminus S}=\varphi 1_{E_1\setminus S}$ and $ E_1\setminus S \subset \supp \varphi \setminus S$, Lemma~\ref{lem:anterior} implies that $\lim_{n\to\infty} P^{n*}1_{E_1\setminus S}=0$.
Substituting above, we get $1_X=\varphi + \psi$ concluding the proof.
\end{proof}

For each $\varepsilon>0$, let us denote
$
   B_{\varepsilon}=\{x\in X:  1-\varepsilon \leq \psi(x)\}.
$

\begin{lem}  \label{lem:epsilon}
  If $\varphi=1_X - \psi$, then  $\lim_{n\to \infty} P^{n*}(\psi 1_{B_\varepsilon})=\psi$ for every $\varepsilon>0$.
  Moreover, if $m(B_{\varepsilon})=0$ for some $\varepsilon>0$, then $\psi=0$ (and hence, $\varphi=1$) $m$-almost everywhere.
\end{lem}
\begin{proof}
For each $k\geq 1$, define
$C_k=\{x\in X: \, \frac{1}{k+1}\leq \varphi(x) < \frac{1}{k}\}$.
Since $\varphi=1_X - \psi$, we have
\[
  C_k = \left\{ x\in X: \, 1-\frac{1}{k}< \psi(x) \leq 1-\frac{1}{k+1}\right\}.
\]
Hence,
\[
1_{C_k}\psi
\leq 1_{C_k}\left(1-\frac{1}{k+1}\right)
=1_{C_k}\cdot\frac{k}{k+1}
\leq k\, 1_{C_k}\varphi.
\]
Moreover, from Lemma~\ref{lem:anterior}, since $C_k \subset E \subset \supp \varphi \setminus S$, we have   that
\[
 0\leq \lim_{n\to\infty}P^{n*}\left(1_{C_k}\psi\right)
\le k\lim_{n\to\infty}P^{n*}\left(1_{C_k}\varphi\right)=0
\]
and thus, for each $k\geq 1$,
\begin{equation}\label{eq:lim=0}
  \lim_{n\to\infty}P^{n*}\left(1_{C_k}\psi\right)=0.
\end{equation}
On the other hand, for each $\varepsilon>0$, there exists $k_0\geq 1$ such that
\[
E\setminus B_\varepsilon =\{x\in X: 0<\psi(x)<1-\varepsilon\} \subset \bigcup_{k=1}^{k_0} C_k.
\]
Then, from~\eqref{eq:lim=0} we get
\begin{equation}\label{eq24}
  \lim_{n\to\infty}P^{n*}\left(1_{E\setminus B_\varepsilon}\psi\right)=0 \quad \text{for all $\varepsilon>0$.}
\end{equation}
Now, observe that  $\psi=\psi1_E=\psi 1_{B_\varepsilon}+ \psi 1_{E\setminus B_\varepsilon}$. Hence, by~\eqref{eq24}, we conclude that
\begin{align*}
\lim_{n\to\infty}P^{n*}\left(\psi1_{B_{\varepsilon}}\right)=\lim_{n\to\infty} P^{n*}\psi = \psi.
\end{align*}

Finally, note that if $m(B_{\varepsilon})=0$ for some $\varepsilon>0$  then $\psi<1-\varepsilon$ $m$-almost everywhere.
From this, it follows that
\begin{align*}
\psi=P^{n*}\psi=P^{n*}(\psi 1_E)< (1-\varepsilon)P^{n*}1_E\to(1-\varepsilon)\psi\quad \text{as $n\to\infty$.}
\end{align*}
Then $\psi=0$ $m$-almost everywhere.
\end{proof}

Now we set
\[
\mathcal{F}=\{S \subset X :  S \ \ \text{is the support of a} \  P\text{-invariant density} \}.
\]
Then, according to Proposition~\ref{propC22} there is at least one $P$-invariant density and thus $\mathcal{F}\not=\emptyset$.
The inclusion up to an $m$-null set induces a partial order in $\mathcal{F}$. Given a chain, by the Zorn lemma, there exists a maximal element  $S\in \mathcal{F}$ of the chain in the sense that if $S \subset S'$ and $S'\in \mathcal{F}$ then $S'=S$ up to an $m$-null set. Let $h\in D(m)$ be a $P$-invariant density such that $S=\supp h$ and consider for this density the sets $E_0$, $E$ and $E_1$ defined above.

\begin{lem} \label{lem:MCDP2}
 {\tt(AMC2)} implies that  $m(E_0)=0$ and  $m(B_{\varepsilon})=0$ for some $\varepsilon>0$.
\end{lem}
\begin{proof}
By Lemma~\ref{lem1}, $P^*1_{E_0}\geq 1_{E_0}$. If $m(E_0)>0$, then Proposition~\ref{prop:P_h-Markov} implies that $P_{E_0}:L^1(m_{E_0})\to L^1(m_{E_0})$ is a Markov operator.
Therefore, by virtue of Theorem~\ref{cor:E}~(6), it immediately follows from {\tt (AMC2)} that there exists a $P_{E_0}$-invariant density $g\in D(m_{E_0})$.
In particular, it satisfies that $P(1_{E_0}g)=1_{E_0}g$ and $\int 1_{E_0} g \, dm = m(E_0)$.
Hence, $\phi= \frac{1}{2}(h+\frac{1_{E_0}g}{m(E_0)})$ is a $P$-invariant density and thus $S'=\supp \phi \in \mathcal{F}$ and $S=\supp h \subsetneq S'$. This contradicts the maximality of $S$, concluding   that $m(E_0)=0$.

Next we will see that there exists $\varepsilon>0$ such that $m(B_\varepsilon)=0$.
From the assumption {\tt (AMC2)}, there are  $0<\kappa<1$ and $\delta>0$ so that
\begin{align}\label{eq26}
\limsup _{n\to\infty} \int_B A_n\phi \, dm \le\kappa \quad \text{for each $\phi \in D(m)$ and $B\in\mathscr{B}$ with $m(B)<\delta$.}
\end{align}
On the other hand, one can find $\varepsilon_0>0$ such that for any $0<\varepsilon\le\varepsilon_0$ it holds $m(B_{\varepsilon})<\delta$ since $m(B_\varepsilon)\to0$ as $\varepsilon\to0$.
We fix $0<\varepsilon<\min\{\varepsilon_0,1-\kappa\}$. We will see that $m(B_{\varepsilon})=0$.
Arguing by contradiction, we assume that $m(B_\varepsilon)>0$.
Then for a function $\phi =\frac{1_{B_\varepsilon}}{m(B_\varepsilon)}$, it follows from~\eqref{eq26} that
\begin{align*}
\kappa&\ge\limsup _{n\to\infty}\int_{B_\varepsilon}A_n\phi\,dm\\
&=\limsup _{n\to\infty}\int_X \phi\cdot A_{n}^*1_{B_\varepsilon}\,dm
\ge\limsup _{n\to\infty}\int_X \phi\cdot A_{n}^*(\psi1_{B_\varepsilon})\,dm
=\int_X \phi \psi \,dm
\end{align*}
by the Lebesgue dominated convergence theorem and Lemma~\ref{lem:epsilon} (see $\varphi=1_X-\psi$ due to $m(E_0)=0$ and Lemma~\ref{lem:1-phi}). But we also have
\begin{align*}
\int_X \phi \psi \, dm=
\frac{1}{m\left(B_\varepsilon\right)}\int_{B_\varepsilon}\psi \,dm
\ge\frac{1}{m\left(B_\varepsilon\right)}\int_{B_\varepsilon}(1-\varepsilon)\,dm
>\kappa.
\end{align*}
Therefore, we conclude $m(B_{\varepsilon})=0$.
\end{proof}

\begin{proof}[Proof of Proposition~\ref{propWAP}]
The result
follows from Lemmas~\ref{lem:1-phi},~\ref{lem:epsilon}~and~\ref{lem:MCDP2}.
\end{proof}

\subsection{On the classes {\tt (AC)},  {\tt(WC)} and {\tt (APW)} }

In Theorem~\ref{propC2} we have proved that {\tt(AC)} $\Rightarrow$ {\tt (MC)} and, in particular from Theorem~\ref{thm:D}, {\tt(AC)} implies {\tt (APM)}.
The following proposition concludes that conditions {\tt (APW)} and {\tt (WC)} introduced in Section~\ref{sec:Questions-(AC)-(UC)} are sufficient for {\tt(AC)}. Thus, as a consequence,  {\tt(APW)} $\Rightarrow$ {\tt(APM)}  and \mbox{{\tt (WC)} $\Rightarrow$ {\tt(MC)}.}

\begin{prop}\label{prop:efinal} It holds that
{\tt (WC)} implies {\tt (AC)}. Furthermore, {\tt (APW)} implies {\tt (WC)}.
\end{prop}
\begin{proof}
Assume that {\tt (WC)} holds. Hence,  given $h\in D(m)$, there is $(\psi _n)_{n\geq 1} \subset F$ such that $P^nh -\psi _n\to 0$ weakly in $L^1(m)$ as $n\to \infty$.
Therefore,
\begin{align*}
\limsup _{n\to\infty}\int _AP^n h\, dm &= \limsup _{n\to\infty}\left(\int _A\psi _n \, dm +\int _A(P^n h - \psi _n) \, dm\right) \\
&= \limsup _{n\to\infty} \int _A\psi _n \, dm \leq \sup _{\psi \in F} \int _A \psi \, dm.
\end{align*}
Since $F$ is a weakly compact set, it follows from the Dunford--Pettis theorem that for any $\varepsilon >0$, there is $\delta >0$ such  that $ \sup _{\psi \in F} \int _A \psi \, dm<\varepsilon$ for any $A\in\mathscr{B}$ with $m(A)<\delta$.
This concludes that $P$ satisfies {\tt (AC)}.

Next we assume that {\tt (APW)} holds. Then, for $h \in D(m)$, since $P$ is a Markov operator, due to the weak convergence in {\tt(APW)},
\begin{align} \label{eq:convex-combination}
0 &=\lim _{n\to\infty} \int P^n\big(h -\sum _{i=1}^r\lambda _i(h )g_i\big) \, dm
=\int \big(\varphi -\sum _{i=1}^r\lambda _i(h )g_i \big) \, dm= 1- \sum _{i=1}^r\lambda _i(h ).
\end{align}
Set
\[
F\coloneqq \left\{\sum_{i=1}^r a_i \, g_i : \ a_i\in \mathbb R,\; \sum _{i=1}^r a_i =1\right\} \subset D(m).
\]
Notice that  $F$ is a weakly compact set and also $P$-invariant since $Pg_i=g_{\rho(i)}$ where $\rho$ is the permutation of $\{1,\dots,r\}$ in the definition of {\tt(APW)}. Moreover, in view of~\eqref{eq:convex-combination}, the weak convergence in {\tt (APW)} means that for any $h\in D(m)$, there exists
\[
\psi _n\coloneqq \sum _{i=1}^r \lambda _i(h)P^ng_i = \sum _{i=1}^r \lambda _i(h)g_{\rho^n(i)} \in F
\]
such that
$P^nh -\psi _n \to 0$  weakly as $n\to \infty$. This concludes that $P$ satisfies {\tt (WC)}.
\end{proof}

\section{Characterization of finitude of physical measures: proof of Theorem~\ref{thm:C}}\label{s:thmc}
Let $(X ,  \mathscr B,  m)$ and  $(\Omega ,  \mathscr F,  \mathbb{P})=(T^\mathbb{N},  \mathscr A^\mathbb{N},  p^\mathbb{N})$ be a Polish probability space and the infinite product space of a probability space $(T,  \mathscr A,  p)$, respectively. Consider a measurable map $f: T\times  X \to X$  and
\[
f^0_\omega=\mathrm{id} \quad \text{and}  \quad f^n_\omega=
f_{\omega_{n-1}}\circ \dots \circ f_{\omega_0} \
\ \text{for}
\  n>0 \ \text{and   $\omega =(\omega _0, \omega _1, \ldots )\in \Omega$},
\]
where we denoted $f_t=f(t,\cdot)$ for $t\in T$.
Recall the notion of stationary measure of $f$ in~\eqref{def:stationary}.
Also, recall that the convergence of measures that we are considering is in the weak* topology. That is, $\mu_n \to \mu$ if and only if
\begin{equation} \label{eq:weak*convergence}
   \int \varphi \,d\mu_n \to \int \varphi \, d\mu \quad \text{for all $\varphi\in C_b(X)$}
\end{equation}
where $C_b(X)$ denotes the set of bounded real-valued continuous functions of $X$.
Since $X$ is Polish space, it is actually a complete separable metric space and thus the following result is applied. This statement can be found in~\cite[Proposition 2.2]{worm2011ergodic} (cf.~\cite[Proof of Proposition 4.4 in Chapter 3]{ethier1986markov}) but we include its proof for convenience of the reader.
\begin{prop} \label{prop:countable-determined-weak}
Let $(X, d)$ be a separable metric space. Then, the convergence in the weak* topology in the space $\mathcal{P}(X)$ of Borel probability measures of $X$ is countably determined, i.e., there exists a countable set $S$ in $C_b(X)$ consisting of bounded Lipschitz functions such that if $\mu$ and $(\mu_n)_{n\geq 1}$ belong to  $\mathcal{P}(X)$  and satisfy that
\begin{equation} \label{eq:4.4}
\int \phi \, d\mu_n \to  \int \phi \, d\mu \quad  \quad \text{for all} \ \phi \in S,
\end{equation}
then $\mu_n \to \mu$ in the weak* topology.  Consequently, if $\nu$, $\mu \in \mathcal{P}(X)$ satisfy
$\int \phi \,d\nu =  \int \phi \, d\mu$ for every $\phi \in S$,
then $\nu = \mu$.
\end{prop}
\begin{proof}
Let $\{x_i\}_{i\geq 1}$ be a dense set of $X$ and define $\phi_{ij}\in C_b(X)$ for $i,j \geq 1$ as
\begin{equation*}
  \phi_{ij}(x)=\max\{2(1-j\,d(x,x_i),0\}.
\end{equation*}
Clearly, $\phi_{ij}$ are bounded Lipschitz functions. Let $S$ be the countable set of such functions. Now,
given any open  set $U$ of $X$, consider the functions $g_m(x)= \inf\{ \sum \phi_{ij}(x), 1\}$ for $m\geq 1$, where the sum extends over those $i,j\leq m$ such that $B(x_i,\frac{1}{j})\subset U$. Since $g_m$ are finite sums of $\phi_{ij}$,
if~\eqref{eq:4.4} holds, then
\[
\liminf_{n\to\infty} \mu_n(U)\geq \lim_{n\to\infty} \int g_m \, d\mu_n = \int g_m \, d\mu \quad \text{for $m\geq 1$.}
\]
Hence, letting $m\to \infty$ we conclude that $\liminf_{n\to\infty} \mu_n(U) \geq \mu(U)$.  By Portemanteau's theorem~(cf.~\cite[Theorem 3.1 in Chapter 3]{ethier1986markov}), this inequality is equivalent to the convergence of $\mu_n$ to $\mu$ in the weak* topology concluding the proof.
\end{proof}

As an application of the above proposition we get the following:

\begin{lem} \label{lem-mu1}
 Let $\mu$ be an ergodic stationary measure of $f$. Then,
 \begin{enumerate}[label=(\roman*)]
  \item $\bar{\mu}(B(\mu))=1$ \quad where $\bar{\mu}=\mathbb{P}\times\mu$;
 \item $\mu(B_\omega(\mu))=1$ \quad for $\mathbb{P}$-almost every $\omega\in \Omega$;
 \item $\mathbb{P}(B_x(\mu))=1$ \quad for $\mu$-almost every $x\in X$.
 \end{enumerate}
Here,
 \begin{equation} \label{eq:movida}
 B_\omega(\mu)=\big\{x\in X: \, (\omega,x)\in B(\mu)\big\} \ \quad \text{and} \quad \  B_x(\mu)=\big\{\omega\in \Omega: \, (\omega,x)\in B(\mu)\big\}
\end{equation}
and
\[
B(\mu)=\left\{ (\omega,x)\in \Omega\times X: \, \lim _{n\to \infty} \frac{1}{n} \sum _{j=0}^{n-1} \delta_{ f_{ \omega}^{j}(x)}=\mu\right\}.
\]
\end{lem}

\begin{proof}
Since $\mu$ is an ergodic stationary measure of $f$, the measure $\bar{\mu}=\mathbb{P}\times \mu$ is an ergodic invariant probability measure of the corresponding skew-product transformation $F$.
Now, for a fixed continuous function $\psi:X\to \mathbb{R}$, we define the set $B(\mu,\psi)$ of points $(\omega,x)\in \Omega\times X$  for which
\[
\lim _{n\to \infty} \frac{1}{n} \sum _{j=0}^{n-1} \psi \circ f_{ \omega}^{j}(x) = \int _X \psi\, d\mu.
\]
Taking into account that
\[
\psi \circ f_{ \omega}^{n}(x) =  \bar{\psi} \circ F^{n}(\omega,x) \quad \text{and} \quad \int \psi \, d\mu = \int \bar{\psi} \, d\bar{\mu}
\]
where
\[
\bar{\psi}(\omega,x)=\psi(x) \quad \text{for } (\omega,x)\in \Omega\times X,
\]
it follows from the Birkhoff ergodic theorem that $B(\mu,\psi)$ has full $\bar{\mu}$-measure.  Now, since $X$ is a Polish space, by Proposition~\ref{prop:countable-determined-weak}
it is not difficult to see that
\[
B({\mu})=\bigcap_{\varphi\in S} B(\mu,\psi)
\]
and therefore $B(\mu)$ has also full $\bar{\mu}$-measure.
Finally, we observe that $B_\omega(\mu)$ and $B_x(\mu)$ are the $\omega$-section and $x$-section of $B(\mu)$ respectively, i.e., it holds~\eqref{eq:movida}.
Hence,
by Fubini's theorem, we have that $\mu(B_\omega(\mu))=1$ and $\mathbb{P}(B_x(\mu))=1$ for $\mathbb{P}$-almost every $\omega\in \Omega$ and $\mu$-almost every $x\in X$ respectively.
\end{proof}

We highlight the following lemma for future reference. This lemma follows immediately by the Fubini theorem and the definition of $B_\omega(\mu)$ and $B_x(\mu)$ in~\eqref{eq:movida}.

\begin{lem} \label{lem:equi-basin}
 The following are equivalent:
\begin{enumerate}
\item $\bar{m}(B(\mu_1)\cup \dots \cup B(\mu_r))=1$  \quad where $\bar{m}=\mathbb{P}\times m$;
\item $m(B_\omega(\mu_1)\cup \dots \cup B_\omega(\mu_r))=1$ \quad  for $\mathbb{P}$-almost every $\omega\in \Omega$;
\item $\mathbb{P}(B_x(\mu_1)\cup \dots \cup B_x(\mu_r))=1$ \quad for $m$-almost every $x\in X$.
\end{enumerate}
\end{lem}

\begin{rem}\label{rem:physical}
By the Birkhoff ergodic theorem and the Fubini theorem, we immediately find that Lemmas~\ref{lem-mu1} and~\ref{lem:equi-basin}  hold for
\begin{equation} \label{eq:Bpsi}
\begin{gathered}
B(\mu,\psi)=\left\{ (\omega,x)\in \Omega\times X : \, \lim _{n\to \infty} \frac{1}{n} \sum _{j=0}^{n-1} \psi (f^j_\omega (x)) = \int \psi \, d\mu  \right\},  \\
B_\omega(\mu, \psi)=\left\{x\in X:  (\omega,x)\in B(\mu,\psi)\right\} \ \ \text{and} \ \  B_x(\mu,\psi)= \left\{\omega\in \Omega:  (\omega,x)\in B(\mu,\psi)\right\}
\end{gathered}
\end{equation}
with bounded real-valued functions $\psi$,
instead of $B(\mu)$, $B_\omega(\mu)$ and $B_x(\mu)$ respectively.
\end{rem} 
\begin{rem}\label{rem:physical2}
The equivalence in Lemma~\ref{lem:equi-basin} remains valid when we replace "$=1$" with "$>0$". This leads us to the following definition: an ergodic stationary probability measure $\mu$ for $f$ is called a \emph{physical measure} if one of the following equivalent conditions holds:
  \begin{enumerate}
    \item $\bar{m}(B(\mu))>0$, where $\bar{m}=\mathbb{P}\times m$;
    \item $m(B_\omega(\mu))>0$ for $\mathbb{P}$-almost every $\omega \in \Omega$;
    \item $\mathbb{P}(B_x(\mu))>0$ for $m$-almost every $x \in X$.
  \end{enumerate}
This definition aligns with the classical notion of physical measure in the deterministic case (that is, when $\Omega$ is a singleton). However, it differs from previous definitions of physical measure for random maps introduced in the literature (see for instance~\cite[Equation~(2)]{alves2007stochastic} or~\cite[page~313]{Araujo2000}).
This concept was previously introduced by requiring $m(\mathfrak{B}(\mu))>0$, where
\[
 \mathfrak{B}(\mu)=\left\{x\in  X: \, \lim _{n\to \infty} \frac{1}{n} \sum _{j=0}^{n-1} \delta_{f_{ \omega}^{j}(x)}=\mu \quad \text{for $\mathbb{P}$-almost every $\omega\in \Omega$} \right\}.
\]
Since $\bar{m}(B(\mu))\geq m(\mathfrak{B}(\mu))$, we conclude that $m(\mathfrak{B}(\mu))>0$ implies $\bar{m}(B(\mu))>0$. However, the converse is not expected to hold.
Moreover, defining $\mathfrak{B}(\mu)$ as the statistical basin of attraction does not seem appropriate, as the following example illustrates.
  \begin{example} \label{example:figura4}
  Let $X=T=[-1,1]$ be equipped with the normalized Lebesgue measure and let $X_-=[-1,0)$,  $X_+=(0,1]$.
 Consider a continuous map $f_0: X\to X$ which has exactly two sinks $p_- = -\frac{1}{2}$, $p_+ = \frac{1}{2}$ whose basin of attraction is $X_-$, $X_+$, respectively, and $0$ is the fixed point of $f_0$.
 With the notation $B(x,r)$ for the ball of radius $r$ centered at $x$,
 we also assume that $f_0(X )\subset B(0,\frac{3}{4})$ and $f_0( B(p_\pm , \frac{1}{4}))\subset  B(p_\pm , \frac{1}{8})$.
 Then, for the random map $f: T\times X\to X$ with additive noise given by $f_t(x) = f_0(x) + \frac{1}{8} t$,  it holds that $f_t(B(p_\pm , \frac{1}{4}))\subset  B(p_\pm , \frac{1}{4})$  for any $t\in T$. See Figure~\ref{fig_eg_Y}.

\begin{figure}[h]
\centering
\begin{subfigure}{0.45\textwidth}
    \includegraphics[width=\textwidth]{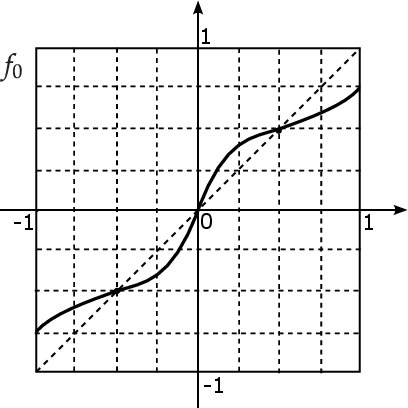}
    \caption{The map $f_0$ on $[-1,1]$}
    \label{fig:first}
\end{subfigure}
\hfill
\begin{subfigure}{0.45\textwidth}
    \includegraphics[width=\textwidth]{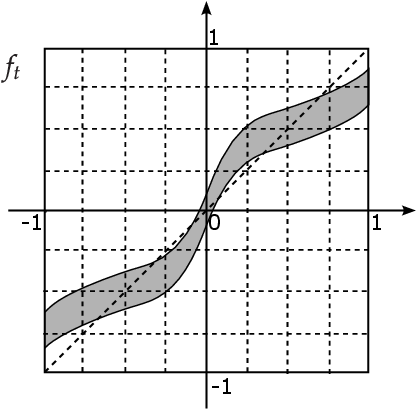}
    \caption{The map $f_t$ on $[-1,1]$}
    \label{fig:second}
\end{subfigure}
\caption{Illustrations of the random map $f$ in Example~\ref{example:figura4}.}
\label{fig_eg_Y}
\end{figure}

  By the argument in Remark~\ref{rm:0302}, the annealed Perron--Frobenius operators  associated with the restrictions of $f$ on both $B(p_- , \frac{1}{4})$  and $B(p_+ , \frac{1}{4})$
    satisfy {\tt (FPM)}.
Therefore, $f$ admits at least two absolutely continuous ergodic stationary measures $\mu _-$, $\mu _+$ whose supports are included in $B(p_- , \frac{1}{4})$, $B(p_+ , \frac{1}{4})$, respectively.
Since the annealed Perron--Frobenius operator associated with $f$ itself also satisfies {\tt (FPM)}, $f$ admits finitely many absolutely continuous ergodic stationary measures $\mu _1,\ldots ,\mu _r$, two of which are $\mu _-$, $\mu _+$,  such that
\[
m( B_\omega (\mu _1) \cup \cdots \cup B_\omega (\mu _r)) =1
\quad \text{for $\mathbb P$-almost every $\omega$.}
\]
On the other hand, by the continuity of $f_0$, there is a neighborhood  $U$ of $0$ such that $f_0(U) \subset B(0,\frac{1}{16})$, so that for each $x\in U$, both $\{ f_t(x) : t\in T\} \cap X_+$ and $\{ f_t(x) : t\in T\} \cap X_-$ have positive Lebesgue measure.
Consequently,  one can find positive $\mathbb P$-measure sets $\Gamma_-$, $\Gamma_+$ and $n_0\geq 1$ such that  if $\omega \in \Gamma_\pm$  then  $f_\omega ^n(x) \in B(p_\pm , \frac{1}{4})$  for any  $n\geq n_0$ and $x\in U$.
This concludes
 that $U\not \subset \bigcup _{j=1}^r \mathfrak{B} (\mu _j) $.
Therefore,
\[
0<m(\mathfrak{B}(\mu_1)\cup \dots \cup \mathfrak{B}(\mu_r))<1.
\]
In conclusion,  if one expects finitely many physical measures where the union of their basins of attraction covers the whole space almost everywhere, then the good notion of the basin should be the fiberwise statistical basin.
  \end{example}
\end{rem}

Let us show that (ii) implies (i) in Theorem~\ref{thm:C}. To do this, according to the equivalences shown in Theorem~\ref{thm:D}, it suffices to show the following proposition.

\begin{prop} \label{prop:FPM-FED}
If $f$ satisfies {\tt(FPM)}, then  ${\mathcal{L}}_f$ satisfies the condition {\tt(FED)}.
\end{prop}
\begin{proof}
Since $f$ satisfies {\tt(FPM)} we have only finitely many absolutely continuous (with respect to $m$) ergodic stationary probability measures
$\mu_1,\dots, \mu_r$ which have pairwise disjoint supports and $m(B_\omega(\mu_1,\psi)\cup\dots \cup B_\omega(\mu_r, \psi))=1$ for $\mathbb{P}$-almost every $\omega\in \Omega$ and any bounded real-valued function $\psi$. Let $h_i$ be the Radon--Nikod\'{y}m derivative of the measure $\mu_i$ with respect to $m$ for $i=1,\dots,r$.
Observe that $h_i$ is an invariant density of ${\mathcal{L}}_f$ because $\mu_i$ is a stationary probability measure of $f$.
Moreover, since the supports of such measures are pairwise disjoints, we also get that the densities  $h_1,\dots,h_r$ have mutually disjoint supports.
To conclude {\tt(FED)} for ${\mathcal{L}}_f$ we need to prove that the density $h=\frac{1}{r} (h_1+\dots+h_r)$ has the maximal support. This means that
\begin{equation} \label{eq:maximal}
  \lim_{n\to\infty} {\mathcal{L}}_f^{n*}1_{\supp h}(x) = 1 \quad \text{for $m$-almost every $x\in X$}
\end{equation}
where ${\mathcal{L}}_f^{\,*}$ is the adjoint operator of ${\mathcal{L}}_f$.

From 3) in Definition~\ref{dfn:11} of {\tt (FPM)}, we also have that for $\psi=1_{\supp h}$ and for $(\mathbb{P}\times m)$-almost every $(\omega, x)\in \Omega\times X$, there is $i=i(\omega,x)\in\{1,\cdots, r\}$ such that
\[
 \lim_{n\to \infty} \frac{1}{n}\sum_{j=0}^{n-1}  1_{\supp h} \circ f^j_\omega(x) =  \mu_i({\supp h})=1.
\]
That is, for $m$-almost every $x\in X$, there is $\Omega_1=\Omega_1(x)\subset\Omega$ with $\mathbb{P}(\Omega_1)=1$ such that
\begin{equation}\label{eq:lim1}
 \lim_{n\to \infty} \frac{1}{n}\sum_{j=0}^{n-1}  1_{\supp h} \circ f^j_\omega(x) =1 \quad \text{for all $\omega \in \Omega_1$.}
\end{equation}
Using the notation introduced in Appendix~\ref{sec:apendix}, and by applying Lemma~\ref{lem:A2}~(4), the Lebesgue dominated convergence theorem, Lemma~\ref{lem:A2}~(2), and Equation~\eqref{eq:lim1}, we deduce that for $m$-almost every $x\in X$,
\begin{align*}
  \lim_{n\to \infty} \frac{1}{n}\sum_{j=0}^{n-1}  {\mathcal{L}}_f^{j*}1_{\supp h}(x)  &=
  \lim_{n\to \infty} \frac{1}{n}\sum_{j=0}^{n-1}  \int \mathcal{L}_\omega^{j*} 1_{\supp h}(x) \, d\mathbb{P}(\omega) \\
  &=\int_{\Omega_1} \,   \lim_{n\to \infty} \frac{1}{n}\sum_{j=0}^{n-1}  1_{\supp h} \circ f^j_\omega(x)\, d\mathbb{P}(\omega)=1.
\end{align*}
Furthermore, according to Remark~\ref{prop:invariant}, we observe that $({\mathcal{L}}_f^{j*}1_{\supp h})_{j\geq 1}$ is monotonically increasing. Therefore, the convergence of the mean implies the convergence of the sequence (this follows as a consequence of the Stolz--Ces\`aro theorem). In other words, we obtain~\eqref{eq:maximal}, thereby concluding the proof.
\end{proof}

Now we will prove that (i) implies (ii). According to Theorem~\ref{thm:D}, if ${\mathcal{L}}_f$ is mean constrictive, then ${\mathcal{L}}_f$ satisfies {\tt (APM)}. That is, ${\mathcal{L}}_f$ admits finitely many ergodic invariant densities $h_1,\dots,h_r$ in $D(m)$ with mutually disjoint supports and $\lambda_1, \dots, \lambda_r$ bounded linear functionals such that for any $\varphi\in L^1(m)$,
\begin{align}\label{eq:AD}
\lim_{n\to\infty} \lV A_n\varphi-\sum_{i=1}^r\lambda_i(\varphi)h_i\rV =0.
\end{align}
Since $h_i$ is an ergodic invariant density for ${\mathcal{L}}_f$, the measure $\mu_i$ is an absolutely continuous (with respect to $m$) ergodic stationary measures of $f$ where $d\mu_i=h_i \,dm$.
Moreover,  since the support $\supp \nu$ of a measure $\nu$ absolutely continuous with respect to $m$ is defined as the support of the  Radon--Nikod\'{y}m derivative $\frac{d\nu}{dm}$  of $\nu$ with respect to $m$, we also obtain that $\supp \mu_1,\dots,\supp \mu_r$ are pairwise disjoints. To prove that $f$ satisfies {\tt(FPM)} it remains to show that the union of basins of attraction of these measures with respect to any bounded observable $\psi$ is $X$ modulus a set of zero $m$-measure. This will be obtained in the following proposition.

\begin{prop} \label{prop:APM-FPM}
  Assume that ${\mathcal{L}}_f$ satisfies {\tt(APM)} as above. Then, for any bounded real-valued function $\psi$,
\[
  m(B_\omega(\mu_1, \psi)\cup \dots \cup B_\omega(\mu_r, \psi))=1 \quad \text{ for $\mathbb{P}$-almost every $\omega\in \Omega$.}
 \]

\end{prop}
\begin{proof}
Consider  $\Omega _0 =\{ \omega \in \Omega :  \mu_i(B_\omega(\mu _i, \psi)) =1, \  \text{for $i=1,\dots,r$}\}$ for a fixed  bounded real-valued function $\psi$.
From the ergodicity of $\mu_i$, by item (ii) in Lemma~\ref{lem-mu1} applied to $B_\omega(\mu_i,\psi)$ according to Remark~\ref{rem:physical},  we get $\mathbb{P}(\Omega_0)=1$.
Set
\[
A_\omega =X\setminus \left(B_\omega (\mu _1, \psi)\cup \dots \cup B_\omega(\mu_r, \psi) \right).
\]
Then, for any $\omega \in \Omega_0$ it holds that $\mu_i(A_\omega) = 0$ for all $i=1,\dots,r$ because $A_\omega \subset X \setminus B_\omega (\mu_i, \psi)$.
Consequently, since $d\mu_i= h_i \, dm$
  \begin{equation}\label{eq:3b}
 \int _{A_\omega} h_i \, dm
  =0 \quad \text{for all
   $i=1,\dots,r$ and $\omega\in \Omega_0$}.
  \end{equation}
Set
\[
\Omega _1 = \bigcap _{n=0}^\infty \sigma ^{-n} (\Omega _0),
\]
which is also a full $\mathbb P$-measure set.
Notice that $\omega \in \Omega _1$ implies $\sigma ^n\omega \in \Omega _0$ for any $n\geq 0$. Thus, from~\eqref{eq:3b}
\begin{equation}\label{eq:4b}
 \int _{A_{\sigma^n\omega}} h_i  \, dm =0 \quad \text{ for each $i=1,\ldots ,r$ and $n\geq 1$ and $\omega\in \Omega_1$.}
\end{equation}

For simplicity of notation, we write $\bar{m}=\mathbb P\times m$.

\begin{claim}\label{claim3}
For $\bar{m}$-almost every~$(\omega , x)\in \Omega \times X$, there is $k\geq 1$ such that $f_\omega^k(x)\in X\setminus A_{\sigma ^k\omega }$.
\end{claim}
\begin{proof}
By contradiction, assume that the claim does not hold. Therefore,
\[
\bar{m}(B_0)>0, \quad B_0= \{ (\omega , x)\in \Omega \times X : \, f_\omega^n(x)\in A_{\sigma ^n\omega } \ \ \text{for all $n\geq 1$}\}.
\]
Set $B= B_0 \cap (\Omega _1\times X)$, which is still a positive $\bar{m}$-measure set because $\Omega _1 \times X$ is a full measure set.
Denoting by $B_\omega = \{ x \in X:  (\omega , x)\in B\}$ the $\omega$-section of $B$, Fubini's theorem implies that for every $n\geq 1$,
\[
\int_\Omega \int_{B_\omega }  1_{A_{\sigma ^n\omega }}\circ f^{n}_\omega  \ dm \, d\mathbb P =
\int_\Omega \int_{B_\omega }  \, d{m}\, d\mathbb P =
\bar{m}(B)
>0.
\]
Hence, with the notation from Appendix~\ref{sec:apendix}, see Lemma~\ref{lem:A2}~(2),
\begin{equation}\label{eq:contradition}
  \frac{1}{n}\sum_{i=0}^{n-1} \int_\Omega \int_{B_\omega }  \mathcal{L}_{\omega}^{i*}1_{A_{\sigma^i\omega}}  \, dm \, d\mathbb P =\bar{m}(B)>0.
\end{equation}
Given $\varphi\in L^1(m)$ and $\omega\in \Omega$, let us write
\[
\mathcal{L}^n_\omega\varphi = \sum_{j=1}^r \lambda_j(\varphi)h_j + \mathcal{Q}^n_\omega \varphi  \quad \text{and} \quad
A_n\varphi = \sum_{j=1}^r \lambda_j(\varphi)h_j + \mathcal{Q}_n \varphi
\]
where we recall that (see Lemma~\ref{lem:A2} (4)),
\[
 A_n\varphi = \frac{1}{n}\sum_{i=0}^{n-1} {\mathcal{L}}^i_f\varphi = \frac{1}{n}\sum_{i=0}^{n-1} \int   \mathcal{L}_{\omega}^{i}\varphi \, d\mathbb P.
\]
Then,
\begin{align*}
  \frac{1}{n}\sum_{i=0}^{n-1} \int_\Omega \int _{X}  &\mathcal Q_\omega ^i 1_X\,  \, {dm}\,d\mathbb P  =
  \int  \big( A_n  1_X\ - \sum_{j=1}^r \lambda_j(1_X) h_j \big) \, {dm} =
   \int  \mathcal Q_n1_X \, {dm} \leq \| \mathcal Q_n1_X\|.
\end{align*}
On the other hand,  by Lemma~\ref{lem:A2}~(3), Equation~\eqref{eq:4b} and the above inequality,
\begin{align*}
\frac{1}{n}\sum_{i=0}^{n-1} &\int_\Omega \int_{B_\omega }  \mathcal{L}_{\omega}^{i*}1_{A_{\sigma^i\omega}}   \, dm \, d\mathbb P \\
&\leq \frac{1}{n}\sum_{i=0}^{n-1} \int_\Omega\int_X \mathcal{L}_{\omega}^{i*}1_{A_{\sigma^i\omega}}    \, {dm}\,d\mathbb P
= \frac{1}{n}\sum_{i=0}^{n-1} \int_\Omega \int _{A_{\sigma ^i\omega }}  \mathcal L_\omega ^i 1_X \,  \,{dm}\,d\mathbb P  \\
&= \frac{1}{n}\sum_{i=0}^{n-1} \int_\Omega \int _{A_{\sigma ^i\omega }}  \mathcal Q_\omega ^i 1_X\,  \, {dm}\,d\mathbb P   \leq
\frac{1}{n}\sum_{i=0}^{n-1} \int_\Omega \int _{X}  \mathcal Q_\omega ^i 1_X\,  \, {dm}\,d\mathbb P
\leq \Vert \mathcal Q_n1_X \Vert  \to 0.
\end{align*}
The last limit (as $n \to \infty$) follows from the assumption~\eqref{eq:AD}. But this limit provides a contradiction with~\eqref{eq:contradition}.
\end{proof}
Fix $(\omega , x)\in \Omega \times M$ (up to zero measure sets) and let $k\geq 1$ be the integer of Claim~\ref{claim3}.
Set $y= f^k_\omega (x) \in X\setminus A_{\sigma ^k\omega }$.
By definition of $A_{\sigma^k\omega}$,
 there is $i$ such that $y\in B_{\sigma ^k\omega }(\mu _i, \psi)$.
Then, in view of~\eqref{eq:Bpsi}, $(\sigma^k\omega,y)\in B(\mu_i, \psi)$ and
 hence
\[
 \lim_{n\to\infty}\frac{1}{n}\sum_{j=0}^{n-1}
\psi( f^j_{\sigma^k\omega} (y)) = \int \psi \, d\mu _i .
\]
On the other hand,
\[
   \frac{1}{n}\sum_{j=0}^{n-1}\psi(f^j_\omega(x))=
   \frac{n-k}{n} \left( \frac{1}{n-k}\sum_{j=0}^{n-k-1}
   \psi(f_{\sigma^k\omega}^j(y)) \right) +
   \frac{1}{n}\sum_{j=0}^{k-1}\psi(f^j_{\omega}(x)) \to \int \psi\, d\mu_i
\]
as $n\to \infty$. This show that $x \in B_\omega(\mu_i, \psi)$
concluding the proof.
\end{proof}

Finally, we summarize and complete the last details of the proof of Theorem~\ref{thm:C}:

\begin{proof}[Proof of Theorem~\ref{thm:C}] If $f$ satisfies {\tt (FPM)}, then $\mathcal{L}_f$ is {\tt (FED)} by Proposition~\ref{prop:FPM-FED} and hence Theorem~\ref{thm:D} implies that  $\mathcal{L}_f$ is {\tt (MC)}. This proves (ii) implies (i) in Theorem~\ref{thm:C}. On the other hand, if $\mathcal{L}_f$ is {\tt(MC)}, then it is also {\tt (APM)} by Theorem~\ref{thm:D} and hence,  one concludes that $f$ satisfies {\tt(FPM)} from  Proposition~\ref{prop:APM-FPM}. This shows (i) implies (ii). Finally, (1), (2) and (3) follow from Lemma~\ref{lem:equi-basin} and Remark~\ref{rem:fiberwise}.
\end{proof}

\section{Sub-hierarchy in {\tt (UC)}} \label{s:UC}

In this section, we complete the proof of the implications in the hierarchies of Figure~\ref{fig:subhierarchy}, as well as provide practical sufficient conditions for {\tt (FPM)} other than conditions in Section~\ref{ss:1.1}.
{In the sequel, $(X,\mathscr{B},m)$ denotes a Polish probability space,  $P:L^1(m) \to L^1(m)$  is a Markov operator and $P(x,A)$ is a transition probability that induces $P$.}
We first give  equivalent conditions for {\tt (UC)}.

\begin{prop} \label{prop:UC}
The following assertions are equivalent:
\begin{enumerate}
\item
$P$ satisfies {\tt (UC)};
\item
$P:L^1(m)\to L^1(m)$ is a quasi-compact operator;
\item
for any $\varepsilon>0$,  there are  $\delta >0$ and  $n_0 \geq 1$ such that
\[
\sup_{\varphi\in D(m)} \int_A P^{n_0}\varphi \, dm < \varepsilon \quad \text{for all $A\in \mathscr{B}$ with $m(A)< \delta$;}
\]
\item
there are $n_0 \geq 1$, $0<\varepsilon<1$ and $\delta >0$ such that
\[
\sup_{\varphi\in D(m)} \int_A P^{n_0}\varphi \, dm < \varepsilon \quad \text{for all $A\in \mathscr{B}$ with $m(A)< \delta$;}
\]
\item
for any $\varepsilon>0$,  there are  $\delta >0$ and  $n_0 \geq 1$ such that  $P^{n_0}(x,A)<\varepsilon$  for all $A\in \mathscr{B}$ with $m(A)< \delta$ and   $m$-almost every  $x\in X$ (depending on $A$);
\item
there are  $n_0 \geq 1$, $0<\varepsilon<1$ and $\delta >0$ such that $P^{n_0}(x,A)<\varepsilon$ for all $A\in\mathscr{B}$ with $m(A)<\delta$ and $m$-almost every $x\in X$ (depending on A);
\item
for any $\varepsilon>0$,  there are  $\delta >0$, $n_0 \geq 1$  and a probability $\mu$ absolutely continuous with respect to $m$ such that $P^{n_0}(x,A)<\varepsilon$ for all $A\in\mathscr{B}$ with $\mu(A)<\delta$ and $m$-almost every $x\in X$ (depending on $A$);
 \item
 there are $n_0\geq 1$, $0<\varepsilon<1$, $\delta>0$  and a probability $\mu$ absolutely continuous with respect to $m$ such that $P^{n_0}(x,A)<\varepsilon$ for all $A\in\mathscr{B}$ with $\mu(A)<\delta$ and $m$-almost every $x\in X$ (depending on $A$).
\end{enumerate}
\end{prop}

\begin{proof}
The equivalence between (1)--(4) follows from~\cite[Theorem~2]{Bart95}. On the other hand, observe that
\begin{equation} \label{eq:duality}
\int_A P^n\varphi \, dm = \int \varphi(x) \, P^{n*}1_A(x) \, dm = \int \varphi(x) P^n(x,A) \, dm
\end{equation}
for all $\varphi \in L^1(m)$, $A\in \mathscr{B}$ and $n\geq 1$. Taking into account that if $\varphi\in D(m)$ then $\int \varphi\, dm =1$, from~\eqref{eq:duality} one immediately gets that~(5) implies~(3) and (6) implies  (4).

We are going to see now that (3) implies~(5). By contradiction, assume that there exists $\varepsilon>0$ satisfying that for any $\delta>0$ and $n\geq 1$ there are $A\in\mathscr{B}$ with $m(A)<\delta$ and set $E\subset X$ with $m(E)>0$ such that $P^n(x,A)\geq \varepsilon$ for all $x\in E$. Take $\varphi=\frac{1}{m(E)} \, 1_E \in D(m)$. Hence, (3) implies that $\int_A P^n\varphi \, dm <\varepsilon$. However, by~\eqref{eq:duality}
\[
\int_A P^n\varphi \, dm = \int \varphi(x) P^n(x,A) \,dm = \frac{1}{m(E)} \int_E P^n(x,A) \, dm  \geq \varepsilon
\]
we get a contradiction.

The implication of (6) from~(4)  follows analogously as the implication of (5) from~(3).
This concludes the equivalence from (1) to (6). To complete the proof, we will show that (5) is equivalent to (7) and (6) is equivalent to (8).

Observe that clearly, (5) implies (7) and (6) implies (8) by taking $\mu=m$. As before, the converse implications are analogous and thus we only prove (7) implies (5). To do this, by (7) we have that $\mu$ is absolutely continuous with respect to $m$. In particular, according to~\cite[Lemma~1]{HK1964}, $\mu$ is uniform absolutely continuous with respect to $m$. That is, for any $\eta>0$ there is $\alpha>0$ such that $\mu(A)<\eta$ for all $A\in\mathscr{B}$ with $m(A)<\alpha$. Then taking $\eta=\delta$ in (7), we get that for any $\varepsilon>0$,  there are  $\alpha >0$ and  $n \geq 1$ such that  $P^n(x,A)<\varepsilon$  for all $A\in \mathscr{B}$ with $m(A)< \alpha$ and   $m$-almost every  $x\in X$ (depending on $A$). This concludes (5).
\end{proof}

As already mentioned in Section~\ref{sec:(UC)}, it follows from Proposition~\ref{prop:UC} (specifically, part (8)) that {\tt (D)} implies {\tt (UC)}.
The following proposition emphasizes a sufficient condition to achieve {\tt (D)}, as demonstrated in the proof of Theorem~\ref{thm:B}.
\begin{prop}
  \label{prop:Felle+UC=D}
Let $P(x,A)$ be a transition probability. If there is $n_0 \in \mathbb{N}$ such that $P^{n_0}(x,A)$ is  strongly Feller continuous and $m$-nonsingular, then $P(x,A)$ satisfies~{\tt(D)}.
\end{prop}
\begin{proof} This follows from Claim~\ref{claim:compact}.
\end{proof}

\begin{rem} \label{rem:AraujoD}
If a random map $f$ satisfies the Ara\'ujo or Brin--Kifer conditions, the transition probability associated with $f$ satisfies {\tt(D)}. This follows from the fact that, under these conditions, the assumption in Theorem~\ref{thm:A} are equivalent to those in Proposition~\ref{prop:Felle+UC=D}, as established by Proposition~\ref{prop:ultra-Feller+B}. Moreover, as noted in Remarks~\ref{rem:gAraujo} and~\ref{rem:gBrin--Kifer}, the assumption in Theorem~\ref{thm:A} are satisfied by Ara\'ujo's or Brin--Kifer's conditions.
\end{rem}

Another important consequence of Proposition~\ref{prop:UC} is the following practical sufficient conditions for {\tt (FPM)}.
\begin{thm}  \label{AAraujo2000-generalization}
Consider a Markov operator $P:L^1(m) \to L^1(m)$ and let $P(x,A)$ be an associated transition probability.
Assume that $P^{n_0}(x,dy)=p(x,y)\,dm(y)$ and one of the following   holds:
  \begin{enumerate}[leftmargin=1cm]
    \item there are $\gamma>0$ and $\alpha>0$ such that
    \begin{enumerate}
    \item $m(\supp p(x,\cdot))>\gamma$ for $m$-almost every $x\in X$, and
      \item $\alpha \leq p(x,y)$ for $m$-almost every $x\in X$ and $m$-almost every $y\in \supp p(x,\cdot)$;
    \end{enumerate}
        \item there is $\beta>0$ such that  $p(x,y) \leq \beta$ for $m$-almost every $x\in X$ and $m$-almost every $y\in \supp p(x,\cdot)$.
  \end{enumerate}
Then, $P$ satisfies {\tt(UC)}. In particular, if $P=\mathcal{L}_f$, then  $f$ satisfies {\tt(FPM)}.
\end{thm}
\begin{proof}
By Proposition~\ref{prop:UC} (6), the condition (2)  implies that $P$ satisfies {\tt (UC)}. Therefore, it follows from the implications in Figure~\ref{fig:hierarchy} and Theorem~\ref{thm:C} that if $P=\mathcal{L}_f$, then $f$ satisfies {\tt (FPM)}.

Assume the condition~(1). Then,
for any $A\in \mathscr B$ with $m(A)>1-\frac{\gamma}{2}$, it holds
\begin{align*}
P^{n_0}(x,A)&=\int_A p(x,y)\, dm(y) = \int_{A\cap \mathrm{supp} \, p(x,\cdot)} p(x,y) \, dm(y)\\
& \geq  \alpha m(A\cap \mathrm{supp}\, p(x,\cdot))>\alpha \, \frac{\gamma}{2}.
\end{align*}
In the last inequality, we used that
\begin{align*} m(A\cap \mathrm{supp}\, p(x,\cdot)) &=m(A) + m( \mathrm{supp}\, p(x,\cdot)) - m(A\cup \mathrm{supp}\, p(x,\cdot)) \\ &> (1-\frac{\gamma}{2}) + \gamma -1 = \frac{\gamma}{2}.
\end{align*}
Hence, by considering   complements and using Proposition~\ref{prop:UC} (6)  again,
we get {\tt (UC)} for $P$   and
 $f$ satisfies {\tt (FPM)} provided $P=\mathcal L_f$.
\end{proof}

\begin{rem}\label{rmk:6.7}
As previously mentioned, Theorem~\ref{AAraujo2000-generalization} (2)  is another weakening of Brin--Kifer's condition (see (2) in page~\pageref{Brin--Kifer}).
Moreover, we find Ara\'ujo--Ayta\c{c}~\cite{AA2017} for  Theorem~\ref{AAraujo2000-generalization} (1).
Indeed, they considered the case when $X$ is a compact manifold equipped with the normalized Lebesgue measure $m$, and  assumed that there exist $\alpha , \beta , \gamma >0$ and $t_*\in T$ such that for every $x\in X$,
\begin{itemize}
\item[(i)] $\supp p(x,\cdot )$ includes the ball of radius $\gamma$ centered at $f_{t_*}(x)$, and
\item[(ii)] $\alpha \leq p (x,y)\leq \beta$ for $m$-almost every $y\in \supp p(x,\cdot ) $.
\end{itemize}
Obviously, the conditions in Theorem~\ref{AAraujo2000-generalization} relaxed Ara\'ujo--Ayta\c{c}'s condition to get that {\tt (UC)}.\footnote{They also assumed an aperiodicity condition to ensure uniform ergodicity, meaning that $X$ cannot be decomposed into $\ell$ subsets $X=X_1\cup \cdots \cup X_\ell$ ($\ell \geq 2$) such that $p(x,  X_{i +1\pmod{\ell}})$ for all $x\in X_i$ and $1\leq i\leq \ell$.
We do not assume it, as indicated in the second example in Section~\ref{subsec:additive}, which satisfies the conditions in Theorem~\ref{AAraujo2000-generalization} but  violates the aperiodicity condition.} Moreover, in view of the following section, under the assumptions in~\cite[Thm.~A]{AA2017}, one gets {\tt (D*)}.
\end{rem}

\subsection{{\tt (D*)} and  uniform ergodicity}\label{ss:6.1}

Let us also give equivalent conditions for   {\tt (D*)}.
\begin{prop} \label{prop:equi-uni-erg}
The following assertions are equivalent:
\begin{enumerate}
\item $P(x,A)$ satisfies {\tt (D*)};
\item there is a probability measure $\pi$ absolutely continuous  with respect to $m$ such that
\[
\lim_{n\to \infty} \sup_{x\in X} \left\lVert P^n(x,\cdot) - \pi\right\rVert_{TV}=0;
\]
\item there are a probability measure $\pi$ absolutely continuous  with respect to $m$ and constants $C>0$, $0<\lambda<1$ such that
\[
\left\lVert P^n(x,\cdot)-\pi\right\rVert_{TV}\leq C \lambda^n \quad  \text{for all $x\in X$}.
\]
\end{enumerate}
\end{prop}
\begin{proof}
The equivalence between (2) and (3) follows from~\cite[Theorem~16.0.2]{MT2012}. It is also clear that (2) implies (1) by taking the measure $\mu$ in (1) as the measure $\pi$ in~(2).
Finally, the implication of  (2) from (1) follows basically from~\cite[Theorem~2]{Dorea2006}.
Namely, Dorea and Pereira proved that (1) implies the convergence  in (2) but they do not conclude that $\pi$ is absolutely continuous with respect to $m$.
To prove this, observe that the convergence in (2) implies that $\pi$ is the unique measure  such that
\[
\pi(A)=\int P(x,A) \, d\pi(x) \quad \text{for all $A\in\mathscr{B}$}.
\]
See {next} Remark~\ref{rem:uni-erg} for more details.

On the other hand, (1) implies condition {\tt(D)} and by Proposition~\ref{prop:UC} the induced  Markov operators in $L^1(m)$ is {\tt(UC)}. Although it is not difficult to get directly {\tt(S)} from condition {\tt(D)} using Theorem~\ref{GST} and the relation $\int {P^n}(x,A)\, dm = \int P^{n*}1_A(x) \, dm=\int_A P^n1_X \, dm$, we can immediately obtain such implication from Figure~\ref{fig:hierarchy}. Thus, there is $g\in D(m)$ such that $Pg=g$.
Hence,
\begin{align} \label{eq:uniqueness}
m_g(A) &\eqdef \int_A g \, dm = \int_A Pg \, dm = \int g \, P^*1_A \, dm = \int P(x,A) \, dm_g   \quad \text{for all $A\in\mathscr{B}$}.
\end{align}
Therefore, by the uniqueness, $m_g=\pi$, proving that $\pi$ is absolutely continuous.
\end{proof}

\begin{rem} \label{rem:uni-erg}
In Markov processes theory, condition (2)  in Proposition~\ref{prop:equi-uni-erg} without the absolute continuity of $\pi$ with respect to $m$ is called
uniform ergodicity (cf.~\cite{MT2012}).
Apart from the background,
it would still make sense to call any of the equivalent conditions in Proposition~\ref{prop:equi-uni-erg} uniform ergodicity (for Markov operators on $L^1(m)$) because the conditions imply that
 there is a unique invariant density of the Markov operator $P:L^1(m)\to L^1(m)$ induced by $P(x,A)$.
 To see this, observe first that (2) in Proposition~\ref{prop:equi-uni-erg} implies that $\pi$ is the unique invariant probability  in the sense that
\[
\pi(A)=\int P(x,A) \, d\pi(x) \quad \text{for all $A\in\mathscr{B}$}.
\]
Indeed, if $\nu$ is a probability satisfying also the above relation, then
\[
\nu(A)=\int P(x,A)\,d\nu(x) = \int P(x,A) P(y,dx)\,d\nu(y) = \dots = \int P^n(x,A)\, d\nu.
\]
Since  $P^n(x,A)$ converges uniformly to $\pi (A)$ on $x$ as $n\to \infty$,  the right-hand side of  the above expression converges to $\pi(A)$,  and therefore, $\nu=\pi$.

Now, since $\pi$ is absolutely continuous with respect to $m$, let us write $d\pi= h \, dm$ with $h\in D(m)$. Then,
\begin{align*}
\int_A h(x) \,dm &= \pi (A) =\int P(x,A)\, d\pi = \int h(x) P(x,A) \, dm   \\
&= \int h(x) P^*1_A(x) \, dm  = \int_A Ph(x)  \, dm.
\end{align*}
This implies that $Ph=h$. Moreover, by the uniqueness shown above,  arguing as in~\eqref{eq:uniqueness},
we conclude that $h$ is the unique invariant density of $P$.
\end{rem}

\section{Examples and counterexamples: proof of propositions} \label{s:examples}

In this section, we prove all propositions given in Section~\ref{ss:e}.

\subsection{Proof of Propositions~\ref{prop:ucd} and~\ref{prop:dd}: Additive type noise}
We first prove Proposition~\ref{prop:ucd}.
Recall that our random map is given by~\eqref{eq:dd}.

To see {\tt(UC)},
 notice that if $x\neq 0$ then
 \[
 \{ t \in T : f_t(x) \in A\} = \{ t \in T : f_0(x) +t \in A \pmod{1}\}
 \]
 is the translated set of $A$ by $f_0(x)$.
Hence, by the invariance of the Lebesgue measure for translations,
\begin{equation}\label{eq:0617}
P(x, A)= p\left(\{t \in T: f_t(x) \in A\}\right) =p(A) =m(A).
\end{equation}
Now, {\tt(UC)}
 immediately follows from the interpretation of {\tt(UC)}
in Section~\ref{sec:(UC)} (e.g.~for $n_0=1$, $\delta =\varepsilon =\frac{1}{2}$ and $\mu =m$).

On the other hand, since
$\{ t\in T: f_t(0) =0\} =T$,
we have $P(0,\{0\}) = 1$, namely $P(0, \cdot )= \delta _0$. Thus it follows from~\eqref{eq:recurence} that $P^n(0,\{0\}) = 1$ for any $n\geq 1$.
This means that {\tt(D)} is violated with $x=0$ and $A=\{ 0\}$.

\begin{figure}
\centering
\begin{subfigure}{0.47\textwidth}
 \includegraphics[width=\textwidth]{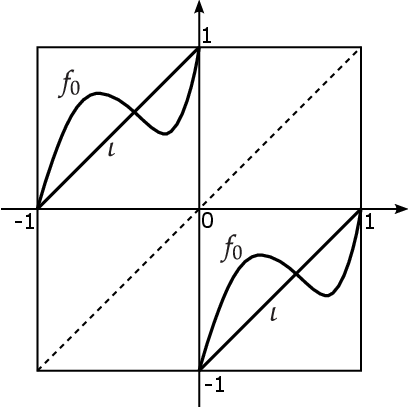}
 \caption{The maps $\iota$ and $f_0$ on $[-1,1]$}
 \label{fig:first}
\end{subfigure}
\hfill
\begin{subfigure}{0.47\textwidth}
 \includegraphics[width=\textwidth]{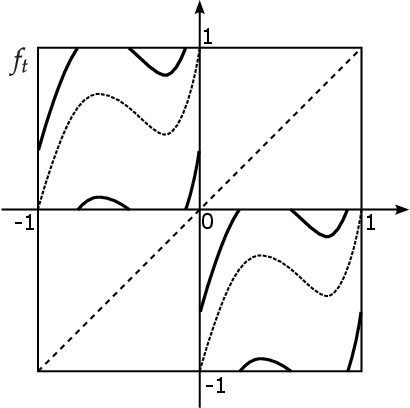}
 \caption{The map $f_t$ on $[-1,1]$}
 \label{fig:second}
\end{subfigure}
\caption{Illustrations of the random map $f$ in Proposition~\ref{prop:ucd}.}
\label{fig_eg_X}
\end{figure}

We next prove Proposition~\ref{prop:dd}. Recall the random map displayed in Figure~\ref{fig_eg_X}.
 Notice that
 $p(A) =2m(A) $
 for any Borel set $A\subset X_+$.
Furthermore, by the argument obtaining~\eqref{eq:0617},
\[
p\left(\left\{t\in T : \widetilde f_t(y) \in B\right\}\right) = 2m(B) \quad \text{for any $y\in X_+$ and Borel set $B\subset X_+$}
\]
and thus
\[
P(x,A) =
\begin{cases}
p\left(\left\{t \in T: \widetilde f_t(x) \in \iota (A)\right\}\right) = 2m(\iota (A)) =2m(A) & \text{for $x\in X_+$}\\
p\left(\left\{t \in T: \widetilde f_t(\iota (x)) \in A\right\}\right) = 2m(A)& \text{for $x\in X_-$}
\end{cases}
\]
for each $A\in \mathscr B$.
By~\eqref{eq:recurence}, $P^n(x, A) =2m(A)$ for each $n\geq 1$, $x\in X$ and $A\in \mathscr B$.
Therefore, obviously {\tt (D)} holds (e.g.~for $n_0=1$, $\delta =\varepsilon =\frac{1}{2}$ and $\mu =m$).
On the other hand, given $n_0\geq 1$, $0< \varepsilon <1$, $\delta > \frac{1}{2}$ and a probability measure $\mu$,
 it holds that $\mu (X_+) \leq \frac{1}{2}$ or $\mu (X_-) \leq \frac{1}{2}$.
For simplicity assume that $\mu (X_+) \leq \frac{1}{2}$, and set $A=X_+$.
Then, $\mu (A)<\delta $ but $P^{n_0}(x,A) =2m(A) =1> \varepsilon$.
This means that {\tt (D*)} is violated.

\subsection{Proof of Propositions~\ref{prop:multi:ex1} { and~\ref{prop:0707}}:
Multiplicative noise}
Recall condition~(b) in Remark~\ref{rem:gAraujo}. {Also recall that condition almost-(b) is the natural relaxation of the condition~(b) from ``for all $x\in X$'' to ``for $m$-almost every $x\in X$''.} Let $f$ be the random map
 given in~\eqref{eq:0621a}, i.e.~random perturbation of a measurable map $f_0: [0,1]\to [0,1]$ by multiplicative noise.
To prove Proposition~\ref{prop:multi:ex1},
we need the following lemma.

\begin{lem}\label{lem:multiple}
Denote the set of fixed points and zeros of $f_{0}$ by $F$ and $Z$, respectively, that is, $F=\left\{x : f_{0}(x)=x\right\}$ and $Z=\left\{x : f_{0}(x)=0\right\}$.
Then, the following holds:
\begin{enumerate}
\item If $0 \in F \cap Z$, then $f$ does not satisfy (b).
\item If $Z=\emptyset$, then $f$ satisfies (b).
\item If $m(Z)=0$, then $f$ satisfies almost-(b).
\end{enumerate}
\end{lem}
\begin{proof}
Assume that $0 \in F \cap Z$.
Then, $\left\{f_{\omega}^{n}(0) : \omega \in \Omega\right\}=\{0\}$ for all $n \geq 1$, implying that $P_{n}(0,\cdot )=\delta _0$ for all $n \geq 1$.
Since $\delta_{0}$ is not absolutely continuous with respect to the Lebesgue measure $m$ of $[0,1]$, (b) does not hold.

Take a point $x$ such that $x\not\in Z$. Then $f_{0}(x)>0$, so $\left\{f_{\omega}(x) : \omega \in \Omega\right\}$ is a closed interval which is not a point set, and $P(x, \cdot)$ is the normalized Lebesgue measure of $[0,1]$, implying that $P(x, \cdot)$ is absolutely continuous for all $x \not\in Z$. Hence, $Z=\emptyset$ implies (b), and $m(Z)=0$ implies almost-(b).
\end{proof}

\begin{proof}[Proof of Proposition~\ref{prop:multi:ex1}]
Let $f_0$ be the measurable map in (1) or (2) of Proposition~\ref{prop:multi:ex1}, that is,
$f_0 (x) = \frac{x}{2}$ for $x\in (0,1]$ and $ f_0(0) =0$ or $\frac{1}{2}$.
To see $\mathcal L_f$ does not satisfy {\tt (S)} for this $f$,
recall that for a given Markov operator $P:L^1(m)\to L^1(m)$, according to Theorem~\ref{GST}, the existence of a $P$-invariant density is equivalent to the following condition:
there is $\delta>0$ such that
\[
\sup_{n\geq 1} \int_A P^n1_X \, dm < 1 \quad \text{for any} \ \ A\in \mathscr{B} \ \ \text{with} \ \ m(A)<\delta.
\]
By duality, this implies that
\[
\int P^{n*}1_A \, dm <1 \quad \text{for all $n\geq 1$ and $A\in \mathscr{B}$ with $m(A)<\delta$.}
\]
Since $P^{n*}1_A(x) =P^n(x, A)\leq 1$ for any $x\in X$, we conclude the following necessary condition for the existence of a $P$-invariant density:
\begin{enumerate}[rightmargin=1.5cm, label=($\star$)]
\item \label{*} there is $\delta>0$ such that for all $n\geq 1$ and $A\in\mathscr{B}$ with $m(A) <\delta $, there exists $E\in \mathscr{B}$ with $m(E)>0$ satisfying that $P^{n}(x,A) < 1$ for all $x\in E$.
\end{enumerate}
We will prove that
 $\mathcal L_f$
 does not satisfy ~($\star$).
 Notice that, for that purpose, it suffices to show that for any $\delta>0$ there are $n_0\geq 1$ and $A_0\in\mathscr{B}$ such that $m(A_0) <\delta $ but $P^{n_0}(x, A_0) =1$ for $m$-almost every $x$.
Observe also that since $0<
 f_t(x) \leq f_0(x)$ for all $t\in[0,1]$ and $x\neq 0$, we have that, if $x\neq 0$ then
\[
0<f^n_\omega(x)\leq f_0^n(x) =\frac{1}{2^n} \quad \text{ for any $n\geq 1$ and
$\omega\in \Omega$.}
\]
 Hence, the support of $P^n(x,\cdot) =\mathbb P(\{ \omega : f^n_\omega(x) \in \, \cdot \, \})$ is contained in the interval $(0,\frac{1}{2^n}]$ of length $\frac{1}{2^n}$ if $x\neq 0$.
 Therefore, for any $\delta>0$, by taking {$n_0\geq -\frac{\log \delta }{\log 2}$} and $A_0=(0,\frac{1}{2^{n_0}}]$,
 we have that $m(A_0)<\delta $ and $P^{n_0}(x, A_0) =1$ whenever $x\neq 0$.
From these observations, ($\star$) does not hold, and consequently, neither does {\tt (S)} for $\mathcal L_f$.
In particular, due to the implications in Figure~\ref{fig:hierarchy}, {\tt (FPM)} does not hold for $\mathcal L_f$.

 Let us complete the proof of items (1) and (2) of Proposition~\ref{prop:multi:ex1}.
 When $f_0$ is the continuous map in (1), obviously (a) in Remark~\ref{rem:gAraujo} holds and $Z=\{0\}$.
 So, by Lemma~\ref{lem:multiple}, $f$ does not satisfy (b) but satisfies almost-(b).
 When $f_0$ is the map in (2) having discontinuity at $0$, (a) does not hold and $Z=\emptyset$.
Therefore, by using Lemma~\ref{lem:multiple} again, we conclude that $f$ satisfies~(b).

We next prove item (3) of Proposition~\ref{prop:multi:ex1}.
Let $f_0(x) =2x \pmod{1}$ for $x\in [0,1]$.
Then, since $0\in F\cap Z$, it follows from Lemma~\ref{lem:multiple} that $f$ does not satisfy (b).
On the other hand, by~\cite[Theorem 3.1]{Iwata2013}, we find that $\mathcal{L}_f$ satisfies {\tt (C)}.
Therefore, $f$ satisfies {\tt (FPM)} by Theorem~\ref{thm:C} together with the implications in Figure~\ref{fig:hierarchy}.
\end{proof}

{

Now, we will prove Proposition~\ref{prop:0707}. Hence, let $f$ be the random map
given in~\eqref{eq:0707c}, i.e.~random perturbation of a measurable map $f_0: [0,1]\to [0,1]$ by multiplicative type noise.
First, we need the following lemma.

\begin{lem}\label{lem:multiple2}
Denote the set of fixed points of $f_{0}$ by $F$.
The following hold:
\begin{enumerate}
\item $F= \emptyset$ if and only if then $f$ satisfies (b);
\item If $m(F)=0$, then $f$ satisfies almost-(b).
\end{enumerate}
\end{lem}
\begin{proof}
Take a point $x\not\in F$. Then $f_{0}(x) -x\neq 0$ and thus $I_x=\left\{f_{t}(x) :\, t \in T\right\}$ is a closed interval which is not a point set.
Hence, $P(x, \cdot)$ is the normalized Lebesgue measure on the closed interval $I_x$. This implies that $P(x, \cdot)$ is absolutely continuous for all $x \not\in F$.
On the other hand, if $x\in F$,
then $\left\{f_{\omega}^{n}(x) :\, \omega \in \Omega\right\}=\{x\}$ for all $n \geq 1$. Consequently $P^{n}(x,\cdot )=\delta _x$ for all $n \geq 1$.
Since $\delta_{x}$ is not absolutely continuous with respect to $m$, (b) does not hold.
The claim immediately follows from these observations.
\end{proof}

\begin{proof}[Proof of Proposition~\ref{prop:0707}]
Now let $f_0$ be the $C^1$ map given by the gradient flow with potential~\eqref{eq:0707b}.
As mentioned, $f_0$ has infinitely many sinks $(s_k)_{k\geq 1}$ and sources $(r_k)_{k\geq 1}$ such that the union of the basins of $s_k$ for $f_0$ coincides with $X\setminus \{r_1, r_2,\ldots \}$.
In particular, $F=\{r_k,s_k: \, k\geq 1\}$ is the set of fixed point of $f_0$.
Hence, it follows from Lemma~\ref{lem:multiple2} that $f$ satisfies almost-(b), but not (b).

Furthermore,
take $x$ in the basin of $s_k$ for $f_0$.
In view of~\eqref{eq:0707c}, $f_t(x)$ is a convex combination of $x$ and $f_0(x)$. Thus, $f_t(x)$ is in the closed interval between these points. In particular (from the orientation preserving) in the closed interval whose endpoints are $x$ and $s_k$. Arguing recursively, $f_\omega^n(x)$ is in the closed interval whose endpoints are $s_k$ and $f_0^{n-1}(x)$ for any $\omega \in \Omega$ and $n\geq 1$.
Since $f_0^n(x) \to s_k$, this implies that $f_\omega ^n(x)\to s_k$ as $n\to\infty$ for any $\omega \in \Omega$.
This completes the proof of Proposition~\ref{prop:0707}.
\end{proof}}

\subsection{Proof of Proposition~\ref{prop:figa}: Random expanding maps}

The constrictivity of $\mathcal L_f$ under the condition~\eqref{expanding_condition} is just the consequence of the work by Boyarsky and Levesque~\cite[Theorem 2]{BR1988}.
Furthermore, it is not difficult to see that $\mathcal L_f$ is not uniformly constrictive as follows. Notice that, for any $n_0\geq 1$ and $x\in X$,
\[
P^{n_0}(x, \cdot ) = \sum _{\omega \in \{ 1,\ldots ,k\}^{n_0}} p_{\omega} \delta _{f_\omega ^{n_0}(x)}
\]
where $p_{\omega} =p^{n_0}(\{\omega \})$ and $f_\omega ^{n_0} =f_{\omega _{{n_0}-1}}\circ \cdots \circ f_{\omega _0}$ for $\omega =(\omega _0,\ldots ,\omega _{{n_0}-1})$.
Thus, if we consider the finite set $A=\{f_\omega ^{n_0}(x) : \omega \in \{ 1,\ldots ,k\}^{n_0}\}$, then $P^{n_0}(x, A )=1$ but $\mu (A) =0$ for any $m$-absolutely continuous probability measure $\mu$.
By virtue of the interpretation of {\tt(UC)} in Section~\ref{sec:(UC)}, this confirms that {\tt (UC)} does not hold.

\subsection{Proof of Proposition~\ref{prop:contracting}: Random contracting maps}

For the proof of Proposition~\ref{prop:contracting},
 the next proposition is important.
Recall the mixing property and exactness of an invariant density for a Markov operator given in Section~\ref{sec:1.4}.
\begin{prop} \label{prop:mixing-no-exact}
For $f$ in Proposition~\ref{prop:contracting}, $1_X$ is an $\mathcal L_f$-invariant density. Moreover, $1_X$ is mixing, but is not exact.
\end{prop}

\begin{proof}
It is obvious that $1_X$ is an invariant density for $\mathcal L_f$.
Furthermore, we immediately find that
\begin{align*}
\mathcal{L}_f^*\varphi(x)=\frac{1}{2}\mathcal{L}_1^*\varphi(x)+\frac{1}{2}\mathcal{L}_2^*\varphi(x)=\frac{1}{2}\varphi\left(\frac{x}{2}\right)+\frac{1}{2}\varphi\left(\frac{x+1}{2}\right)
\end{align*}
which is equivalent to the Perron--Frobenius operator $\mathcal L_g$ of the (deterministic) dyadic map $g(x)=2x\pmod{1}$. Then, for any $A,B\in\mathscr{B}$, we have
\begin{align*}
\int_X \mathcal{L}_f^n 1_A\cdot 1_B \, dm =
\int_X 1_A\cdot \mathcal L_g^n1_B \, dm \to
m(A)m(B)
\end{align*}
as $n\to\infty$ since $m$ is mixing for $g$. Therefore, (using a simple function approximation) we conclude that $1_X$ is mixing.

 On other other hand, for $\varphi =1_{A}-1_{A^c}$ with $A=[0,\frac{1}{2}]$ and any $n\in\mathbb{N}$, we have
\[
\int_X|\mathcal{L}_f^n \varphi |dm=\sum_{k=0}^{2^n-1}\int_X\left|1_{\left[\frac{2k}{2^{n+1}},\frac{2k+1}{2^{n+1}}\right]}-1_{\left[\frac{2k+1}{2^{n+1}},\frac{2k+2}{2^{n+1}}\right]}\right|\,dm=\int_X1\, \,dm=1.
\]
If $1_X$ will be exact, then $\mathcal{L}_f^n\varphi \to \int \varphi\, dm =0$ in $L^1$-norm. Thus, the above computation implies that $1_X$ is not exact.
\end{proof}

Let us move to the proof of Proposition~\ref{prop:contracting}.
It is well-known that, in the class of constrictive Markov operators preserving $1_X$, the mixing property implies the exactness (cf.~\cite[Remark~5.5.1]{LM}). Thus, the above proposition already concludes that $P$ is not constrictive. As another proof, we can also check it directly as follows.

Recall the Dunford--Pettis interpretation of {\tt (C)} (given in Section~\ref{sss:acc}):
$\mathcal L_f$ is constrictive if and only if
for any $\varepsilon>0$ there exists $\delta>0$ such that for any $h \in D(m)$, there is $n_0\geq 1$ for which
\[
\int_A \mathcal{L}_f^nh \, dm<\varepsilon \quad \text{for any $n\geq n_0$ and $A\in\mathscr{B}$ with $m(A)\leq\delta$}.
\]
The key point is that the set $A$ can depend on $n$, compare it with the definition of~{\tt(AC)}.
Fix $\varepsilon>0$ and $\delta>0$.
Let $A=[0,\delta]$ and $h =\frac{1}{\delta}1_A \in D(m)$.
Then
\[
\mathcal{L}_f^n h =\frac{1}{\delta}\sum_{k=0}^{2^n-1} 1_{\left[\frac{k}{2^n},\frac{k}{2^n}+\frac{\delta}{2^n}\right]}.
\]
Thus, by letting $B_n= \supp \mathcal{L}_f^n h $, we have
\[
\int_{B_n}\mathcal{L}_f^n h \, dm = 1\quad \text{and } \quad m(B_n )=\delta \, \text{ for any $n\geq1$},
\]
which implies that $\mathcal{L}_f$ is not constrictive.

On the other hand, since $\mathcal{L}_f$ is mixing with respect to $m$, we can conclude $\mathcal{L}_f$ is asymptotically constrictive as follows.
For any $\varepsilon>0$, set $\delta=\varepsilon$. Then for any $h\in D(m)$ and $B\in\mathscr B$ with $m(B)<\delta$, by the mixing property, we have
\[
\lim_{n\to\infty}\int_B \mathcal{L}_f^nh \, dm=\int_Xh \, dm\cdot m(B) < \delta =\varepsilon.
\]
This completes the proof of Proposition~\ref{prop:contracting}.

\subsection{Proof of Proposition~\ref{prop:rotations}: Random rotations}

To simplify notations, in this subsection we identify any closed interval in $\mathbb{S}^1$ with its corresponding set in $[0,1)$
(a closed interval or a set $[0,a]\cup [b,1)$ with some $a < b$).

\subsubsection{Case (1): Irrational $\alpha-\beta$}
It is sufficient to show that $\mathcal L_f$ is asymptotically stable, that is, $\|\mathcal L_f^n\varphi-1_X\|\to 0$ as $n\to\infty$ for any $\varphi\in D(m)$, since an asymptotically stable Markov operator is constrictive (see~\cite{LM}).
For $\varphi\in D(m)$, we can calculate $\mathcal L_f^n\varphi$ as
\begin{align*}
\mathcal L_f^n\varphi(x)
&=\frac{1}{2^n}\sum_{k=0}^n\binom{n}{k}\varphi(x-(n-k)\alpha-k\beta)
=\frac{1}{2^n}\sum_{k=0}^n\binom{n}{k}\varphi(x-n\alpha+k(\alpha-\beta))\\
&=\frac{1}{2^n}\sum_{k=0}^n\binom{n}{k}\mathcal L_{\beta-\alpha}^{k}\varphi(x-n\alpha)
\end{align*}
where $\mathcal L_\gamma$ denotes the Perron--Frobenius operator for the irrational rotation with angle $\gamma$. Since $\mathcal L_f1_X=1_X$ and $\alpha-\beta$ is irrational, using the mean ergodic theorem weighted with binomial coefficients (see~\cite[Theorem 4.1 and Corollary 4.3]{dykema2009brown}),
\[
\mathcal L_f^n\mathcal L_{-\alpha}^n\varphi
=\mathcal L_f^n\varphi(x+n\alpha)
=\frac{1}{2^n}\sum_{k=0}^n\binom{n}{k}\mathcal L_{\beta-\alpha}^{k}\varphi(x)\to 1_X \text{ in $L^1(m)$ as $n\to\infty$.}
\]
It is clear that $\mathcal L_f\mathcal L_{-\alpha}=\mathcal L_{-\alpha}\mathcal L_f$, $\mathcal L_{-\alpha}1_X=1_X$ and $\|\mathcal L_{-\alpha}\psi\|=\|\psi\|$ for any $\psi\in L^1(m)$. Therefore, we have
\[
\|\mathcal L_f^n\varphi-1_X\|
=\|\mathcal L_{-\alpha}^n(\mathcal L_f^n\varphi-1_X)\|
=\|\mathcal L_{-\alpha}^n \mathcal L_f^n\varphi-\mathcal L_{-\alpha}^n 1_X\|
=\|\mathcal L_f^n\mathcal L_{-\alpha}^n\varphi-1_X\|
\to 0
\]
as $n\to\infty$, which completes the proof.

\subsubsection{Case (2): Irrational $\alpha$ and $\beta$ with rational $\alpha-\beta$}
Let $\alpha-\beta=\frac{\ell}{N}$ for some integer $\ell$ and $N$. {Fix $0<\delta<1$.}
Set $B_i=[\frac{i}{N},\frac{i}{N}+\frac{\delta}{N}]$ for $i=0,\cdots,N-1$ and {$B=B_0\cup \dots \cup B_{N-1}$. Notice that $m(B)=\delta$.}

We first show that $\mathcal L_f 1_{B}=1_{B+\alpha}$. {Indeed, observe that} for any $i=0,\dots,N-1$ there is a unique $j=0,\dots,N-1$ such that $B_i+\alpha=B_{j}+\beta$, where $A+t=[a+t,b+t]$ for $A=[a,b]$ and $t\in[0,1)$. {This follows
for $j=i-\ell \pmod{N}$, since}
\[
\left(\frac{i}{N}+\alpha\right)-\left(\frac{j}{N}+\beta\right)=
\frac{i}{N}-\frac{j}{N}+\alpha-\beta=\frac{i}{N}-\frac{j}{N}+\frac{\ell}{N}=0\pmod{1},
\]
and such $j$ is unique. Then we have
\[
\mathcal{L}_f 1_B = \frac{1}{2}\sum_{i=0}^{N-1}(1_{B_i+\alpha}+1_{B_i+\beta})=\frac{1}{2}\sum_{i=0}^{N-1}(1_{B_i+\alpha}+1_{B_{i}+\alpha})=\sum_{i=0}^{N-1}1_{B_i+\alpha}= 1_{B+\alpha}.
\]
Moreover, we have that $\mathcal{L}_f^n1_B=1_{B+\alpha n}$ for any $n\geq 1$.

Take $\varphi\coloneqq\frac{1}{\delta}1_{B} \in D(m)$ since $m(B)=\delta$. Since $\alpha\not\in \mathbb{Q}$, we can consider a sequence $\{n_j\}_{j\geq 1}$ such that $\alpha n_j\to 0 \pmod{1}$ as $j\to\infty$. Hence,
\[
\int_B \mathcal{L}_f^{n_j}\varphi \, dm\to 1 \quad\text{as $j\to\infty$}.
\]
Therefore, $\mathcal{L}_f$ does not satisfies {\tt (AC)}.

We next prove that $\mathcal{L}_f$ is mean constrictive. It is clear that the function $1_X$ is invariant for $\mathcal{L}_f$. Then, from Remark~\ref{rem:ergodic-implies-MC}, it is sufficient to show that {$1_X$ is an ergodic density}.
{Let $\mathcal{L}_f^*1_A=1_A$, where $\mathcal{L}_f^*$ is the adjoint operator for $\mathcal{L}_f$. We will prove $m(A)\in \{0,1\}$.
 Since obviously $1_A \in L^2(m)$,
 we have the Fourier series of $1_A$ as follows, $1_A(x)=\sum_{n=-\infty}^\infty a_n e^{2n\pi i x}$.
 Then we can calculate
\[
\mathcal{L}_f^*\varphi(x)
=\frac{1}{2}\sum_{n=-\infty}^\infty a_n e^{2n\pi i (x+\alpha)}+\frac{1}{2}\sum_{n=-\infty}^\infty a_n e^{2n\pi i (x+\beta)}.
\]
Then $\mathcal{L}_f^*\varphi=\varphi$ and
\[
\mathcal{L}_f^*\varphi(x)=\frac{1}{2}\mathcal{L}_1^*\varphi(x)+\frac{1}{2}\mathcal{L}_2^*\varphi(x)=\frac{1}{2}\varphi(x+\alpha)+\frac{1}{2}\varphi(x+\beta)
\]
imply
\[
\frac{1}{2}\sum_{n=-\infty}^\infty a_n e^{2n\pi i (x+\alpha)}+\frac{1}{2}\sum_{n=-\infty}^\infty a_n e^{2n\pi i (x+\beta)}=\sum_{n=-\infty}^\infty a_n e^{2n\pi i x}.
\]
Due to the uniqueness of Fourier series, it must be satisfied that, for any $n\in\mathbb{Z}$,
\[
\frac{1}{2} e^{2n\pi i \alpha}+\frac{1}{2} e^{2n\pi i \beta}= 1.
\]
It can be rewritten as
\[
e^{2n\pi i \gamma}+1=2 e^{-2n\pi i \beta}\quad \text{with $\gamma\coloneqq\alpha-\beta$}.
\]
That is,
\[
\cos(2n\pi\gamma)+1=2\cos(-2n\pi\beta)\quad{\rm and}\quad
\sin(2n\pi\gamma)=2\sin(-2n\pi\beta).
\]
Hence
\[
\left(2\cos(2n\pi\beta)-1\right)^2+\left(-2\sin(2n\pi\beta)\right)^2=1,
\]
which leads to $\cos(2n\pi\beta)=1$. This equality holds only for $n=0$ since $\beta$ is irrational. Thus, all coefficient $a_n$ must be 0 for any $n\neq0$, which shows $1_A$ is constant and hence $m(A)=0$ or $1$.

\subsubsection{Case (3): Rational $\alpha$ and $\beta$}
By the same arguments, we have $\mathcal L_f1_B=1_{B+\alpha}$. Unlike the case $\alpha,\beta\in[0,1]\backslash\mathbb{Q}$, we further have $\mathcal{L}_f 1_B = 1_{B}$
by taking $N$ as the least common multiple of the denominator for $\alpha$ and $\alpha-\beta$.
Then, for any $\delta>0$, taking $\varphi=\frac{1}{\delta}1_B$,
we get $m(B)=\delta$ but
\[
\int_B A_n \varphi \, dm=\frac{1}{n}\sum_{j=0}^{n-1}\int_B \mathcal{L}_f^j\varphi \, dm=1,
\]
since the support of $\mathcal{L}_f^n\varphi$ and $A_n\varphi$ is always on $B$.
This concludes that $\mathcal{L}_f$ does not satisfy~{\tt (MC)}.

Finally, since $1_X$ is clearly invariant for $\mathcal{L}_f$, by the fact that {\tt (WAP)} is equivalent to the existence of an invariant density with the maximal support~\cite[Theorem~3.1]{Toyokawa2020}, $\mathcal{L}_f$ satisfies {\tt (WAP)}. This completes the proof.

\subsection{Proof of Proposition~\ref{prop:figb}: Direct sums of random maps}
Notice that the random dynamics $f_+$ on $X_+$ generated by $\tau _1^+$, $\tau _2^+$ satisfies {\tt (S)} because $f$ satisfies {\tt (S)} and preserves $X_+$.
That is, $\mathcal L_{f_+}$ admits an invariant density, denoted by $h_+$.
Let $h$ be the density function on $X$ given by $h(x)=0$ on $X_-$ and $h(x)=h_+(x)$ on $X_+$.
Then, $h$ is obviously an invariant density of $\mathcal L_f$,
namely, {\tt (S)} holds for $f$.

On the other hand, {\tt (WAP)} does not hold. If $\mathcal L_f$ has an invariant density $h$ with the maximal support, then an integral function $h_-$ on $X_-$, given by $ h_-(x) = h(x)$, should be invariant for $\mathcal L_{f_-}$, where $f_-$ is the random dynamics generated by $\tau _1^-$, $\tau _2^-$, due to the invariance of $X_-$, $X_+$ for $f$.
This contradicts the assumption that $f_-$ is not in~{\tt(S)}.

\subsection{Proof of Proposition~\ref{prop:IFS-circle}: Random circle homeomorphisms}
As mentioned earlier, it follows from~\cite[Theorem B and proof of Proposition~4.9]{malicet2017random} that $\mu$ is the unique stationary measure of $f$ and for any probability $\nu$ on $\mathbb{S}^1$, $\nu(B_\omega(\mu))=1$ for $\mathbb{P}$-almost every $\omega\in \Omega$. Therefore, since the reference measure is $m=\frac{1}{2}(\mu+\eta)$,
we have that the random map $f$ satisfies items 1) and 2) in Definition~\ref{dfn:11} and Equation~\eqref{eq:3}. From Proposition~\ref{prop:FPM-FED}, to establish that $f$ does not satisfy {\tt(FPM)}, it suffices to prove that $\mathcal{L}_f$ does not satisfy {\tt (FED)}. To show this, let $h=2\cdot 1_{\supp \mu}$ be the density function of $\mu$ with respect to $m$. Notice that $h$ is the unique ergodic $\mathcal{L}_f$-invariant density. Thus $\mathcal{L}_f$ satisfies~{\tt (FED)} if and only if $h$ has the maximal support. We will prove that actually, $h$ does not have maximal support.

To do this, notice that for any $x$ in a gap of the Cantor set $K$, where the measure $\eta$ is supported, we have that $f^n_\omega(x) \in \mathbb{S}^1\setminus \supp \mu$ for all $n\in\mathbb{N}$ and $\omega\in \Omega$. In particular, we have a set $E$ of $m$-positive measure (the gap where $\eta$ is supported) such that for every $x\in E$,
\[
\mathcal{L}^{n*}_f1_{\supp \mu}(x)= \int 1_{\supp \mu}\circ f^n_\omega(x) \, d\mathbb{P}=0
\]
Hence, since $\supp h = \supp \mu$, we have that $\lim_{n\to\infty} \mathcal{L}^{n*}_f1_ {\supp h} < 1_{\mathbb{S}^1}$ over a set of positive $m$-measure (on $E$) and then $h$ does not have maximal support. Consequently, $\mathcal{L}_f$ does not satisfy {\tt (FED)} and therefore {\tt(FPM)} with respect to $m$ does not hold for $f$. This concludes the proof of Proposition~\ref{prop:IFS-circle}.

\subsection{Proof of Proposition~\ref{prop:dce}: Deterministic systems}

All statements except one for the baker's transformation can be proven by literally repeating the argument in the previous subsections.
On the other hand, for the statement for the baker's transformation, it immediately follows from the fact that the baker's transformation is mixing but not exact with respect to the Lebesgue measure $m$ (cf.~\cite{alexander1984fat}).

\appendix

\section{Annealed Perron--Frobenius operators}
\label{sec:apendix}
Let $(X,\mathscr{B},m)$ be a Polish probability space and consider $(\Omega,\mathscr{F})=(T^\mathbb{N},\mathscr{A}^{\mathbb{N}})$ an infinite product space of a measure space. Now, we introduce a probability measure $\mathbb{P}$ on $(\Omega,\mathscr{F})$ which is shift invariant but not necessarily a Bernoulli probability. Let $f:T\times X \to X$ be a measurable map. We consider iterations $f^n_\omega=f_{\omega_{n}}\circ \dots \circ f_{\omega_1}$ for $n\geq 1$ where the sequence $\omega=(\omega_i)_{i\geq 1}$ is chosen from $\Omega$ according to the probability $\mathbb{P}$.
To emphasize this random choice, we call $f$ as a $\mathbb{P}$-random map. We can introduce the annealed Perron--Frobenius operator ${\mathcal{L}}_f$ given by
\begin{equation}\label{A:eq00}
 {\mathcal{L}}_f \varphi = \int \mathcal{L}_\omega \varphi \, d\mathbb{P} \quad \text{for $\varphi \in L^1(m)$}
\end{equation}
where $\mathcal{L}_\omega$ is the Perron--Frobenius operator of $f_\omega$.
Recall that we introduced $\mathcal{L}_\omega\varphi$ as the Radon--Nikod\'ym derivative of the signed measure $(f_\omega)_* m_\varphi$ with respect to $m$ where
\[
m_\varphi (B) = \int_B \varphi \, dm \quad \text{for all $B\in\mathscr{B}$}.
\]
That is, for each $\varphi \in L^1(m)$,
\begin{equation} \label{A:eq0}
\mathcal{L}_\omega\varphi \in L^1(m) \ \ \text{such that} \ \ (f_{\omega})_*m_\varphi(B) = \int_B \mathcal{L}_\omega\varphi \, dm \ \
\text{for all $B\in\mathscr{B}$}.
\end{equation}
In other words, $(f_{\omega})_*m_\varphi = m_{\mathcal{L}_\omega\varphi}$, i.e.,
\begin{equation}\label{A:eq1}
 \int_{(f_\omega)^{-1}(B)} \varphi \, dm = \int_B \mathcal{L}_\omega\varphi \, dm \quad
\text{for all $B\in\mathscr{B}$}.
\end{equation}
The following lemma characterizes $\mathcal{L}_f$ in similar terms.

\begin{lem} \label{lem:A0}
 For each $\varphi\in L^1(m)$, ${\mathcal{L}}_f\varphi$ is the Radon--Nikod\'{y}m derivative of the measure $\bar{m}_\varphi$ with respect to $m$ where
\[
\bar{m}_\varphi (B) = \int (f_\omega)_* m_\varphi(B) \, d\mathbb{P}
\quad \text{for all $B\in\mathscr{B}$}.
\]
That is,
\[
\mathcal{L}_f\varphi \in L^1(m) \ \ \text{such that} \ \ \bar{m}_\varphi(B) = \int_B {\mathcal{L}}_f\varphi \, dm \ \
\text{for all $B\in\mathscr{B}$}.
\]
In particular,
\begin{equation}\label{eq:estationary}
 {\mathcal{L}}_f\varphi = \varphi \quad \iff \quad m_\varphi(B) = \int (f_\omega)_* m_\varphi (B) \, d\mathbb{P} \quad \text{for all $B\in\mathscr{B}$}.
\end{equation}
\end{lem}
\begin{proof} According to the definition of the annealed operator $\mathcal{L}_f$ in~\eqref{A:eq00} and the Perron--Frobenius operator of $f_\omega$ in~\eqref{A:eq0}, it holds
\[
\int_B \mathcal{L}_f\varphi \, dm = \int_B \int \mathcal{L}_\omega\varphi \, d\mathbb{P} \,dm = \int \int_B \mathcal{L}_\omega\varphi \, dm \,d\mathbb{P} = \int (f_\omega)_*m_\varphi (B)\, d\mathbb{P}.
\]
This proves the first part of the lemma. The second part follows immediately by observing that $\mathcal{L}_f\varphi=\varphi$ if and only if
$\int_B \mathcal{L}_f \varphi \, dm = m_\varphi(B)$. Hence, the above equation reads as ${m}_\varphi =\bar{m}_\varphi$ and thus~\eqref{eq:estationary} follows.
\end{proof}

\begin{rem} Since the Perron--Frobenius operator also acts on the space of finite signed measures as
\[
 \mathcal{L}_f\mu (B)= \int 1_B \circ f_\omega (x) \, d\mathbb{P}(\omega) d\mu(x) \quad \text{for all $B\in\mathscr{B}$,}
\]
the equivalence~\eqref{eq:estationary} can be generalized as follows:
\begin{equation*}\label{eq:estationarymeasure}
 {\mathcal{L}}_f\mu = \mu \quad \iff \quad \mu(B) = \int (f_\omega)_* \mu (B) \, d\mathbb{P} \quad \text{for all $B\in\mathscr{B}$}.
\end{equation*}
\end{rem}
\begin{rem} \label{rem:anexo} Note that when $\mathbb{P}=p^\mathbb{N}$ is a Bernoulli measure on $\Omega$,
\[
\int (f_\omega)_* m_\varphi (B) \, d\mathbb{P} = \int (f_{\omega_1})_*m_\varphi (B) \, dp(\omega_1).
\]
Thus, from~\eqref{eq:estationary}, $\varphi$ is an $\mathcal{L}_f$-invariant density if and only if
$m_\varphi$ is an $f$-stationary measure in the sense introduced in~\eqref{def:stationary} that is absolutely continuous~with~respect~to~$m$.
\end{rem}

Let us write
\begin{equation*}
\mathcal{L}^n_\omega = \mathcal{L}_{\omega_n}\circ \dots \circ \mathcal{L}_{\omega_1}
\quad \text{for $\psi\in L^\infty(m)$ and $\omega\in \Omega$}
\end{equation*}
{where $\sigma$ is the left shift on $\Omega$ preserving $\mathbb{P}$.}
As usual, we introduce the adjoint operator of ${\mathcal{L}}_f:L^1(m)\to L^1(m)$ as the operator ${\mathcal{L}}^{\,*}_f:L^\infty(m)\to L^\infty(m)$ satisfying
\[
\langle {\mathcal{L}}^{\,*}_f\psi, \varphi \rangle = \langle \psi, {\mathcal{L}}_{f}\varphi \rangle \quad
\text{for $\psi\in L^\infty(m)$ and $\varphi \in L^1(m)$ \ \ where \ \ $\langle h,g\rangle = \int hg \, dm$.}
\]
Recall that since $(\mathcal{L}^n_f)^*=(\mathcal{L}^*_f)^n$
where $\mathcal{L}^n_f$ and $(\mathcal{L}^*_f)^n$ are the $n$-th iteration of $\mathcal{L}_f$ and $\mathcal{L}^*_f$ respectively, we simply write it as $\mathcal{L}^{n*}_f$. Similarly, we write $(\mathcal{L}^{n}_\omega)^*=\mathcal{L}^{n*}_\omega$.

\begin{lem} \label{lem:A2}For any $n\geq 1$, $\varphi\in L^1(m)$ and $\psi\in L^\infty(m)$, it holds
\begin{enumerate}[itemsep=0.25cm]
\item \label{lem:A2item1} $\mathcal{L}_\omega^n=\mathcal{L}_{f^n_\omega}$;
\item \label{eq:P*} $\mathcal{L}^{n*}_\omega \psi = \psi \circ f^n_\omega$;
\item \label{lem:A2item4}$\langle {\mathcal{L}}^{n*}_\omega\psi, \varphi \rangle = \langle \psi, {\mathcal{L}}^{\, n}_{\omega}\varphi \rangle$.

 \end{enumerate}
 Moreover, if $\mathbb{P}$ is a Bernoulli measure on $\Omega$ (i.e., an infinite product $p^\mathbb{N}$ of a probability measure ${p}$ on $T$), it holds
 \begin{enumerate}[resume]
 \item \label{lem:A2item2} ${\mathcal{L}}_f^n \varphi = \int \mathcal{L}^n_\omega \varphi \, d\mathbb{P}$;
 \item \label{lem:A2item3} ${\mathcal{L}}_f^{n*} \psi = \int \mathcal{L}_\omega^{n*} \psi \, d\mathbb{P}$.
\end{enumerate}
\end{lem}
\begin{proof}
Let us prove (1). To do this we proceed by induction. It is clear that (1) holds for $n=1$. Assume that $(f^{n-1}_\omega)_* m_\varphi = m_{\mathcal{L}^{n-1}_\omega\varphi}$.
Then
\[
(f^n_\omega)_* m_\varphi=(f_{\sigma^{n-1}\omega})_* ((f^{n-1}_\omega)_* m_\varphi) =(f_{\sigma^{n-1}\omega})_*m_{\mathcal{L}^{n-1}_\omega\varphi}.
\]
From this, having in mind~\eqref{A:eq1}, we get
\begin{align*}
 (f^n_\omega)_* m_\varphi (B) = \int_{(f_{\sigma^{n-1}\omega})^{-1}(B)} \mathcal{L}_{\omega}^{n-1}\varphi \, dm = \int_B \mathcal{L}_{\sigma^{n-1}\omega} \circ \mathcal{L}_{\omega}^{n-1}\varphi \, dm = \int_B \mathcal{L}_{\omega}^{n}\varphi \, dm.
\end{align*}
This implies that $\mathcal{L}_{f^n_\omega}=\mathcal{L}_\omega^n$.

Now, (2) follows immediately from (1) since $\mathcal{L}_\omega^{n*}= (\mathcal{L}_\omega^{n})^* =(\mathcal{L}_{f^n_\omega})^*$. Hence, as is well-known, $(\mathcal{L}_{f^n_\omega})^*\psi=\psi \circ f^n_\omega$. Finally, observe that~(3) is just the duality relation. To conclude the proof, from now on we assume that the probability $\mathbb{P}$ is Bernoulli.

\begin{claim} \label{claimA1} For any $n\geq 1$,
\[
m_{\mathcal{L}^n_f\varphi}(B)=\int(f^n_\omega)_*m_\varphi(B) \, d\mathbb{P} \quad \text{for all $B\in\mathscr{B}$ and $\varphi\in L^1(m)$}.
\]
\end{claim}
\begin{proof}
 We prove the claim by induction. For $n=1$, we need to prove that $m_{\mathcal{L}_f\varphi}(B)=\int(f^n_\omega)_*m_\varphi(B) \, d\mathbb{P}$. Observe that this is just Lemma~\ref{lem:A0}. Thus, we can assume that the claim holds for~$n-1$. Then, using again Lemma~\ref{lem:A0},
 \begin{align*}
 m_{\mathcal{L}^n_f\varphi} (B) &= \int_B \mathcal{L}^n_f\varphi \, dm = \int_B \mathcal{L}_f(\mathcal{L}^{n-1}_f\varphi)\, dm = \int_B (f_\omega)_*m_{\mathcal{L}^{n-1}_f\varphi}(B)\, d\mathbb{P}.
 \end{align*}
 By induction we have that
 \begin{align*}
 \int_B (f_\omega)_*m_{\mathcal{L}^{n-1}_f\varphi}(B)\, d\mathbb{P} &= \iint (f^{n-1}_{\bar{\omega}})_* m_\varphi(f^{-1}_\omega(B))\, d\mathbb{P}(\bar\omega)\,d\mathbb{P}(\omega) \\ &= \iint (f^{}_\omega\circ f^{n-1}_{\bar{\omega}})_* m_\varphi(B)\, d\mathbb{P}(\bar\omega)\,d\mathbb{P}(\omega).
 \end{align*}
Taking into account that $\mathbb{P}$ is a Bernoulli measure and that $f_\omega=f_{\omega_1}$ and $f^{n-1}_{\bar{\omega}}=f_{\bar\omega_{n-1}}\circ \dots \circ f_{\bar\omega_{1}}$ where $\omega=(\omega_i)_{i\geq 1}$ and $\bar{\omega}=(\bar\omega_i)_{i\geq 1}$, we have
\[
\iint (f^{}_\omega\circ f^{n-1}_{\bar{\omega}})_* m_\varphi(B)\, d\mathbb{P}(\bar\omega)\,d\mathbb{P}(\omega) = \int (f^{n}_\omega)_* m_\varphi(B)\, d\mathbb{P}.
\]
Putting together all the above equations, we get the claim.
\end{proof}

Now we conclude (4). To do this, we use the first item in this lemma, the fact that the Perron--Frobenius operator of $f^n_\omega$ satisfies that $m_{\mathcal{L}_{f^{n}_\omega}\varphi}=(f^n_\omega)_*m_\varphi$, and finally Claim~\ref{claimA1}. Then we get that
\begin{align*}
 \int_B \int \mathcal{L}_\omega^n\varphi \, d\mathbb{P}\,dm &= \int \int_B \mathcal{L}_{f^{n}_\omega}\varphi \, dm \, d\mathbb{P}
 = \int (f^n_\omega)_*m_{\varphi}(B)\, d\mathbb{P} = \int_B P^n\varphi \, dm.
\end{align*}
From this, since $B\in \mathscr{B}$ is arbitrary, we obtain that $\mathcal{L}_f^n\varphi=\int \mathcal{L}_\omega^n\varphi \, d\mathbb{P}\,dm$ as desired.

Finally, we will prove (5). By (3), (4) and the duality, we have
\begin{align*}
 \langle \mathcal{L}^{n*}_f\psi,\varphi \rangle &= \langle \psi,\mathcal{L}^{n}_f\varphi \rangle =\langle \psi,\int\mathcal{L}^{n}_\omega\varphi \, d\mathbb{P} \rangle = \int \langle \psi,\mathcal{L}^{n}_\omega\varphi \rangle \, d\mathbb{P} \\
 &= \int \langle \mathcal{L}^{n*}_\omega\psi,\varphi \rangle \, d\mathbb{P} = \langle \int\mathcal{L}^{n*}_\omega\psi \, d\mathbb{P},\varphi \rangle.
\end{align*}
This implies $\mathcal{L}^{n*}_f\psi=\int\mathcal{L}^{n*}_\omega\psi \, d\mathbb{P}$ as required.
\end{proof}

\section{Generalized restrictions of Markov operators} \label{sec:Markov-restriction2}
We will extend the theory of restrictions of Markov operators developed in Section~\ref{sec:Markov-restriction}.
Roughly speaking, we want to replace the reference measure $m$ with an absolutely continuous invariant measure and restrict the Markov operator to the support of such measure recovering the Markov property among other ergodic properties.

Let $P:L^1(m)\to L^1(m)$ be a Markov operator.
Take $h \in D(m)$ and define the probability measure $m_h$ as usual by $dm_h=h\,dm$.
Denote $S=\supp h$ and note that $S\in \mathscr{B}$ with $m(S)>0$.
Let us consider  $L^1(m_h)=L^1(S, \mathscr{B}_{S}, m_{h})$ where $\mathscr{B}_{S}$ denotes the trace $\sigma$-algebra of $S$ in $\mathscr{B}$.
As in Section~\ref{sec:Markov-restriction},  $L^1(m_h) \hookrightarrow L^1(m)$.
Abusing notation, we write this inclusion by $1_{S}$ and identify $L^1(m_S)$ with 
\[
1_S(L^1(m_S))=\{\phi \in L^1(m): \supp \phi \subset S\} \subset L^1(m).
\]
 We define the operator
\begin{align*}
P_{h}:L^1(m_h) \to L^1(m_h), \qquad  P_{h}\phi = \frac{1_{S} P(1_S \phi h)}{h} \quad \text{for $\phi\in L^1(m_h)$}.
\end{align*}
Actually, $P_h$ acts on $L^1(m)$ by the same formula.
Moreover, if $h=\frac{1_S }{ m(S)}$, then  $P_h\phi=1_SP(1_S\phi)$.
That is, $m_h$ and $P_h$ coincides with the measure $m_S$ and the operator $P_S$  introduced in~\eqref{eq:def-mS} and~\eqref{def:P_S} respectively.

It is clear that $P_h$ is a bounded linear positive operator on $L^1(m_h)$.
Moreover, $P_h$ is also a contraction on $L^1(m_h)$ since $\|P_h\|_{\rm op} \leq 1$.
On the other hand,  $L^{\infty}(m_h)\hookrightarrow L^{\infty}(m)$ by the same canonical inclusion $1_S$.
As before, we identify $L^1(m_h)$ with $1_{S}(L^{\infty}(m_h))$.
We also have that $P^*_{h}$, the adjoint operator of $P_{h}$, acting on $L^{\infty}(m_h)$ coincides with $1_{S}P^*1_{S}$.
Indeed,
for each $\varphi\in L^1(m_h)$ and $\psi\in L^{\infty}(m_h)$,
\begin{align*}
\int_S \varphi \cdot P_{h}^*\psi \, dm_h
&=\int_S P_{h}\varphi\cdot \psi\, dm_h
=\int_X 1_S \frac{P(\varphi h)}{h} \cdot \psi h \, dm \\
&=\int_X \varphi h\cdot P^*(1_{S}\psi) \, dm =\int_{S} \varphi \cdot P^*(1_{S}\psi) \, dm_h
\end{align*}
as desired.

\begin{prop} \label{prop:P_h-Markov} Let $P:L^1(m)\to L^1(m)$ be a Markov operator and consider $S\in \mathscr{B}$ with $m(S)>0$.
The following conditions are equivalent:

\begin{enumerate}
\item\label{PS:item1}
$P^*1_{X\setminus S} \leq 1_{X\setminus S}$.
Equivalently, $P^*1_{S} \geq 1_S$;
\item\label{PS:item2}
$\supp P1_S \subset S$.
Equivalently, $\supp P^*1_{X\setminus S} \subset X\setminus S$; 
\item\label{PS:item3}
$1_SP(1_S\phi)=P(1_S\phi)$ for all $\phi\in L^1(m)$.
Equivalently,
\[
P_S\phi = P(1_S\phi) \ \ \text{for all $\phi\in L^1(m)$} \quad \text{or} \quad P(1_S\phi) \in L^1(m_S) \ \ \text{for all $\phi\in L^1(m)$};
\]
\item\label{PS:item4}
for every $\phi \in L^1(m)$ it holds that
\[
    \int_S P(1_S\phi) \, dm = \int_S \phi \, dm;
    \]
  \item\label{PS:item5} $P_S:L^1(m_S)\to L^1(m_S)$ is a Markov operator;
  \item \label{PS:item6} $P_h: L^1(m_h)\to L^1(m_h)$ is a Markov operator for any $h\in D(m)$ with $S=\supp h$.
\end{enumerate}
\end{prop}

\begin{proof}
Let us prove the equivalence between the above conditions.

First, we will prove~\eqref{PS:item1} $\Rightarrow$~\eqref{PS:item2}.
Suppose that~\eqref{PS:item2} does not hold.
Then, there is $B \subset \supp P1_S$ with $m(B\setminus S)>0$.
From this and the condition~\eqref{PS:item1}, it follows
\[
  0<\int_{B\setminus S} P1_{S} \, dm = \int 1_S P^*1_{B\setminus S} \, dm \leq \int 1_S P^*1_{X\setminus S} \, dm \leq \int 1_S 1_{X\setminus S} \, dm =0
\]
which is a contradiction.
Note that the equivalence indicated in~\eqref{PS:item1} immediately follows by using $1_X=P^*1_{S} + P^*1_{X\setminus S}$.
Furthermore, the equivalence indicated in~\eqref{PS:item2} follows immediately by the duality
\[
\int_{X\setminus S} P1_{S}\,dm = \int_{S} P^*1_{X\setminus S}\, dm.
\]

Now, we will show that~\eqref{PS:item2} $\Rightarrow$~\eqref{PS:item3}.
Let $\phi \in L^1(m)$.
Note that $\supp 1_S\phi \subset S = \supp 1_S$.
Thus, by Lemma~\ref{supp} and condition~\eqref{PS:item2}, $\supp P(1_S\phi) \subset \supp P1_S \subset S$ up to $m$-null set.
In particular,  $1_S  P(1_S\phi) = P(1_S\phi)$.
That is, condition~\eqref{PS:item3} holds.

We will see that~\eqref{PS:item3} $\Rightarrow$~\eqref{PS:item4}.
Assuming~\eqref{PS:item3} and using that $P$ is a Markov operator, for any $\phi \in L^1(m)$,
\[
\int_S P(1_S\phi)\, dm = \int 1_S P(1_S\phi) \, dm = \int P(1_S\phi)\, dm =  \int 1_S\phi \, dm = \int_S \phi \, dm
\]
and thus~\eqref{PS:item4} holds.

Next, we will prove that~\eqref{PS:item4} and~\eqref{PS:item5} are equivalent.
Clearly, $P_S$ is a bounded linear positive operator.
Note that, by definition, for any $\phi\in L^1(m_S)$,
\[
   \int P_S\phi \, dm_S = \frac{1}{m(S)} \int_S P(1_S\phi) \, dm   \quad \text{and} \quad \int \phi \, dm_S = \frac{1}{m(S)}\int_S \phi \, dm.
\]
Thus, clearly we get that ~\eqref{PS:item4} implies that $\int P_s\phi \, dm_S=\int \phi \, dm_S$ and viceversa.

We will obtain the implication~\eqref{PS:item4} to~\eqref{PS:item6}.
Similar as above, for any $\varphi\in L^1(m)$ and $h\in D(m)$ with $S=\supp h$,
$P_h$ is clearly a bounded linear positive operator.
Moreover, taking into account that $\varphi h = 1_S \varphi h$,
\begin{equation}\label{eq:Ph-final}
\int P_h\varphi \, dm_h =\int_S P(1_S\varphi h) \, dm    \quad \text{and} \quad \int \varphi \, dm_h = \int_S \varphi h \, dm.
\end{equation}
Hence, applying~\eqref{PS:item4} to $\phi=\varphi h$, we immediately get
\[
\int P_h \varphi \, dm = \int \varphi \, dm_h.
\]
Thus, $P_h$ is Markov and we conclude~\eqref{PS:item6}.

Since~\eqref{PS:item6} $\Rightarrow$~\eqref{PS:item5} is clear, the rest to show is~\eqref{PS:item4} $\Rightarrow$~\eqref{PS:item1}.
Suppose contrarily that there is some $B\subset S$ such that $m(B)>0$ and $1_B P^* 1_S<1_B$.
This implies that
\[
\int_S 1_B\,dm>\int_S 1_BP^*1_S\,dm=\int_S P(1_B1_S)\,dm
\]
but clearly contradicts to the assumption~\eqref{PS:item4} with $\phi=1_B$.
 \end{proof}

\section{Ergodicity of invariant densities} \label{appendix:B}

Let $P:L^1(m) \to L^1(m)$ be a Markov operator as introduced in Section~\ref{s:MO}.  We associate $P$ with a new operator, which we will still denote by $P$, acting
on the set of probability measures $\mu$  on $(X,\mathscr{B})$ which are absolutely continuous with~respect~to~$m$~by
$$
    P\mu(A)=\int P^*1_A(x) \, d\mu  \quad \text{for any $A\in \mathscr{B}$}.
$$
Here $P^*:L^\infty(m)\to L^\infty(m)$ denotes the adjoint operator of $P:L^1(m)\to L^1(m)$. Recall that $D(m)=\{\phi \in L^1(m): \phi \geq 0, \ \|\phi\|=1 \}$ and for a given $\phi \in D(m)$, we denote $m_{\phi}$ the probability measure given by $dm_{\phi}=\phi\, dm$. We say that $\phi \in D(m)$ is a $P$-invariant density if $P\phi =\phi$. Similarly, a probability measure $\mu$ on $(X,\mathscr{B})$ is said to be $P$-invariant if $P\mu=\mu$. Denote by $D_P(m)$ and $I_P(m)$, respectively, the convex sets of $P$-invariant densities and $P$-invariant probability measures that are absolutely continuous with respect to $m$.
\begin{lem} \label{lema:apendix-ergodicty}
 Let $\phi \in D(m)$. Then, $P\phi=\phi$ if and only if ${P}m_{\phi}=m_{\phi}$.
 In particular,  $D_P(m)$ is identified with $I_{P}({m})$.
\end{lem}
\begin{proof}
If $P\phi=\phi$, then for any $A\in\mathscr{B}$,
$$
Pm_{\phi}(A)=\int P^*1_A \, dm_{\phi} =  \int P^*1_A  \phi \, dm = \int_A P\phi \, dm =  \int_A \phi \, dm =m_\phi (A).$$
Conversely, if $Pm_{\phi} =m_{\phi}$, then for any $A\in \mathscr{B}$,
$$
   m_\phi(A) = \int P^*1_A \phi \, dm = \int_A P\phi \, dm.
$$
Thus $P\phi$ is the Radon--Nikod\'{y}m derivative of $m_\phi$ with respect to $m$. Since this derivative is exactly $\phi$, we have $P\phi=\phi$.
\end{proof}

Recall that an invariant probability measure $\nu$ of a transformation $g$ of a measurable space $(Y,\mathscr{A})$, i.e.,  $\nu=g_*\nu$, is said to be \emph{ergodic} if $\nu(A)\in \{0,1\}$ for all set $A\in\mathscr{A}$ such that
{$A \subset g^{-1}(A)$. Equivalently, for all $A\in \mathscr{A}$ such that $A=g^{-1}(A)$.}  Moreover, it is also well known that one can weaken both previous conditions of invariance of the set $A$ by simply asking that the relation holds up to a $\mu$-null set. {That is, if
$\mathcal{L}^*_g1_A \geq 1_A$ or $\mathcal{L}^*_g1_A = 1_A$, respectively, for $\mu$-almost everywhere  where $\mathcal{L}^*_g\psi=\psi \circ g$ for $\psi\in L^\infty(m)$ is the adjoint  Perron--Frobenius operator of $g$. The following proposition shows some of these equivalences in a more general setting.}

\begin{prop}\label{thm:apendix-ergodic}
  Let $\mu\in I_P(m)$. Then the following conditions are equivalent:
  \begin{enumerate}
    \item $\mu(A)\in \{0,1\}$ for any $A\in \mathscr{B}$ such that $P^*1_A \geq 1_A$ for $m$-almost everywhere;
    \item $\mu(A)\in \{0,1\}$ for any $A\in \mathscr{B}$ such that $P^*1_A \geq 1_A$ for $\mu$-almost everywhere;
    \item $\mu(A)\in \{0,1\}$ for any $A\in \mathscr{B}$ such that $P^*1_A = 1_A$ for $\mu$-almost everywhere.
\end{enumerate}
\end{prop}

\begin{proof}
(2) $\Rightarrow$ (3): Obvious.

(3) $\Rightarrow$ (2): Let us consider $A\in \mathscr{B}$ such that $P^*1_A \geq 1_A$ for $\mu$-almost everywhere. Since $\mu \in I_P(m)$, then
$\mu(A)=\int P^*1_A \, d\mu$. On the other hand, $\mu(A)=\int 1_A \,d\mu$. Thus, since $P^*1_A \geq 1_A$ for $\mu$-almost everywhere, we get $0\leq \int P^*1_A -1_A \, d\mu =0$. From this follows that $P^*1_A=1_A$ for $\mu$-almost everywhere. Consequently, it immediately follows that (2) implies (3).

(2) $\Rightarrow$ (1): Let us consider $A\in \mathscr{B}$ such that $P^*1_A \geq 1_A$ for $m$-almost everywhere. This means that there is $B\in \mathscr{B}$ such that $m(B)=0$ and $P^*1_A(x) \geq 1_A(x)$ for all $x\in X\setminus B$. Since $\mu$ is absolutely continuous with respect to $m$, then $\mu(B)=0$. Thus, $P^*1_A \geq 1_A$ for $\mu$-almost everywhere. From this, it immediately follows that (3) implies (2).

(1) $\Rightarrow$ (2): Let us consider $A\in \mathscr{B}$ such that $P^*1_A \geq 1_A$ for $\mu$-almost everywhere. Since $\mu\in I_P(m)$, by Lemma~\ref{lema:apendix-ergodicty}, we have $\phi\in D_P(m)$ such that $\mu=m_\phi$ (i.e., $d\mu=\phi \, dm$). Then, by Proposition~\ref{prop:invariant}, we have that  $P^*1_{X\setminus S} \leq 1_{X\setminus S}$ for $m$-almost everywhere where $S=\supp \phi$. Then,
\begin{equation}\label{eq:1SP}
    1_SP^*1_A=1_SP^*1_{S\cap A} + 1_S P^*1_{(X\setminus S)\cap A}=1_SP^*1_{S\cap A} \leq P^*1_{S\cap A}
\end{equation}
for $m$-almost everywhere. But observe that  $\supp 1_SP^*1_A \subset S$  and $m$ and $\mu$ are equivalent on $S$. Thus, since $P^*1_A \geq 1_A$ for $\mu$-almost everywhere we have $1_{S\cap A}=1_{S}1_A \leq  1_SP^*1_A$ $m$-almost everywhere. Putting this together with~\eqref{eq:1SP}, we get that $1_{S\cap A} \leq P^*1_{S\cap A}$ $\mu$-almost everywhere. By assumption, $m(S\cap A)m(X\setminus (S\cap A))=0$. Consequently, since $\mu(A)=\mu(S\cap A)$ and $\mu(X\setminus A)=\mu(X\setminus (S\cap A))$, one has that $\mu(A)\mu(X\setminus A)=0$.
\end{proof}

Let $h \in D_P(m)$. From, Lemma~\ref{lema:apendix-ergodicty},  we can apply Proposition~\ref{thm:apendix-ergodic} to  $\mu=m_{h}$.

\begin{dfn} \label{def:ergodic}
Let $h \in D_P(m)$ and set $\mu =m_h$. If any of the equivalent items in Proposition~\ref{thm:apendix-ergodic} holds, we say that \emph{$\mu$ is an ergodic $P$-invariant measure}, \emph{$h$ is an ergodic $P$-invariant density}, \emph{$(P,\mu)$ is ergodic} or \emph{$(P,h)$ is ergodic}.\footnote{The reference measure $m$ is always implicit when we mention $P$.}
\end{dfn}

Observe that if $P1_X=1_X$, then according to Lemma~\ref{lema:apendix-ergodicty}, it holds that $Pm=m$, i.e., $m\in I_P(m)$. Then, Proposition~\ref{thm:apendix-ergodic} implies that $1_X$ is an ergodic $P$-invariant density if and only if $m(A)\in \{0,1\}$ for all $A\in\mathscr{B}$ with $P^*1_A=1_A$. Motivated by this last condition and following Krengel~\cite[page~126]{Krengel}, we introduce the following definition:

\begin{dfn} A Markov operator $P:L^1(m)\to L^1(m)$ is called
\emph{ergodic in the sense of Krengel} if the set of $P^*$-invariant sets  $\mathscr{B}_i=\{A\in \mathscr{B}: P^*1_A=1_A\ \text{$m$-almost everywhere}\}$ is trivial, i.e., $\mathscr{B}_i=\{\emptyset, X\}$ up to an $m$-null set.
\end{dfn}

In view of this definition, the above observation can be written as follows:
\begin{rem} \label{rem:apendix}
If $P1_X=1_X$, then the following are equivalent:
\begin{enumerate}
  \item $(P,1_X)$ is ergodic (or in other words, $(P,m)$ is ergodic);
  \item $P$ is ergodic in the sense of Krengel.
\end{enumerate}
In general,  $P$ could be ergodic in the sense of Krengel and $1_X \not\in D_P(m)$. See, for instance, Remark~\ref{rem:reciprocal}. Note that Definition~\ref{def:ergodic} requires that $1_X$ is a $P$-invariant density to be ergodic.
\end{rem}

Now,  recalling the theory of restriction of Markov operators introduced in Section~\ref{sec:Markov-restriction} and generalized in Appendix~\ref{sec:Markov-restriction2} we have the following:

\begin{thm} \label{thm:ergodic2}
Let $h\in D_P(m)$ and $S=\supp h$. Then the following conditions are equivalent:
 \begin{enumerate}
  \item $(P,h)$ is ergodic (or in other words, $(P,m_h)$ is ergodic);
  \item $(P_h,1_S)$ is ergodic (or in other words, $(P_h,m_h)$ is ergodic);
  \item $P_h$ is ergodic in the sense of Krengel;
  \item $P_S$ is ergodic in the sense of Krengel.
  \end{enumerate}
  \end{thm}

\begin{proof} Recall that $P^*_h \psi =1_S P^*(1_S\psi)$ for all $\psi \in L^\infty(m_h)$.

(1) $\Rightarrow$ (3): Let us consider $A\in \mathscr{B}_S$ such that $P^*_h1_A = 1_A$ for $\mu$-almost everywhere. Then $1_S P^*1_A=1_A$ $\mu$-almost everywhere and hence also $m$-almost everywhere.  In particular, $P^*1_A \geq 1_A$ $m$-almost everywhere and thus, by assumption, $m_h(A)\in \{0,1\}$.

(3) $\Leftrightarrow$ (2): Notice that $P_h1_S=1_S$. Thus, the equivalence follows from Remark~\ref{rem:apendix}.

(2) $\Rightarrow$ (1):  Let us consider $A\in \mathscr{B}$ such that $P^*1_A \geq  1_A$ for $\mu$-almost everywhere. Since $h\in D_P(m)$, we have $P^*1_{X\setminus S} \leq 1_{X\setminus S}$ for $m$-almost everywhere. Hence, as argued in the previous theorem (see~\eqref{eq:1SP}),
 $$1_{S\cap A} \leq 1_SP^*1_A=1_SP^*1_{S\cap A} + 1_S P^*1_{(X\setminus S)\cap A}=1_SP^*1_{S\cap A} = P_h^*1_{S\cap A}$$
 $m$-almost everywhere. Then, by assumption we have that $m_h(S\cap A)\in \{0,1\}$. But, since $m_h(A)=m_h(S\cap A)$ we conclude the implication.

 (3) $\Leftrightarrow$ (4): Notice that $P^*_S$ and $P^*_h$ are both of the form $1_SP^*1_S$  acting, respectively, on $L^1(m_S)$ and $L^1(m_h)$. Since both measures $m_S$ and $m_h$ are equivalent, we have that $P^*_S1_A=1_A$ $m_S$-almost everywhere if and only if $P^*_h1_A=1_A$ $m_h$-almost everywhere. From this, it immediately follows the equivalence.
\end{proof}

\begin{dfn}
A Markov operator $P:L^1(m)\to L^1(m)$ is called
\emph{conservative}  if there is $\varphi\in L^1(m)$ such that $\varphi>0$ on $X$ and
$$
X=\left\{x \in X:  \sum_{i=0}^\infty P^i\varphi(x)=\infty\right\} \quad \text{up to an $m$-null set}.
$$
\end{dfn}

\begin{prop} \label{prop:conservative}
Let $h\in D(m)$ be a $P$-invariant density and set $S=\supp h$. Then both, $P_h$ and $P_S$ are conservative Markov operators.
In particular, if $h\in D_P(m)$ is ergodic then,
\begin{enumerate}[leftmargin=0.75cm]
\item $P_S:L^1(m_S) \to L^1(m_S)$ is a conservative and ergodic Markov operator  and  $D_{P_S}(m_S)=\{h\}$;
\item $P_h:L^1(m_h) \to L^1(m_h)$ is a conservative and ergodic Markov operator and \mbox{$D_{P_h}(m_h)=\{1_S\}$.}
\end{enumerate}
\end{prop}
\begin{proof} Since $h$ is $P$-invariant, according to Proposition~\ref{prop:wap} and Proposition~\ref{prop:P_h-Markov}, $P_S$ and $P_h$ are Markov operators and $P_Sh=h$. Also, it is not difficult to check that $P_h1_S=1_S$. On the other hand, $P_S$ and $P_h$ are conservative since $h>0$ and $1_S>0$ and $\sum_{i\geq 0} P_S^i h =\infty$ and $\sum_{i\geq 0} P_h^i1_S=\infty$ on $S$.
If in addition $h$ is ergodic, then from Theorem~\ref{thm:ergodic2} we get that $P_S$ and $P_h$ are conservative and ergodic Markov operators.
Finally, from~\cite[Theorem~A in Chapter~VI]{foguel2007ergodic} it follows that $D_{P_S}(m_S)=\{h\}$ and $D_{P_h}=\{1_S\}$.
\end{proof}

\begin{rem} \label{rem:reciprocal}
In view of the above proposition, if $h$ is an ergodic $P$-invariant density with $h>0$ on $X$ and $1_X \not = h$, then $P$ is ergodic in the sense of Krengel but \mbox{$1_X \not \in D_{P}(m)=\{h\}$}.
\end{rem}

{The next proposition is well-known for the case of transformations. One can generalize this to the case for Markov operators in terms of invariant densities.}

\begin{prop} \label{prop:ergodic}
 Let $h \in D_P(m)$. Then, $h$ is ergodic if and only if  $h$ is an extremal point of $D_P(m)$. That is, it cannot be
decomposed as
\[
\text{$h = th_1 +(1-t)h_2$ with $t  \in  (0, 1)$ and $h_1, h_2 \in D_P(m)$.}
\]
 Moreover, if $h_1 \in D_P(m)$ is ergodic and $h_2\in D_P(m)$, then either
 \[
 h_1=\frac{1_{\supp h_1}}{\int_{\supp h_1}h_2dm}\,h_2 \quad  \text{or} \quad m(\supp h_1 \cap \supp h_2) =0.
 \]
 In particular, if $h_1\in D_P(m)$ and $h_2\in D_P(m)$ are both ergodic, then either $h_1=h_2$ or they have mutually disjoint supports up to an $m$-null set.
\end{prop}
\begin{proof}
Suppose that $h$ is ergodic and set $S=\supp h$.  By Proposition~\ref{prop:conservative}, $P_S$ is a Markov operator and $D_{P_S}(m_S)=\{h\}$. In particular, we have $P_Sg = P(1_Sg)$ for all $g \in L^1(m_S)$. See Proposition~\ref{prop:wap} or Proposition~\ref{prop:P_h-Markov}.
Now,
if $h=ah_1+(1-a)h_2$ for some $0<a<1$ and $h_1,h_2\in D_P(m)$, then $h=h_1=h_2$. Indeed, observe that since $h_i$ is a $P$-invariant density, according again to Proposition~\ref{prop:conservative},  $P_{S_i}$ is a Markov operator and  $P_{S_i}h_i=h_i$ where $S_i=\mathrm{supp}\, h_i$ for $i=1,2$.  Then, since $S_i\subset S$, we have $P_Sh_i=P(1_Sh_i) =P(1_{S_i}h_i)=P_{S_i}h_i=h_i$ for $i=1,2$. Consequently, it follows $h=h_1=h_2$ since $D_{P_S}(m_S)=\{h\}$. This proves that $h$ is an extremal point of $D_P(m)$.

Now, we will prove the converse. Suppose $h$ is not ergodic and denote by $\mu:=m_{h}$.
Take a set $A\in \mathscr{B}$ such that $P^*1_A \geq 1_A$  and $0<\mu(A) < 1$.  Hence, since $\mu(S\setminus A) =1 -\mu(A)$,  we can write  $h$ as a convex combination as follows
\begin{align*}
h=\mu(A)\cdot\frac{1_Ah}{\mu(A)} + \mu({S\setminus A})\cdot\frac{1_{S\setminus A}h}{\mu({S\setminus A})}.
\end{align*}
Moreover, since $ (1_Ah)/\mu(A)$ and  $(1_{S\setminus A}h)/\mu(S\setminus A)$ have both $L^1$-norm equals one, to show that $h$ is  not an extremal point in $D_P(m)$ it suffices to prove that $h 1_A$ and $h 1_{S\setminus A}$ are $P$-invariant. To prove this, observe first that since $h=1_A h+1_{S\setminus A} h$ and $h=Ph=P(1_Ah) + P( 1_{S\setminus A}h)$, one can write
\begin{equation}\label{eq:anpendix-final}
   1_A h - P(1_Ah) = P(1_{S\setminus A}h) -1_{S\setminus A} h.
\end{equation}
From Lemma~\ref{lem:P_S-Markov}, it follows that $\supp P1_A \subseteq A$ up to an $m$-null set. Morevover, since $\supp h = S = \supp 1_{S}$, by the $P$-invariance of $h$ and Lemma~\ref{supp}, it follows that $  S=\supp  h = \supp Ph  = \supp P1_{S}$ up to an $m$-null set. Hence,   $\supp P1_{S\setminus A}\subseteq S\setminus A$. Additionally, since   $\supp (1_Ah) \subseteq \supp 1_A$ and $\supp (1_{S\setminus A}h) \subseteq \supp 1_{S\setminus A}$, Lemma~\ref{supp} implies that $\supp P(1_A h) \subseteq \supp P1_A \subseteq A$ and $\supp P(1_{S\setminus A}h) \subseteq \supp P1_{S\setminus A} \subseteq S\setminus A$.
Consequentely,  $\supp (1_A h - P(1_Ah))  \subseteq A$ and  $\supp (P(1_{S\setminus A}h) - 1_{S\setminus A}h ) \subseteq X\setminus A$ and thus it follows from~\eqref{eq:anpendix-final}  that
\begin{align*}
P(1_Ah)=1_Ah \quad \text{ and } \quad P(1_{S\setminus A}h)=1_{S\setminus A}h
\end{align*}
as desired.

Finally, we will prove the second part of the proposition.  Let us assume $h_1$ and $h_2$ as in the statement. Clearly if $h_1=h_2$ then $m(S)>0$ where $S=\supp h_1 \cap \supp h_2$. Conversely, suppose that $m(S)>0$. Since $P1_S \leq P1_{\supp h_i}$, by Lemma~\ref{supp},  $\supp P1_{S} \subset \supp P1_{\supp h_i} = \supp Ph_i = \supp h_i$ for $i=1,2$ and thus $\supp P1_S \subset S$.
Furthermore, the ergodicity of $h_1$ implies $S=\supp h_1$ up to an $m$-null set.
According to Proposition~\ref{prop:P_h-Markov}, $P_S$ is a Markov operator, $P_S\phi=P(1_S\phi)$ and
\begin{equation} \label{eq:contra}
\int_{S} P(1_{S}\phi) \, dm = \int_{S} 1_{S}\phi \, dm
\end{equation}
for all $\phi\in L^1(m)$.  This implies that $1_Sh_i$ is $P$-invariant. Indeed, since $Ph_i=h_i$, it follows that $1_SP(1_Sh_i)=P_Sh_i=P(1_{S}h_i) \leq  h_i$. Hence $P(1_Sh_i) \leq 1_Sh_i$. On the other hand,
if $P(1_{S}h_i)<1_{S}h_i$ on a set $A\in\mathscr{B}$ of positive $m$-measure, then
$$ \int_{S} P(1_{S}h_i) \, dm < \int_{A} 1_Sh_i \, dm + \int_{S\setminus A} 1_Sh_i \, dm = \int_{S} 1_{S}h_i \, dm$$
contradicting~\eqref{eq:contra}.  Therefore, we have $m(A)=0$ and $P(1_Sh_i)=1_Sh_i$.
Then, $g_i=\frac{1_Sh_i}{H_i} \in D_P(m)$ where $H_i=\int_S h_i \,dm$ (notice that $g_1=h_1$ since $S=\supp h_1$).
Moreover, $m_{g_1}$ is absolutely continuous with respect to $m_{h_1}$ since $dm_{g_1}= \frac{1_S}{H_1}\, dm_{h_1}$. In particular, since $m_{h_1}$ is ergodic, we also have that $m_{g_1}$ is ergodic. Since $S=\supp g_1$, we obtain from Proposition~\ref{prop:conservative} that $D_{P_S}(m_S)=\{g_1\}$. But, since $S=\supp g_2$ and $g_2$ is $P$-invariant, we also have that $g_2 \in D_{P_S}(m_S)$. Thus, $g_2=g_1$, or $h_1=\frac{1_Sh_2}{H_2}$ as desired.
Moreover, from this, if both $h_1$ and $h_2$ are ergodic, then if follows that $h_1=h_2$ which completes the proof.
\end{proof}

\begin{rem} \label{rem:disjoin-support}
Recall that the support of $h \in L^1(m)$ is only well-defined up to a set of measure zero. Thus, without loss of generality, from the above corollary, we can assume that the supports of any two ergodic $P$-invariant densities $h_1$ and $h_2$ are identical or disjoint.
\end{rem}

To conclude, we will relate the previous definition of ergodicity for invariant measures of Markov operators with the classical approach in random dynamical systems and probability theory. First of all, let us consider an annealed Perron--Frobenius operator $\mathcal{L}_f:X\to X$ associated with a random map $f:T\times X\to X$ (with respect to a Bernoulli probability measure $\mathbb{P}$). Recall that, as we showed in Remark~\ref{rem:P=Lf}, when $X$ is a Polish space, any Markov operator can be represented as an annealed Perron--Frobenius operator. As explained in Section~\ref{sec:(UC)}, associated with this map we have a transition probability $P(x,A)$ which coincides, for each $A\in \mathscr{B}$, with $\mathcal{L}_f^*1_A$ on $m$-almost everywhere.
Let us consider an absolutely continuous probability measure $\mu$ on $X$ which is $\mathcal{L}_f$-invariant. A classical approach to introducing the ergodicity of $\mu$ is first to leave this measure to an invariant measure of a deterministic dynamical system that represents $f$ in some sense. Then, ergodicity is defined throughout this leaf measure. A deterministic dynamical system that represents $f$ is its associated skew-product given by
$$ F:\Omega \times X \to \Omega \times X, \quad F(\omega,x)=(\sigma\omega, f_\omega (x))$$
with  $f_\omega = f(t,\cdot)$ if $\omega_0=t$,  $\omega=(\omega_i)_{i\in \mathbb{Z}} \in \Omega=T^{\mathbb{N}}$ and $\sigma:\Omega\to \Omega$ the shift operator.
The measure $\mu$ is lifted to the $F$-invariant measure $\mathbb{P}\times \mu$. But we can also consider the shift operator acting on  $(X^\mathbb{Z},\mathscr{B}^\mathbb{Z})$.  By Kolmogorov extension theorem, one can construct a shift-invariant probability measure $\mathbb{P}_{\mu}$ on $X^\mathbb{Z}$  from the consistent sequence of measures $\{\mathbb{P}_{\mu}^n\}_{n\geq 0}$ where $\mathbb{P}_{\mu}^n$ is defined on $X^{2n+1}$ by
$$
  \int \varphi \, d\mathbb{P}^n_{\mu}= \int \varphi(x_{-n},\dots, x_{n}) \, P(x_{n-1},dx_{n})\dots P(x_{-n},dx_{-n+1})\mu(dx_{-n}).
$$
The following result follows basically from classical results
of Markov chain and random dynamical systems. We refer to~\cite[Section~5.2]{hairer2006ergodic} and~\cite[Proposition~1.3]{liu2006smooth} (see also~\cite{Kifer1986,morita1988deterministic,O83}) for more details.

\begin{thm}
Let $f:T\times X\to X$ be a random  map (with respect to a Bernoulli probability measure $\mathbb{P}$) and denote by $\mathcal{L}_f:L^1(m)\to L^1(m)$ the associated annealed Perron--Frobenius operator. Let $\mu \in I_m(\mathcal{L}_f)$.   Then the following conditions are equivalent:
\begin{enumerate}
\item
 $(\mathcal{L}_f,\mu)$ is ergodic;
    \item
     $\mathbb{P}_\mu$ is an ergodic shift invariant measure on $X^\mathbb{Z}$;
    \item
      $\mathbb{P}\times \mu$ is an ergodic $F$-invariant probability measure on $\Omega\times X$.
\end{enumerate}
\end{thm}

\section{Mixing and exactness of invariant densities} \label{appendix:D}

In this appendix we recall definition of mixing and exactness for  (deterministic)  dynamics, and relate them with mixing and exactness for  Markov operators defined in Section~\ref{sec:1.4}.
Although the content of this appendix should be a folklore among experts, we include it for the convenience of the reader.

Let $(X, \mathscr B, m)$ be a probability space.
Let $g: X\to X$ be an $m$-nonsingular measurable map,
and consider a $g$-invariant probability measure $\mu$ on $X$.
Then, $\mu$ is called {\emph mixing} if $\mu (g^{-n} A \cap B)\to \mu(A)\mu(B)$ as $n\to\infty$ for any $A, B\in\mathscr B$ (cf.~\cite{W2000}).
Recall also that   $\mu$ is said to be \emph{exact}  if
\[
\bigcap_{n\ge0}g^{-n}\mathscr{B}=\{\emptyset,X\}\pmod{\mu}.
\]
From Lin's theorem (\cite{Lin}), $\mu$ is exact if and only if
\[
\lim_{n\to\infty}\lV\mathcal{L}_g^n\varphi\rV_{L^1(\mu)}=0 \quad \text{for any $\varphi\in L^1_0(\mu) \coloneqq\left\{\psi\in L^1(\mu):\int_X\psi \, d\mu =0\right\}$.}
\]
Below as in Appendix \ref{sec:Markov-restriction2}, we identify $\varphi\in L^1(m_S)$ with
$1_S\varphi \in L^1(m)$ for an $m$-positive measure set $S$.
\begin{lem}
Assume that $\mu$   is a $g$-invariant probability measure absolutely continuous with respect to $m$ with the density map $h$.
Let    $S\coloneqq\supp h$.
Then, $\mu$ is mixing if and only if
\[
\lim_{n\to\infty}\mathcal{L}_g^n\varphi=h\int_X\varphi dm \quad \text{weakly in $L^1(m)$  for any $\varphi\in L^1(m_S)$.}
\]
Furthermore, $\mu$ is exact if and only if
\begin{equation}\label{eq:0701ex2}
\lim_{n\to\infty}\mathcal{L}_g^n\varphi=h\int_X\varphi \, dm \quad \text{strongly in $L^1(m)$  for any $\varphi\in L^1(m_S)$.}
\end{equation}
\end{lem}
\begin{proof}
Note first that the mixing property of $\mu$  is equivalent to requiring that
\[
\int \psi \circ g^{n}  \cdot \varphi\,d\mu \to \int \psi \, d\mu \int \varphi \,d\mu \quad \text{as $n\to\infty$}
\]
for any $\psi \in L^\infty(\mu )$ and $\varphi \in L^1(\mu )$.
Due to the duality between $\mathcal L_g$ and $\psi \mapsto \psi \circ g$, this can be written as
\[
\int \psi  \cdot  \mathcal L_g^n( \varphi h)\,dm \to \int \psi \left(h \int (\varphi h )\, dm \right) dm \quad \text{as $n\to\infty$}.
\]
Hence, we immediately get the claim for mixing.

We next show the claim for exactness.
Since exactness is a hereditary property between absolutely continuous measures, it follows from Lin's theorem mentioned above that $\mu$ is exact if and only if
\begin{equation}\label{eq:0701ex}
\lim_{n\to\infty}\lV\mathcal{L}_g^n\varphi\rV_{L^1(m_S)}=0 \text{ for any $\varphi\in L^1_0(m_S)$.}
\end{equation}
Thus, it suffices to show the equivalence between \eqref{eq:0701ex} and \eqref{eq:0701ex2}.

\eqref{eq:0701ex}$\Rightarrow$\eqref{eq:0701ex2}:
Note that $\mathcal{L}_g^n\varphi-h\int_X\varphi dm=\mathcal{L}_g^n(\varphi-h\int_X\varphi dm)$ and
\[
\int_X\left(\varphi-h\int_X\varphi dm\right)dm=\frac{1}{m(S)}\int_X\varphi dm_S-\int_Xhdm\cdot\frac{1}{m(S)}\int_X\varphi dm_S=0.
\]

\eqref{eq:0701ex2}$\Rightarrow$\eqref{eq:0701ex}: It is straightforward if we take $\varphi\in L^1_0(m_S)$.
\end{proof}

We finally remark that, if a map $g:X\to X$  admits an ergodic   invariant probability    measure $\mu$,
then it holds by Birkhoff's pointwise ergodic theorem
\[
\lim _{n\to\infty}\int  \frac{1}{n} \sum _{j=0}^{n-1} \varphi \circ  g^j  \cdot \psi \, d(\mathbb P\times \mu) = \int \psi d\mu \int \varphi d\mu
\]
for any $\varphi \in L^\infty(\mu )$, $\psi \in L^1 (\mu )$.
Thus, in the case when $\mu =h\, dm$ with some   $h\in D(m)$,
\[
\lim _{n\to\infty}   \frac{1}{n} \sum _{j=0}^{n-1} \mathcal L_{ g} ^j\psi = h\int \psi d\mu   \quad \text{weakly in $L^1(m)$ for each $\psi \in L^1(m_S)$.}
\]
On the other hand, since  $h$ is an ergodic invariant density and the restriction of   $\mathcal L_{ g }$   on $L^1(m_S)$ satisfies {\tt (WAP)}, it follows from Yosida--Kakutani's mean ergodic theorem (\cite[Theorem 1]{yosida1941operator}) that
\[
\lim _{n\to\infty}   \frac{1}{n} \sum _{j=0}^{n-1}  \mathcal L_{ g } ^j\psi = h\int \psi d\mu   \quad \text{strongly in $L^1(m)$ for each $\psi \in L^1(m_S)$.}
\]

\section*{Acknowledgement}
We are sincerely grateful to V\'ictor Ara\'ujo, Hiroki Sumi, Mitsuhiro Shishikura and Masayuki Asaoka for their valuable comments.
We are also grateful to the anonymous referees for their valuable comments. 
P.~G.~Barrientos was supported by grant PID2020-113052GB-I00 funded by MCIN,
PQ 305352/2020-2 (CNPq) and JCNE E-26/201.305/2022 (FAPERJ).
F.~Nakamura was supported by JSPS KAKENHI Grant Number 19K21834.
Y.~Nakano was supported by JSPS KAKENHI Grant  Numbers 19K14575, 19K21834, 21K03332 and 22K03342.
H.~Toyokawa was supported by JSPS KAKENHI Grant Numbers 19K21834 and 21K20330.

\bibliographystyle{siam}
\bibliography{bibliography}

\begin{thebibliography}{10}

\bibitem{AS1970}
{\sc R.~Abraham and S.~Smale}, {\em {Nongenericity of $\Omega$-stability}}, in
  Global Analysis, vol.~14 of Proceedings of Symposia in Pure Mathematics,
  American Mathematical Society, 1970, pp.~5--8.

\bibitem{ANV2015}
{\sc R.~Aimino, M.~Nicol, and S.~Vaienti}, {\em Annealed and quenched limit
  theorems for random expanding dynamical systems}, Probability Theory and
  Related Fields, 162 (2015), pp.~233--274.

\bibitem{alexander1984fat}
{\sc J.~C. Alexander and J.~A. Yorke}, {\em Fat baker's transformations},
  Ergodic {T}heory and {D}ynamical {S}ystems, 4 (1984), pp.~1--23.

\bibitem{alves2007stochastic}
{\sc J.~F. Alves, V.~Ara{\'u}jo, and C.~H. V{\'a}squez}, {\em Stochastic
  stability of non-uniformly hyperbolic diffeomorphisms}, Stochastics and
  Dynamics, 7 (2007), pp.~299--333.

\bibitem{Anosov1967}
{\sc D.~V. Anosov}, {\em {Geodesic flows on closed {R}iemannian manifolds of
  negative curvature}}, Trudy Matematicheskogo Instituta Imeni V.~A.~Steklova,
  90 (1967), pp.~3--210.

\bibitem{Araujo2000}
{\sc V.~Ara{\'u}jo}, {\em Attractors and time averages for random maps},
  Annales de l'Institut Henri Poincar{\'e} C, Analyse Non Lin{\'e}aire, 17
  (2000), pp.~307--369.

\bibitem{Araujo2001}
{\sc V.~Ara{\'u}jo}, {\em Infinitely many stochastically stable attractors},
  Nonlinearity, 14 (2001), pp.~583--596.

\bibitem{AA2017}
{\sc V.~Ara{\'u}jo and H.~Ayta{\c{c}}}, {\em Decay of correlations and laws of
  rare events for transitive random maps}, Nonlinearity, 30 (2017),
  pp.~1834--1852.

\bibitem{APP2014}
{\sc V.~Ara{\'u}jo, M.~J. Pacifico, and M.~Pinheiro}, {\em Adapted random
  perturbations for non-uniformly expanding maps}, Stochastics and Dynamics, 14
  (2014).
\newblock Paper No.~1450007, 27 pp.

\bibitem{AT2005}
{\sc V.~Ara{\'u}jo and A.~Tahzibi}, {\em Stochastic stability at the boundary
  of expanding maps}, Nonlinearity, 18 (2005), pp.~939--958.

\bibitem{Arnoldbook}
{\sc L.~Arnold}, {\em Random dynamical systems}, Springer Monographs in
  Mathematics, Springer-Verlag, Berlin, 1998.

\bibitem{Athreya2003}
{\sc K.~B. Athreya}, {\em Stationary measures for some {M}arkov chain models in
  ecology and economics}, Economic Theory, 23 (2004), pp.~107--122.

\bibitem{baladi2000positive}
{\sc V.~Baladi}, {\em Positive transfer operators and decay of correlations},
  vol.~16 of Advanced Series in Nonlinear Dynamics, World Scientific, 2000.

\bibitem{baladi2018dynamical}
\leavevmode\vrule height 2pt depth -1.6pt width 23pt, {\em Dynamical zeta
  functions and dynamical determinants for hyperbolic maps}, vol.~68 of
  Ergebnisse der Mathematik und ihrer Grenzgebiete, Springer, Cham, 2018.

\bibitem{BG2009}
{\sc V.~Baladi and S.~Gou{\"e}zel}, {\em {Good Banach spaces for piecewise
  hyperbolic maps via interpolation}}, Annales de l'Institut Henri Poincar{\'e}
  C, Analyse Non Lin{\'e}aire, 26 (2009), pp.~1453--1481.

\bibitem{BV1996}
{\sc V.~Baladi and M.~Viana}, {\em Strong stochastic stability and rate of
  mixing for unimodal maps}, Annales Scientifiques de l'Ecole Normale
  Sup{\'e}rieure, 29 (1996), pp.~483--517.

\bibitem{Barrientos2021}
{\sc P.~G. Barrientos}, {\em Historic wandering domains near cycles},
  Nonlinearity, 35 (2022), pp.~3191--3208.

\bibitem{BR2021a}
{\sc P.~G. Barrientos and A.~Raibekas}, {\em Robust degenerate unfoldings of
  cycles and tangencies}, Journal of Dynamics and Differential Equations, 33
  (2021), pp.~177--209.

\bibitem{BR2021b}
\leavevmode\vrule height 2pt depth -1.6pt width 23pt, {\em Berger domains and
  kolmogorov typicality of infinitely many invariant circles}, Journal of
  Differential Equations, 408 (2024), pp.~254--278.

\bibitem{barrientos2023typical}
{\sc P.~G. Barrientos and J.~D. Rojas}, {\em Typical coexistence of infinitely
  many strange attractors}, Mathematische Zeitschrift, 303 (2023), p.~34.

\bibitem{Bart95}
{\sc W.~Bartoszek}, {\em On uniformly smoothing stochastic operators},
  Commentationes Mathematicae Universitatis Carolinae, 36 (1995), pp.~203--206.

\bibitem{Bartoszek2008}
\leavevmode\vrule height 2pt depth -1.6pt width 23pt, {\em The work of
  {P}rofessor {A}ndrzej {L}asota on asymptotic stability and recent progress},
  Opuscula Mathematica, 28 (2008), pp.~395--413.

\bibitem{BV2006}
{\sc M.~Benedicks and M.~Viana}, {\em {Random perturbations and statistical
  properties of {H}{\'e}non-like maps}}, Annales de l'Institut Henri
  Poincar{\'e} C, Analyse Non Lin{\'e}aire, 23 (2006), pp.~713--752.

\bibitem{Berger2016}
{\sc P.~Berger}, {\em Generic family with robustly infinitely many sinks},
  Inventiones mathematicae, 205 (2016), pp.~121--172.

\bibitem{Berger2017}
\leavevmode\vrule height 2pt depth -1.6pt width 23pt, {\em Emergence and
  non-typicality of the finiteness of the attractors in many topologies},
  Proceedings of the Steklov Institute of Mathematics, 297 (2017), pp.~1--27.

\bibitem{BB2020}
{\sc P.~Berger and S.~Biebler}, {\em Emergence of wandering stable components},
  Journal of the American Mathematical Society,  (2022).

\bibitem{bogachev2007measure}
{\sc V.~I. Bogachev}, {\em Measure theory}, vol.~1, Springer-Verlag, Berlin,
  2007.

\bibitem{BDV2006}
{\sc C.~Bonatti, L.~J. D\'{\i}az, and M.~Viana}, {\em Dynamics beyond uniform
  hyperbolicity}, vol.~102 of Encyclopaedia of Mathematical Sciences,
  Springer-Verlag, Berlin, 2005.

\bibitem{Bowen1975}
{\sc R.~Bowen}, {\em {Equilibrium states and the ergodic theory of {A}nosov
  diffeomorphisms}}, vol.~470 of Springer Lecture Notes in Math,
  Springer-Verlag, Berlin-New York, 1975.

\bibitem{BR1988}
{\sc A.~Boyarsky and R.~Levesque}, {\em Spectral decomposition for combinations
  of {M}arkov operators}, Journal of Mathematical Analysis and Applications,
  132 (1988), pp.~251--263.

\bibitem{BK1987}
{\sc M.~Brin and Y.~Kifer}, {\em {Dynamics of {M}arkov chains and stable
  manifolds for random diffeomorphisms}}, Ergodic {T}heory and {D}ynamical
  {S}ystems, 7 (1987), pp.~351--374.

\bibitem{Buzzi1999}
{\sc J.~Buzzi}, {\em {Exponential decay of correlations for random
  Lasota--Yorke maps}}, Communications in mathematical physics, 208 (1999),
  pp.~25--54.

\bibitem{Colli1998}
{\sc E.~Colli}, {\em Infinitely many coexisting strange attractors}, Annales de
  l'Institut Henri Poincar{\'e} C, Analyse Non Lin{\'e}aire, 15 (1998),
  pp.~539--579.

\bibitem{deroin2007dynamique}
{\sc B.~Deroin, V.~Kleptsyn, and A.~Navas}, {\em Sur la dynamique
  unidimensionnelle en r{\'e}gularit{\'e} interm{\'e}diaire}, Acta Mathematica,
  199 (2007), pp.~199--262.

\bibitem{Dorea2006}
{\sc C.~C.~Y. Dorea and A.~G.~C. Pereira}, {\em A note on a variation of
  {D}oeblin's condition for uniform ergodicity of {M}arkov chains}, Acta
  Mathematica Hungarica, 110 (2006), pp.~287--292.

\bibitem{douc2018markov}
{\sc R.~Douc, E.~Moulines, P.~Priouret, and P.~Soulier}, {\em Markov chains},
  Springer Series in Operations Research and Financial Engineering, Springer,
  Cham, 2018.

\bibitem{dykema2009brown}
{\sc K.~Dykema and H.~Schultz}, {\em {Brown measure and iterates of the
  {A}luthge transform for some operators arising from measurable actions}},
  Transactions of the American Mathematical Society, 361 (2009),
  pp.~6583--6593.

\bibitem{Emelyanov}
{\sc E.~Y. Emel'yanov}, {\em Non-spectral asymptotic analysis of one-parameter
  operator semigroups}, vol.~173 of Operator Theory: Advances and Applications,
  Birkh\"{a}user Verlag, Basel, 2007.

\bibitem{ethier1986markov}
{\sc S.~N. Ethier and T.~G. Kurtz}, {\em {M}arkov {P}rocesses:
  {C}haracterization and {C}onvergence}, John Wiley \& Sons Inc, 1986.

\bibitem{foguel2007ergodic}
{\sc S.~R. Foguel}, {\em The ergodic theory of {M}arkov processes}, Israel
  Journal of Mathematics, 4 (1966), pp.~11--22.

\bibitem{galatolo2020existence}
{\sc S.~Galatolo, M.~Monge, and I.~Nisoli}, {\em Existence of noise induced
  order, a computer aided proof}, Nonlinearity, 33 (2020), pp.~4237--4276.

\bibitem{GST1993}
{\sc S.~Gonchenko, D.~V. Turaev, and L.~P. Shil'nikov}, {\em On the existence
  of {N}ewhouse domains in a neighborhood of systems with a structurally
  unstable {P}oincar\'{e} homoclinic curve (the higher-dimensional case)},
  Russian Academy of Sciences, Doklady, Mathmatics, 47 (1993), pp.~268--273.

\bibitem{Gouezel2015}
{\sc S.~Gou{\"e}zel}, {\em Limit theorems in dynamical systems using the
  spectral method}, in Hyperbolic dynamics, fluctuations, and large deviations,
  vol.~89 of Proceedings of Symposia in Pure Mathematics, American Mathematical
  Society, 2015, pp.~161--193.

\bibitem{hairer2006ergodic}
{\sc M.~Hairer}, {\em Ergodic properties of {M}arkov processes}.
\newblock \url{https://www.hairer.org/notes/Markov.pdf}, 2006.

\bibitem{hairer2009non}
\leavevmode\vrule height 2pt depth -1.6pt width 23pt, {\em Ergodic properties
  of a class of non-{M}arkovian processes}, in Trends in stochastic analysis,
  vol.~353 of London Mathematical Society Lecture Note Series, Cambridge
  University Press, 2009, pp.~65--98.

\bibitem{HK1964}
{\sc A.~B. Hajian and S.~Kakutani}, {\em Weakly wandering sets and invariant
  measures}, Transactions of the American Mathematical Society, 110 (1964),
  pp.~136--151.

\bibitem{HH2001}
{\sc H.~Hennion and L.~Herv\'{e}}, {\em Limit theorems for {M}arkov chains and
  stochastic properties of dynamical systems by quasi-compactness}, vol.~1766
  of Lecture Notes in Mathematics, Springer-Verlag, Berlin, 2001.

\bibitem{II1991}
{\sc T.~Inoue and H.~Ishitani}, {\em Asymptotic periodicity of densities and
  ergodic properties for nonsingular systems}, Hiroshima mathematical journal,
  21 (1991), pp.~597--620.

\bibitem{islam2005generalization}
{\sc M.~S. Islam, P.~G\'{o}ra, and A.~Boyarsky}, {\em A generalization of
  {S}traube's theorem: existence of absolutely continuous invariant measures
  for random maps}, Journal of Applied Mathematics and Stochastic Analysis,
  (2005), pp.~133--141.

\bibitem{ito1964invariant}
{\sc Y.~Ito}, {\em Invariant measures for {M}arkov processes}, Transactions of
  the American Mathematical Society, 110 (1964), pp.~152--184.

\bibitem{Iwata2013}
{\sc Y.~Iwata}, {\em Multiplicative stochastic perturbations of one-dimensional
  maps}, Stochastics and Dynamics, 13 (2013).
\newblock Paper No.~1250020, 10 pp.

\bibitem{JKR2015}
{\sc J.~Jost, M.~Kell, and C.~S. Rodrigues}, {\em {Representation of {M}arkov
  chains by random maps: existence and regularity conditions}}, Calculus of
  Variations and Partial Differential Equations, 54 (2015), pp.~2637--2655.

\bibitem{Kifer1974}
{\sc J.~I. Kifer}, {\em On small random perturbations of some smooth dynamical
  systems}, Mathematics of the USSR-Izvestiya, 8 (1974), pp.~1083--1108.

\bibitem{Kifer1986}
{\sc Y.~Kifer}, {\em Ergodic theory of random transformations}, vol.~10 of
  Progress in Probability and Statistics, Birkh\"auser Boston, 1986.

\bibitem{KNS2021}
{\sc S.~Kiriki, Y.~Nakano, and T.~Soma}, {\em Historic and physical wandering
  domains for wild blender-horseshoes}, Nonlinearity, 36 (2023), p.~4007.

\bibitem{KS2017}
{\sc S.~Kiriki and T.~Soma}, {\em Takens' last problem and existence of
  non-trivial wandering domains}, Advances in Mathematics, 306 (2017),
  pp.~524--588.

\bibitem{Klenke2008}
{\sc A.~Klenke}, {\em Probability theory}, Universitext, Springer-Verlag
  London, 2008.

\bibitem{Komornik1989}
{\sc J.~Komorn{\'\i}k}, {\em Asymptotic decomposition of smoothing positive
  operators}, Acta Universitatis Carolinae. Mathematica et Physica, 30 (1989),
  pp.~77--81.

\bibitem{komornik1993asymptotic}
{\sc J.~Komorn{\'\i}k}, {\em Asymptotic periodicity of {M}arkov and related
  operators}, in Dynamics reported, Springer, 1993, pp.~31--68.

\bibitem{KL1987}
{\sc J.~Komorn{\'\i}k and A.~Lasota}, {\em Asymptotic decomposition of {M}arkov
  operators}, Bulletin of the Polish Academy of Sciences, Mathematics, 35
  (1987), pp.~321--327.

\bibitem{komornik1991asymptotic}
{\sc J.~Komorn{\'\i}k and E.~G. Thomas}, {\em Asymptotic periodicity of
  {M}arkov operators on signed measures}, Mathematica Bohemica, 116 (1991),
  pp.~174--180.

\bibitem{Krengel}
{\sc U.~Krengel}, {\em Ergodic theorems}, vol.~6 of De Gruyter Studies in
  Mathematics, Walter de Gruyter, 1985.

\bibitem{kushner2001heavy}
{\sc H.~J. Kushner}, {\em Heavy traffic analysis of controlled queueing and
  communication networks}, vol.~47 of Applications of Mathematics,
  Springer-Verlag, New York, 2001.

\bibitem{kusolitsch2010theorem}
{\sc N.~Kusolitsch}, {\em Why the theorem of {S}cheff{\'e} should be rather
  called a theorem of {R}iesz}, Periodica Mathematica Hungarica, 61 (2010),
  pp.~225--229.

\bibitem{LLY1984}
{\sc A.~Lasota, T.-Y. Li, and J.~A. Yorke}, {\em Asymptotic periodicity of the
  iterates of {M}arkov operators}, Transactions of the American Mathematical
  Society, 286 (1984), pp.~751--764.

\bibitem{LM}
{\sc A.~Lasota and M.~C. Mackey}, {\em Chaos, fractals, and noise}, vol.~97 of
  Applied Mathematical Sciences, Springer-Verlag, New York, second~ed., 1994.

\bibitem{Leal2008}
{\sc B.~Leal}, {\em High dimension diffeomorphisms exhibiting infinitely many
  strange attractors}, Annales de l'Institut Henri Poincar{\'e} C, Analyse Non
  Lin{\'e}aire, 25 (2008), pp.~587--607.

\bibitem{Lin}
{\sc M.~Lin}, {\em Mixing for {M}arkov operators}, Zeitschrift f{\"u}r
  Wahrscheinlichkeitstheorie und Verwandte Gebiete, 19 (1971), pp.~231--242.

\bibitem{liu2006smooth}
{\sc P.-D. Liu and M.~Qian}, {\em Smooth ergodic theory of random dynamical
  systems}, vol.~1606 of Lecture Notes in Mathematics, Springer-Verlag, Berlin,
  1995.

\bibitem{Liverani1995}
{\sc C.~Liverani}, {\em Decay of correlations}, Annals of Mathematics, 142
  (1995), pp.~239--301.

\bibitem{malicet2017random}
{\sc D.~Malicet}, {\em Random walks on {H}omeo({$S^1$})}, Communications in
  Mathematical Physics, 356 (2017), pp.~1083--1116.

\bibitem{matias_2021}
{\sc E.~Matias}, {\em Markovian random iterations of homeomorphisms of the
  circle}, Ergodic Theory and Dynamical Systems, 42 (2022), pp.~2935--2956.

\bibitem{MT2012}
{\sc S.~Meyn and R.~L. Tweedie}, {\em Markov chains and stochastic stability},
  Cambridge University Press, second~ed., 2009.

\bibitem{morita1988deterministic}
{\sc T.~Morita}, {\em Deterministic version lemmas in ergodic theory of random
  dynamical systems}, Hiroshima mathematical journal, 18 (1988), pp.~15--29.

\bibitem{neveu1965mathematical}
{\sc J.~Neveu}, {\em Mathematical foundations of the calculus of probability},
  Holden-Day, 1965.

\bibitem{neveu1967existence}
\leavevmode\vrule height 2pt depth -1.6pt width 23pt, {\em Existence of bounded
  invariant measures in ergodic theory}, in {Proceedings of the Fifth Berkeley
  Symposium on Mathematical Statistics and Probability}, vol.~2, University
  California Press, 1967, pp.~461--472.

\bibitem{Newhouse1970}
{\sc S.~E. Newhouse}, {\em {Nondensity of {A}xiom {A(a)} on $S^2$}}, in Global
  Analysis, vol.~14 of Proceedings of Symposia in Pure Mathematics, American
  Mathematical Society, 1970, pp.~191--203.

\bibitem{Newhouse1979}
{\sc S.~E. Newhouse}, {\em The abundance of wild hyperbolic sets and non-smooth
  stable sets for diffeomorphisms}, Publications Math{\'e}matiques de
  l'IH{\'E}S, 50 (1979), pp.~101--151.

\bibitem{O83}
{\sc T.~Ohno}, {\em Asymptotic behaviors of dynamical systems with random
  parameters}, Publications of the Research Institute for Mathematical
  Sciences, 19 (1983), pp.~83--98.

\bibitem{Palis2000}
{\sc J.~Palis}, {\em A global view of dynamics and a conjecture on the
  denseness of finitude of attractors}, Ast\'{e}risque,  (2000), pp.~335--347.

\bibitem{Palis2005}
\leavevmode\vrule height 2pt depth -1.6pt width 23pt, {\em A global perspective
  for non-conservative dynamics}, Annales de l'Institut Henri Poincar{\'e} C,
  Analyse Non Lin{\'e}aire, 22 (2005), pp.~485--507.

\bibitem{PV1994}
{\sc J.~Palis and M.~Viana}, {\em High dimension diffeomorphisms displaying
  infinitely many periodic attractors}, Annals of mathematics,  (1994),
  pp.~207--250.

\bibitem{peres1996absolute}
{\sc Y.~Peres and B.~Solomyak}, {\em Absolute continuity of {B}ernoulli
  convolutions, a simple proof}, Mathematical Research Letters, 3 (1996),
  pp.~231--239.

\bibitem{rudin1991functional}
{\sc W.~Rudin}, {\em Functional analysis}, International Series in Pure and
  Applied Mathematics, McGraw-Hill, second~ed., 1991.

\bibitem{Ruelle1976}
{\sc D.~Ruelle}, {\em {A measure associated with {A}xiom {A} attractors}},
  American Journal of Mathematics,  (1976), pp.~619--654.

\bibitem{Ruelle2001}
\leavevmode\vrule height 2pt depth -1.6pt width 23pt, {\em Historical behaviour
  in smooth dynamical systems}, in Global analysis of dynamical systems, CRC
  Press, 2001, pp.~63--66.

\bibitem{Sinai1972}
{\sc Y.~G. Sinai}, {\em Gibbs measures in ergodic theory}, Russian Mathematical
  Surveys, 27 (1972), p.~21.

\bibitem{Smale1967}
{\sc S.~Smale}, {\em Differentiable dynamical systems}, Bulletin of the
  American mathematical Society, 73 (1967), pp.~747--817.

\bibitem{socala1988existence}
{\sc J.~Soca\l~a}, {\em On the existence of invariant densities for {M}arkov
  operators}, Annales Polonici Mathematici, 48 (1988), pp.~51--56.

\bibitem{straube1981existence}
{\sc E.~Straube}, {\em On the existence of invariant, absolutely continuous
  measures}, Communications in Mathematical Physics, 81 (1981), pp.~27--30.

\bibitem{Sumi2021}
{\sc H.~Sumi}, {\em {Negativity of Lyapunov exponents and convergence of
  generic random polynomial dynamical systems and random relaxed Newton's
  methods}}, Communications in Mathematical Physics, 384 (2021),
  pp.~1513--1583.

\bibitem{Takens2008}
{\sc F.~Takens}, {\em Orbits with historic behaviour, or non-existence of
  averages}, Nonlinearity, 21 (2008), pp.~T33--T36.

\bibitem{Toyokawa2020}
{\sc H.~Toyokawa}, {\em {$\sigma$}-finite invariant densities for eventually
  conservative {M}arkov operators}, Discrete and Continuous Dynamical Systems.
  Series A, 40 (2020), pp.~2641--2669.

\bibitem{W2000}
{\sc P.~Walters}, {\em An introduction to ergodic theory}, vol.~79 of Graduate
  Texts in Mathematics, Springer-Verlag, New York-Berlin, 1982.

\bibitem{worm2011ergodic}
{\sc D.~T. Worm and S.~C. Hille}, {\em Ergodic decompositions associated with
  regular {M}arkov operators on {P}olish spaces}, Ergodic Theory and Dynamical
  Systems, 31 (2011), pp.~571--597.

\bibitem{wu2000uniformly}
{\sc L.~Wu}, {\em Uniformly integrable operators and large deviations for
  {M}arkov processes}, Journal of Functional Analysis, 172 (2000),
  pp.~301--376.

\bibitem{yosida1941operator}
{\sc K.~Yosida and S.~Kakutani}, {\em {Operator-theoretical treatment of
  {M}arkoff's process and mean ergodic theorem}}, Annals of Mathematics,
  (1941), pp.~188--228.

\end{thebibliography}

\end{document}